\documentclass[12pt]{amsart}
\usepackage{amsfonts,amsmath,amsthm,amssymb,mathrsfs}
\usepackage{MnSymbol}
\usepackage{tikz, tikz-cd}\usetikzlibrary{arrows}
\usetikzlibrary{decorations.markings}
\usepackage[square]{natbib}
\setcitestyle{numbers}

\usepackage[all]{xy}
\usepackage[vcentermath]{youngtab}
\usepackage{ytableau}
\usepackage{hyperref}

\usepackage[margin=1.1in]{geometry}

\definecolor{darkred}{rgb}{1,0,0}

\newtheorem{theorem}{Theorem}[section]
\newtheorem{thm}[theorem]{Theorem}

\newtheorem{lem}[theorem]{Lemma}

\newtheorem{prop}[theorem]{Proposition}

\newtheorem{cor}[theorem]{Corollary}
\newtheorem{conj}[theorem]{Conjecture}

\theoremstyle{remark}

\theoremstyle{definition}

\newtheorem{lemdef}[theorem]{Definition/ Lemma}
\newtheorem{rmk}[theorem]{Remark}
\newtheorem{example}[theorem]{Example}
\newtheorem{defn}[theorem]{Definition}

\newtheorem*{claim*}{Claim}
\newtheorem*{defn*}{Definition}

\newcommand{\bbc}{\mathbb{C}}

\newcommand{\bbn}{\mathbb{N}}

\newcommand{\bbp}{\mathbb{P}}
\newcommand{\bbq}{\mathbb{Q}}
\newcommand{\bbr}{\mathbb{R}}

\newcommand{\bbz}{\mathbb{Z}}

\newcommand{\mca}{\mathcal{A}}
\newcommand{\mcb}{\mathcal{B}}
\newcommand{\mcc}{\mathcal{C}}
\newcommand{\mcd}{\mathcal{D}}

\newcommand{\mcf}{\mathcal{F}}

\newcommand{\mci}{\mathcal{I}}

\newcommand{\mcl}{\mathcal{L}}
\newcommand{\mcm}{\mathcal{M}}

\newcommand{\mcu}{\mathcal{U}}

\newcommand{\tQ}{\tilde{Q}}

\renewcommand{\aa}{\alpha}
\newcommand{\bb}{\beta}
\newcommand{\lam}{\lambda}

\newcommand{\ov}{\overline}
\newcommand{\eps}{\epsilon}

\newcommand{\GL}{\textnormal{GL}}
\newcommand{\SL}{\textnormal{SL}}

\newcommand{\Gr}{{\rm Gr}}

\def\tGr{\widetilde \Gr}

\newcommand\preceqdot{\mathrel{\ooalign{$\prec$\cr
  \hidewidth\raise0.225ex\hbox{$\cdot\mkern0.5mu$}\cr}}}

\DeclareMathOperator{\Aut}{Aut}

\newcommand{\Bound}{{\rm B}}
\newcommand{\rmci}{\vec{\mci}}
\newcommand{\rightI}{\vec{I}}

\title{Cyclic symmetry loci in Grassmannians}
\author{Chris Fraser}
\address{University of Minnesota, 206 Church St SE, Minneapolis.}
\email{cfraser@umn.edu}

\numberwithin{equation}{section}

\begin{document}
\begin{abstract}
The Grassmannian $\Gr(k,n)$ admits an action by a finite cyclic group via the cyclic shift map. We give a simple description of the points fixed by each element of this cyclic group, extending Karp's description of the points fixed by the cyclic shift itself. We give a cell decomposition of the set of totally nonnegative points in each cyclic symmetry locus and describe efficient total positivity tests, extending results of Postnikov to the cyclically symmetric setting. We describe a conjectural generalized cluster structure on cyclic symmetry loci provided the order of the orbifold point is sufficiently large. The generalized exchange relations we find should be a Higher Teichm\"uller analogue of the relations Chekhov and Shapiro used to study Teichm\"uller theory of orbifolds.  
\end{abstract}

\maketitle
\vspace{-.3in}

\section*{Introduction}\label{secn:intro}
Let $\Gr(k,n)$ denote the Grassmannian of $k$-subspaces in $\bbc^n$. Cyclic permutation of the coordinates on $\bbc^n$ induces an automorphism of order $n$, the {\sl cyclic shift automorphism} $\rho \in \Aut(\Gr(k,n))$.\footnote{We suppress a certain sign from the definition of $\rho$ until the body of the paper.} This cyclic symmetry is an important feature underlying the structures describing Poisson geometry, total nonnegativity, and clusters, for the Grassmannian. 

For example, the cyclic shift map is an automorphism of $\Gr(k,n)$ with its standard Poisson structure \cite{Yakimov}. The Grasmannian bears a distinguished semialgebraic subset $\Gr(k,n)_{\geq 0} \subset \Gr(k,n)$, the {\sl TNN Grassmannian}, consisting of those points whose Pl\"ucker coordinates lie in $\bbr_{\geq 0}$. Postnikov (following Lusztig \cite{Lusztig94}) endowed $\Gr(k,n)_{\geq 0}$ with a cell decomposition \cite{Postnikov}. The cyclic shift map is a cellular self-homeomorphism of $\Gr(k,n)_{\geq 0}$, and the resulting cyclic symmetry features prominently in the combinatorics which indexes the cells. Recent work of Galashin, Karp, and Lam \cite{GKLI} uses the cyclic shift map as a main ingredient in the construction of a {\sl contractive flow } on $\Gr(k,n)_{\geq 0}$. This flow yields a homeomorphism of $\Gr(k,n)_{\geq 0}$ with a closed ball, confirming a longstanding conjecture of Postnikov. 

From a cluster perspective, the cyclic shift map is a {\sl cluster automorphism}, permuting the cluster variables, the clusters, and the (Gross-Hacking-Keel-Kontsevich) {\sl $\theta$ basis} \cite{GHKK} in the cluster algebra structure on the homogeneous coordinate ring $\bbc[\Gr(k,n)]$. The last of these facts gives rise to a new proof of Rhoades' protoypical cyclic seiving result for for the action of promotion on SSYT of rectangular shape \cite{Rhoades,ShenWeng}.

\medskip

This paper systematically studies the fixed point loci of this interesting cyclic action on the Grassmannian. For $\ell \in \bbz / n \bbz$, define the {\sl $\ell$-fixed locus} $\Gr(k,n)^{\rho^\ell} \subset \Gr(k,n)$ as the subset of $\rho^\ell$-fixed points. We also refer to $\Gr(k,n)^{\rho^\ell}$ as a {\sl cyclic symmetry locus}. Karp showed that the $1$-fixed locus in $\Gr(k,n)$ consists of exactly $\binom{n}k$ many points, exactly one of which is TNN \cite{Karp}. 
On the other hand, the $n$-fixed locus is $\Gr(k,n)$ itself, with a cluster algebra structure and a rich theory of total nonnegativity. This paper aims to appropriately extend such results to arbitrary cyclic symmetry loci. 

\medskip

First, we describe each cyclic symmetry locus as a projective algebraic variety (cf. Proposition~\ref{prop:components}): it is an explicit disjoint union of algebraic varieties, each of which is a product of (smaller) Grassmannians. We point out that, after an appropriate linear change of coordinates on $\bbc^n$, we can view the cyclic symmetry locus as a  
set of torus-fixed points, and as a Richardson variety in $\Gr(k,n)$. 

Second, we describe the semialgebraic subset $\Gr(k,n)^{\rho^\ell}_{\geq 0} \subset \Gr(k,n)^{\rho^\ell}$ consisting of TNN and $\ell$-fixed points (cf. Theorem~\ref{thm:cells}). We give a cell decomposition of this space and describe the cell closure partial order as the dual poset of an order ideal in Bruhat order on the affine symmetric group $\tilde{S}_\ell$. As $k \to \infty$, these lower order ideas cover all of $\tilde{S}_\ell$. We also show that $\Gr(k,n)^{\rho^\ell}_{\geq 0}$ is homeomorphic to a closed ball using the techniques from \cite{GKLI}.

Third, we study the existence and structure of efficient {\sl total positivity (TP) tests} for $\ell$-fixed points (cf.~Theorem~\ref{thm:TPtests}). We ask: given an $\ell$-fixed point~$X$, how many (and which) Pl\"ucker coordinates must we test in order to guarantee that {\sl all} Pl\"ucker coordinates of~$X$ are positive. We demonstrate the existence of efficient TP tests
for the $\ell$-fixed locus in $\Gr(k,n)$, for {\sl any value} of the parameters $n,k,\ell$. Our construction uses a cyclically symmetric version of {\sl bridge decompositions} of plabic graphs, and yields efficient TP tests for all cells in  $\Gr(k,n)_{\geq 0}^{\rho^\ell}$, not merely for the top cell $\Gr(k,n)^{\rho^\ell}_{>0}$.  

Fourth and finally, we investigate the existence of cluster structures on the cyclic symmetry locus. The locus is a disconnected space, but it admits a {\sl distinguished component} $\mcd = \mcd_n(k,\ell)$ containing the $\ell$-fixed TNN points. Let $p = \frac{n}{\gcd(\ell,n)}$ denote the order of $\rho^\ell \in {\rm Aut}(\Gr(k,n))$, which we refer to as the {\sl order of the orbifold point}. Provided $p \in [k,\infty)$, we endow $\mcd$ with a (conjectural) atlas of generalized cluster charts, each of which is an efficient total positivity test for $\Gr(k,n)^{\rho^\ell}$.
{\sl Generalized cluster algebras} \cite{ChekShap,CAPG} are a a well-behaved generalization of Fomin and Zelevinsky's cluster algebras in which the usual exchange binomials are replaced by longer sums of monomials. Our atlas is conjectural: we are currently only able to show that our upper generalized cluster algebra coincides with $\bbc[\mcd]$ for convenient values of the parameters $k,\ell,n$ (cf.~Theorem~\ref{thm:onestep} and discussion thereafter). We henceforth omit the adjective conjectural from the introduction.  The assumption $p \in [k,\infty)$ is exactly the assumption that $\mcd_{p\ell}(k,\ell)$ is a product of projective spaces $(\bbp^{\ell-1})^k$, not merely a product of Grassmannians.

Our generalized cluster structure on $\mcd$ has one generalized cluster exchange relation, with the remaining relations binomial. Fixing $k,\ell$ and varying the order of the orbifold point $p \in [k,\infty)$, the cluster structure on $\bbc[\mcd]$ is a specialization of a generalized cluster algebra 
$\mca_{\rm cyc}(k,\ell)$. In the latter, the coefficients $z_s$ of the generalized exchange relation are indeterminates. To obtain the former, we specialize the indeterminate $z_s$ to a $q$-binomial coefficient at a $p$th root of unity: $z_s \mapsto \binom k s_q \in \bbc$ with $q := e^{\frac{2\pi i}{p}}$. The specializations $p = k$ and $p=\infty$ correspond to the right and left {\sl companion} cluster algebras, a pair of Fomin-Zelevsinky cluster algebras associated to any generalized cluster algebra \cite{NakRup}.

Under an extra assumption $k<\ell$ on the parameters, Gekhtman, Shapiro, and Vainshtein have constructed a cluster structure on an affine space of {\sl periodic band matrices} whose width and periodicity are recorded by parameters $k$ and $\ell$ \cite{GSVStaircase}. We introduce a notion of {\sl quasi-homomorphism} \cite{FraserQH} of generalized cluster algebras and show that $\mca_{\rm cyc}(k,\ell)$ is a quasi-homomorphic image of the GSV structure on band matrices. The algebra $\mca_{\rm cyc}(k,\ell)$ makes sense when $k \leq \ell$, so this connection might be a way of extending \cite{GSVStaircase} to these cases. We also conjecture (Conjecture~\ref{conj:quantumaffine})  a relationship between $\mca_{\rm cyc}(k,\ell)$ and the representation theory of quantum affine $\mathfrak{sl}_k$ at $\ell$th roots of unity building on \cite{Gleitz}.

Another motivation for studying cyclic symmetry loci comes from Higher Teichm\"uller theory. The first appearance of generalized cluster algebras in nature \cite{ChekShap} was in the setting of decorated Teichm\"uller theory of orbifold surfaces (cf.~\cite{GSVDouble,GSVStaircase, Gleitz,WKB} for other appearances.)
In that story, the exchange polynomial for an arc ending near an orbifold point of order $p$ is the {\sl orbifold Ptolemy relation} $a^2+2 \cos(\pi/p)ab+b^2$ where $a$ and $b$ are the cluster monomials participating in the exchange relation. 
Our ``$\SL_k$-{\sl higher orbifold Ptolemy relation}''  $\sum_{s=0}^k {k \brack s}_{q = e^{\frac{2\pi i}{p}}} a^{s}b^{k-s}$ specializes to the above orbifold Ptolemy relation when $k=2$. We interpret the specialized $q$-binomial coefficients appearing in this relation as certain ratios of Pl\"ucker coordinates of the unique point in $\Gr(k,n)^{\rho}_{>0}$. 

Another viewpoint on the generalized cluster algebra $\bbc[\mcd_n(k,\ell)]$ is that it is is a ``folding'' of the Grassmannian cluster algebra along the $\rho^\ell$ symmetry. The naive approach to folding a cluster algebra 
along a finite group of symmetries is poorly behaved \cite[Chapter 4]{CABook}. Our constructions for cyclic symmetry loci suggest a recipe for overcoming these obstacles in examples (cf.~especially Lemma~\ref{lem:vanillaDT}). 
A different approach to quotienting cluster algebras by group actions is given in \cite{PaqSchiff}. 

\medskip

{\it Organization.} Section \ref{secn:loci} defines the $\ell$-fixed locus. Section~\ref{secn:background} collects background on $\Gr(k,n)_{\geq 0}$, $\bbc[\Gr(k,n)]$ as a cluster algebra, and generalized cluster algebras (including a definition of quasi-homomorphisms). Section~\ref{secn:leapweak} introduces a partially ordered set (the {\sl bridge order}) used in our inductive proofs. Section~\ref{secn:components} describes the $\ell$-fixed locus as a projective variety. Section~\ref{secn:cells} studies the TNN part of the cyclic symmetry locus. Section~\ref{secn:TPtests} constructs efficient TP tests. Section~\ref{secn:clusters} gives our main theorem and conjecture concerning clusters and summarizes partial results. Section~\ref{secn:quantum} discusses connections with band matrices and quantum affine algebras. Section~\ref{secn:folding} discusses our approach to folding. Section~\ref{secn:examples} studies the finite type examples, including $p<k$ examples. It exhibits proper containment of cluster algebra in upper cluster algebra, and shows that $\rho$ need not be a cluster automorphism of $\bbc[\mcd]$. Section~\ref{secn:proofs} collects some lengthier proofs.

\section*{Acknowledgements}
We thank P.~Pylyavksyy for encouraging us to investigate cyclic symmetry loci and generalized cluster structures. We thank F.~Bergeron, S.~Fomin, M.~Gekhtman, S.~Hopkins, S.~Karp, J.~Levinson, V.~Reiner, M.~Sherman-Bennett, and K.~Trampel for helpful conversations. This work was supported by the NSF grant DMS-1745638.

\section{Cyclic symmetry loci}\label{secn:loci}
We collect background on the Grassmannian and its cyclic shift automorphism. The subset of points fixed by a given iterate of the cyclic shift map is our main object of study. 

Let $\Gr(k,n)$ denote the Grassmann manifold of $k$-dimensional subspaces in $\bbc^n$. We view $\Gr(k,n) \subset \bbp^{\binom n k -1}$ as a closed subvariety of projective space in the following way. A choice of ordered basis for $X \in \Gr(k,n)$ determines a $k \times n$ matrix $M$ whose rows are the basis vectors. Then the {\sl Pl\"ucker embedding} $\Gr(k,n) \hookrightarrow \bbp^{\binom n k -1}$
sends
$$X \mapsto (\Delta_I(X))_{I \in \binom {[n]}k}:= (\Delta_I(M))_{I \in \binom {[n]}k}.$$
The map is injective and does not depend on the choice of basis for $X$. The numbers $\Delta_I(X)$ are the {\sl Pl\"ucker coordinates} of $X$. They are homogeneous coordinates: they are well-defined only up to simultaneous rescaling. The image of Pl\"ucker embedding is defined by well known quadratic relations in Pl\"ucker coordinates known as the {\sl Pl\"ucker relations}. 

We denote by $\bbc[\Gr(k,n)]$ the homogeneous coordinate ring of the Grassmannian in its Pl\"ucker embedding. Concretely, this is the $\bbc$-algebra generated by the $\binom n k$ symbols $\Delta_I$, subject to the Pl\"ucker relations. For $\mcc \subset \binom{[n]}k$ we write $\Delta(\mcc)$ for $\{\Delta(I) \colon I \in \mcc\}$. We often abbreviate $\Delta_I$ to $I$, writing e.g. $124$ in place of $\Delta_{124}$.  

Linear automorphisms of $\bbc^n$ induce automorphisms of $\Gr(k,n)$. (These are the {\sl only} automorphisms of $\Gr(k,n)$ when $n \neq 2k$, and when $n=2k$, there is an
``extra'' automorphism induced by {\sl Grassmann duality} $\Gr(k,n) \cong \Gr(n-k,n)$.)

\begin{defn}
For fixed $1 \leq k \leq n$, let $\rho$ denote the (signed) circulant matrix 
\begin{equation}\label{eq:rhomatrix}
\rho = \begin{pmatrix}
0 & 0 & \cdots & 0 & (-1)^{k-1}\\
1 & 0 & \cdots & 0 & 0\\
0 & 1 & \cdots & 0 & 0\\
\vdots & \vdots & \ddots& \vdots & \vdots\\
0& 0& \cdots& 1& 0\\
\end{pmatrix} \in \GL(\bbc^n).
\end{equation}
The {\sl cyclic shift map} $\rho \in \Aut(\Gr(k,n))$ is the automorphism induced by $\rho \in \GL(\bbc^n)$. 
\end{defn}

Since $\rho^n \in \GL(\bbc^n)$ is a scalar multiple of the identity matrix, $\rho^n =1 \in {\rm Aut}(\Gr(k,n))$.


We denote by the same symbol $\rho$ the cyclic shift $i \mapsto (i+1 \mod n)$ on $[n]$, and also the induced map on $k$-subsets $\rho \colon \binom{[n]}k \to \binom{[n]}k$. Due to the choice of sign $(-1)^{k-1}$ in \eqref{eq:rhomatrix}, the 
pullback $\rho^* \in \Aut(\bbc[\Gr(k,n)])$ acts on Pl\"ucker coordinate by $$\rho^*(\Delta_{I}) = \Delta_{\rho(I)} .$$

\begin{defn}\label{defn:locus}
The {\sl cyclic symmetry locus} is the subset 
\begin{equation}\label{eq:locus}
\Gr(k,n)^{\rho^\ell} := \{X \in \Gr(k,n) \colon \rho^\ell(X) = X\}
\end{equation}
of $\rho^\ell$-fixed points. We call it the {\sl $\ell$-fixed locus} when emphasizing a particular value of $\ell$.

\end{defn}

With $n$ fixed in our minds, we always denote by $p = \frac{n}{\gcd(\ell,n)}$ the order of $\rho^\ell \in {\rm Aut}(\Gr(k,n))$. We refer to $p$ as the {\sl order of the orbifold point} in anticipated analogy with \cite{ChekShap}. It is often convenient to assume that $n=p\ell$, which represents no loss in generality because 
$\Gr(k,n)^{\rho^\ell} = \Gr(k,n)^{\rho^{\ell'}}$ where 
$\ell' := \gcd(\ell,n)$.

Because Pl\"ucker coordinates are homogeneous coordinates, the condition that $X$ be $\ell$-fixed is the condition that for some $\zeta \in \bbc^*$, 
\begin{equation}\label{eq:character}
\Delta_I(X) = \zeta\Delta_I(\rho^\ell(X)) = \zeta \Delta_{\rho^\ell(I)}(X) \text{ for all } I \in \binom{[n]}k.
\end{equation}
Since $\rho^\ell$ has order $p$ when acting on $\binom{[n]}k$, we 
conclude that $\zeta$ must be a $p$th root of unity.

The linear equations on Pl\"ucker coordinates \eqref{eq:character}, for a fixed $p$th root $\zeta$, determine a subvariety of $\Gr(k,n)$. The 
$\ell$-fixed locus is a disjoint union of these subvarieties. Each such subvariety is typically nonempty, disconnected (in particular, reducible as an algebraic variety), and non equidimensional (cf.~Section~\ref{secn:components}).

For most of the paper, we are interested in points in the Grassmannian with nonnegative Pl\"ucker coordinates. For these, the scalar 
$\zeta$ in \eqref{eq:character} must equal~1.

The following theorem of Karp is our starting point. 
\begin{thm}[{\cite{Karp}}]\label{thm:Karpfinitelymany}
The 1-fixed locus in $\Gr(k,n)$ consists of exactly $\binom n k $ points. Exactly one of these points has all of its Pl\"ucker coordinates real and nonnegative.
\end{thm}

That is, the 1-fixed locus is a zero-dimensional space. On the other hand, the $n$-fixed locus in $\Gr(k,n)$ is the Grassmannian itself.

\begin{example}\label{eg:firstegs}
As a first example beyond 1-fixed and $n$-fixed loci, 
consider $\Gr(2,4)^{\rho^2}$, which has $p = 2$. 
Points in $\Gr(2,4)$ have six Pl\"ucker coordinates constrained by the Pl\"ucker relation $\Delta_{13}\Delta_{24} = \Delta_{12}\Delta_{34}+\Delta_{14}\Delta_{23}$. If $X \in \Gr(2,4)^{\rho^2}$, then $\rho^2$ acts on the Pl\"ucker coordinates by the scalar $\zeta \in \{\pm 1\}$ as in \eqref{eq:character}.  

When $\zeta = +1$, the Pl\"ucker relation becomes $\Delta_{13}\Delta_{24} = \Delta_{12}^2+\Delta_{23}^2$, which defines the nonsingular quadric surface in $\bbp^3$. In particular $\dim \Gr(2,4)^{\rho^2= 1} = 2$.
On a Zariski-open subset $\Delta_{13}(X) \neq 0$, and $X \in \Gr(2,4)^{\rho^2}$ bears a matrix representative 
$M = \begin{pmatrix} 
1 & y & 0 & -x \\
0 & x & 1 &  y
\end{pmatrix}$ for some $x,y \in \bbc$. 

When $\zeta =-1$, it follows that $\Delta_{13} = \Delta_{24} = 0$. The Pl\"ucker relation becomes $\Delta_{12}^2+\Delta_{23}^2 = 0$. 
The space $\Gr(2,4)^{\rho^2=-1}$ consists of two points represented by the matrices 
$\begin{pmatrix} 
1 & 0 & \pm i & 0 \\
0 & 1 & 0 & \pm i 
\end{pmatrix}$. 

Thus $\Gr(2,4)^{\rho^2} \subset \bbp^5$ is the disjoint union of a complex surface and two points. 
\end{example}

\section{Background}\label{secn:background}
We summarize the basics of the TNN Grassmannian, (generalized) cluster algebras, and Grassmannians. We introduce quasi-homomorphisms of generalized cluster algebras. 

\subsection{TNN background}\label{subsecn:bgcells}

\begin{defn}
The {\sl matroid} of $X \in \Gr(k,n)$ is its collection of non-vanishing Pl\"ucker coordinates: ${\rm matroid}(X) := \{I \in \binom{[n]}k \colon \Delta_I(X) \neq 0\}$. For $\mcm \subset \binom{[n]}k$, the {\sl matroid stratum} is the quasiprojective subvariety ${\rm GGMS}(\mcm) := \{X \in \Gr(k,n) \colon {\rm matroid}(X) = \Gr(k,n)\}$.
\end{defn}

The letters GGMS stand for Gelfand, Goresky, MacPherson, and Serganova.  
The concept of matroid has an abstract definition (a collection of $k$-subsets satisfying the {\sl exchange condition}), and the matroids arising from points in the Grassmannian are called {\sl realizable } (over $\bbc$). We have a decomposition $\Gr(k,n) = \coprod_\mcm {\rm GGMS}(\mcm)$ indexed by realizable matroids $\mcm \subset \binom{[n]}k$.

Let $\bbr \bbp_{\geq 0}^N$ denote the image of $(\bbr_{\geq 0}^{N+1} \setminus 0) \subset (\bbc^{N+1} \setminus 0) \twoheadrightarrow \bbp^{N}$ with the latter the usual quotient map. 
Define similarly $\bbr \bbp_{> 0}^N$.

\begin{defn}
The TNN Grassmannian is $\Gr(k,n)_{\geq 0} := \Gr(k,n) \cap \bbr \bbp_{\geq 0}^{\binom n k -1}$. The TP Grassmannian is 
  $\Gr(k,n)_{>0} := \Gr(k,n) \cap \bbr \bbp_{> 0}^{\binom n k -1}$. 
For $\mcm \subset \binom{[n]}k$, we denote by $\Gr(\mcm)_{>0} := {\rm GGMS}(\mcm) \cap \Gr(k,n)_{\geq 0}$ and denote by $\Gr(\mcm)_{\geq 0}$ its closure in the Hausdorff topology. 
\end{defn}

The abbrevations TNN and TP stand for {\sl totally nonnegative} and {\sl totally positive}. The matroids $\mcm$ for which $\Gr(\mcm)_{>0}$ is nonempty are the {\sl positroids}. 
Any positroid $\mcm$ has an associated {\sl positroid variety} $\Gr(\mcm) \subset \Gr(k,n)$, i.e. the projective subvariety of $\Gr(k,n)$ defined by the vanishing of $\Delta_I \colon I \notin \mcm$. As a special case, the {\sl uniform matroid} $\mcm = \binom{[n]}k$ is a positroid, for whom 
the notions $\Gr(\mcm)$, $\Gr(\mcm)_{>0}$, and $\Gr(\mcm)_{\geq 0}$ recover  $\Gr(k,n)$, $\Gr(k,n)_{>0}$, and $\Gr(k,n)_{\geq 0}$ respectively. (The last of these claims is nontrivial but true.) 

A {\sl cell decomposition} of a topological space $X$ is a decomposition $X = \coprod_{\aa \in \mca} X_\aa$ of $X$ into subspaces $X_\aa$ (called {\sl cells}). Each cell must be homeomorphic to an open ball of some dimension, and the boundary ${\rm cl}(X_\aa) \setminus X_\aa \subset X$ of any cell must equal a union of smaller-dimensional cells. The {\sl face poset} of such a cell decomposition is $(\mca,\leq)$ where the relation is $\bb \leq \aa$ if $X_\bb \subset {\rm cl}(X_\aa)$. We view it as a ranked poset with rank function  ${\rm rk}(\aa) = \dim X_\aa$. The {\sl dual} of the face poset is the poset obtained by reversing all inequalities and replacing cell dimension by cell codimension. 

A cell decomposition is a {\sl CW complex} if the topology of $X$ is the weak topology with respect to the cellular inclusions $X_\aa \subset X$. A CW complex is {\sl regular} if these inclusions extend to homeomorphisms of ${\rm cl}(X_\aa)$ with a closed ball. 

Postnikov \cite{Postnikov} proved that the decomposition $\Gr(k,n)_{\geq 0} = \coprod_\mcm \Gr(\mcm)_{>0}$ into {\sl positroid cells} is a cell decomposition. Moreover, it is a CW complex \cite{PSW} and is regular \cite{GKLIII}. 
The combinatorics underlying the cell decomposition of $\Gr(k,n)_{\geq 0}$ is very rich \cite{Postnikov, KLS}. The cells (i.e., the positroids $\mcm \subset \binom{[n]}k$) are labeled by several elegant families of combinatorial objects, related to each other by known bijections. These objects include bounded affine permutations, Grassmann necklaces, decorated permutations, plabic graphs modulo move equivalence, and Le diagrams. The first of these is most important to our presentation.

\begin{defn}
Let $\tilde{S}_\ell$ the group of bijections $f \colon \bbz \to \bbz$ which are $\ell$-{\sl periodic}: $f(i+\ell) = f(i)+\ell$ for all $i \in \bbz$. The recipe $f \mapsto \frac{1}{\ell}\sum_{i=1}^\ell(f(i)-i)$ is a group homomorphism ${\rm av} \colon \tilde{S}_\ell \twoheadrightarrow \bbz$ whose fibers we denote by $\tilde{S}^k_\ell:= {\rm av}^{-1}(k) \subset \tilde{S}_\ell$.   

We say that $f \in \tilde{S}_\ell$ is $n$-{\sl bounded} if $i \leq f(i) \leq i+n$ for all $i \in \bbz$. We let $\Bound_n(k,\ell) \subset \tilde{S}^\ell_k$ denote the set of $n$-bounded $\ell$-periodic affine permutations with ${\rm av}(f) = k$. 
\end{defn}

We typically denote $f \in \Bound_n(k,\ell)$ by its {\sl window notation} $[f(1),\dots,f(\ell)]$. This data specifies $f$ by $\ell$-periodicity.

If $f \in \tilde{S}_\ell$ then $f(1),\dots,f(\ell)$ is a permutation modulo $\ell$, so av is indeed a homomorphism to $\bbz$, not to $\bbq$. 
Its kernel $\tilde{S}^0_\ell$ is the {\sl affine symmetric group}, which is a Coxeter group of affine type $\tilde{A}_{\ell-1}$. The Coxeter generators are the simple transpositions $(s_i)_{i \in [\ell]}$ switching the values $i+a\ell \leftrightarrow i+1+a\ell$ for all $a \in \bbz$. The reflections in the Coxeter group are the transpositions $t_{i,j+s\ell}$ switching the values $i+a\ell \leftrightarrow j+(a+s)\ell$ for $a \in \bbz$.

We denote by $\ell(w)$ the Coxeter length of $w \in \tilde{S}^0_\ell$. We denote by $\leq_R$ (resp. $\leq_B$) the {\sl right weak order} (resp. {\sl Bruhat order}) on $\tilde{S}^0_\ell$, whose cover relations are of the form $w \lessdot ws_i$ whenever $\ell(ws_i) = \ell(w)+1$ (resp. $w \lessdot wt_{i,j+s\ell}$ whenever $\ell(wt_{i,j+s\ell}) = \ell(w)+1$).

Let ${\rm id}_k \in \tilde{S}^k_\ell$ be the bijection $i \mapsto i+k$ for $i \in \bbz$. We extend $\ell$, $\leq_B$, and $\leq_R$ to $\tilde{S}^k$ by defining $\ell(f) := \ell({\rm id}_k^{-1}f) \in \tilde{S}^0_\ell$,  etc. We define these three notions for $\Bound_n(k,\ell) \subset \tilde{S}^k_\ell$ by restriction. Both partial orders are ranked by the length function.

Positroids $\mcm \subset \binom {[n]} k$ and affine permutations $f \in \Bound_n(k,n)$ are in bijection as follows. Let $X \in \Gr(\mcm)_{>0}$ with $k\times n$ matrix representative $M$. Suppose $M$ has column vectors $v_1,\dots,v_n$. Extend the matrix $M$ 
$n$-periodically in both directions to a $k \times \infty$ matrix.  Then the affine permutation $f = f(\mcm)$ determined by the positroid $\mcm$ is the affine permutation with values 
\begin{equation}\label{eq:matroidtoperm}
f(i) = \min\{j \geq i \colon \, v_i \in {\rm span} (v_{i+1},\dots,v_j)\}.
\end{equation}
We use subscripts to denote this bijection, writing e.g. $\mcm_f$ for the positroid corresponding to $f \in \Bound_n(k,n)$. We also write $\Gr(f)_{>0} = := \Gr(\mcm_f)_{>0}.$ A rephrasing 

\begin{thm}[{\cite{Postnikov,KLS}}]\label{thm:cellsusual}
The face poset of $\Gr(k,n)_{\geq 0} = \cup_{f \in \Bound_n(k,n)} \Gr(\mcm_f)_{>0}$ is {\sl dual to} the Bruhat order on $Bound_n(k,n)$ (as ranked posets).  
\end{thm}


\begin{rmk}
In light of Theorem~\ref{thm:cellsusual}, almost all previous attention has been on the special case of $\Bound_n(k,n)$. We need the generalization $\Bound_n(k,\ell)$ for describing cells in cyclic symmetry loci. Galashin and Lam previously considered an even more general subset of affine permutations in the context of triangulations of amplituhedra, namely those whose values satisfy $i-a \leq f(i) \leq f(i)+b$ for all $i \in \bbz$, where $a,b \in \bbn$ are fixed numbers \cite{GalashinLamParity}. 
\end{rmk}

A third type of object in bijection with positroids $\mcm \subset \binom{[n]}{k}$ is a {\sl Grassmann necklace}. These have a simple intrinsic definition, which we omit for brevity. It suffices for our purposes us to say that a Grassmann necklace is a certain type of $n$-tuple $\rmci = (I_1,\dots,I_n) \subset \binom{[n]}k$, and that if $\rmci = \rmci_f$ is the Grassmann necklace corresponding to the positroid $\mcm_f$, then $I_{i+1} = I_i \setminus i \cup \{f(i) \mod n\}$. Conversely, one has $\mcm_f = \{I \in \binom{[n]}k \colon \, I \geq_i I_i \text{ for all } i  \in [n]\}$. Here, $<_i$ denotes the total order $i<_i \cdots<_i n <_i 1<_i\cdots<_ii-1$ on $[n]$, and the comparison $\geq_i$ on $k$-subsets is done lexicographically.

\subsection{Grassmannian cluster algebras}
We assume familiarity with cluster algebras of geometric type as defined by Fomin and Zelevinsky \cite{CAI}. A seed in such a cluster algebra is a pair $({\bf x},\tilde{B})$. Here, ${\bf x}$ is an {\sl extended cluster} consisting of $N$ mutable and $\ell$ frozen variables, and $\tilde{B}$ is an {\sl extended exchange matrix}, an $(N+\ell) \times N$ integer matrix. The first $N$ rows of $\tilde{B}$, namely the {\sl exchange matrix} $B$, must be skew-symmetrizable. When $B$ is skew-symmetric, $\tilde{B}$ can be encoded by an {\sl extended quiver} $\tilde{Q}$ (a quiver with a choice of frozen vertices).

The homogeneous coordinate ring $\bbc[\Gr(k,n)]$ is a prototypical example  of such a cluster algebra \cite{Scott}. We assume familiarity with the concept of a {\sl maximal weakly separated collection} $\mcc \subset \binom{[n]}k$. Such a collection gives rise to an extended quiver $\tQ(\mcc)$ (the dual quiver of the {\sl plabic tiling} determined by $\mcc$), and to a seed $\Sigma(\mcc) = (\Delta(\mcc),\tQ(\mcc))$ whose cluster algebra $\mca(\Sigma(\mcc)) = \bbc[\Gr(k,n)]$. The seeds $\Sigma(\mcc)$ indexed by weakly separated collections are the {\sl only} seeds in $\bbc[\Gr(k,n)]$ consisting entirely of Pl\"ucker coordinates. They comprise a finite subset of seeds in the Grassmannian cluster algebra, and this finite subset is connected by mutations. Most Grassmannians, however, have infinitely many clusters, so these are not all seeds.

\subsection{Generalized cluster algebras}\label{subsecn:bgclusters}
Generalized cluster algebras were introduced by Chekhov and Shapiro \cite{ChekShap} (cf.~also \cite[Lemma 1.5]{CAPG}). We use the prefix CS- as an adjective indicating generalized cluster algebras, and use the prefix FZ- to indicate Fomin-Zelevinsky cluster algebras.

There are various convention choices involved in defining a CS-cluster algebra. Our next definition is the simplest possible version, and suffices for constructing seeds in cyclic symmetry loci. 
\begin{defn}\label{defn:CSseed}
A CS-seed is a triple $(\tilde{\bf x},\tilde{B},{\bf z})$ where $(\tilde{\bf x},\tilde{B})$ is a FZ-seed of geometric type and ${\bf z}$ is an $N$-tuple of {\sl coefficient strings}. The $k$th coefficient string  $(z_{k;0},\dots,z_{k;d_k})$ is a collection of indeterminates subject to the conditions $d_k \geq 1$, $z_{k;0} = z_{k;d_k} = 1$, and the palindromicity condition $z_{k;s} = z_{k;d_k-s}$. We refer to $z_{k;s}$ as a {\sl coefficient string variable} and treat it is an extended cluster variable. Each coefficient string encodes an {\sl exchange polynomial} $Z_k(u,v) = \sum_{s=0}^{d_k}z_{k;s}u^sv^{d_k-s} \in \bbc[u,v]$. A mutable index $k \in [N]$ is {\sl special} if $d_k > 1$. We denote by $D = (d_k \delta_{k,j})_{i,j \in [N]}$ the diagonal matrix with entries $d_k$. 

Associate to a CS-seed the monomials $M^\pm_k = \prod_{i\in [N+\ell]} x_i^{ [\pm b_{ik}]_+}.$ To mutate a CS-seed in direction $k \in [N]$, replace the extended exchange matrix $\tilde{B}$ by the matrix $\mu_k(\tilde{B}D)D^{-1}$, where $\mu_k$ is FZ- matrix mutation. Leave the coefficient strings unchanged. Finally, replace $x_k$ by $x_k' = \frac{1}{x_k}Z_k(M^+_k,M^-_k)$ (and do nothing to the other cluster variables). 
\end{defn}

For non-special variables, the exchange polynomial is the usual binomial appearing on the right hand side of a FZ- exchange relation.

The usual definitions for FZ-cluster algebras make sense for CS-algebras. One has the {\sl cluster algebra}, {\sl upper cluster algebra}, {\sl lower bound algebra}, and {\sl upper bound algebra}, associated to any seed. Sometimes, it is more to convenient to localize these algebras at the frozen variables. We denote such localizations by $^\circ$. For example, if $\mca$ is the cluster algebra (the $\bbc$-algebra generated by all extended cluster variables), then $\mca^\circ$ denotes the algebra generated by extended cluster variables and the inverses of frozen variables. A {\sl cluster monomial} is a product of cluster variables from any cluster. 

The following operation is important to our main result (Theorem~\ref{thm:onestep}). Given a collection of complex numbers $(\eta_{k;s})_{k \in {N}, \, s \in [d_k-1]}$ satisfying palindromicity $\eta_{k;s} = \eta_{k;d_k-s}$, we have a {\sl specialization} of the CS-cluster algebra $\mca$ defined by the substitutions $z_{k;s} \mapsto \eta_{k;s}$. We still use the terminology cluster variable, cluster algebra, upper cluster algebra, etc., for the same notions but defined after performing such a specialization.

The {\sl right companion cluster algebra} of a generalized cluster algebra is the FZ-cluster algebra with initial exchange matrix $\tilde{B}D$ \cite{NakRup}. Equivalently, it is obtained by the algebra specialization $z_{k;s} \mapsto 0$. Chekhov and Shapiro showed that a generalized cluster algebra has finitely many CS-seeds if and only if its  right companion cluster algebra is an FZ- cluster algebra of finite type. 
More generally, Cao and Li proved that the exchange graph of a generalized cluster algebra coincides with the exchange graph of its right companion \cite[Theorem 3.7]{CaoLi}.

There is also a {\sl left companion cluster algebra} \cite{NakRup} of a CS-cluster algebra, an FZ-cluster algebra with initial exchange matrix $DB$. A cluster variable $x_k$ in the left companion can be obtained from a corresponding cluster variable in its CS-cluster algebra $\mca(\Sigma)$ by specializing all coefficient strings $z_{k,s} = \binom {d_k}s$, and then raising the corresponding CS-cluster variable to the $\frac{1}{d_k}$th power.

\subsection{Quasi-homomorphisms of generalized cluster algebras}
Later on, we would like to compare our generalized cluster structure on cyclic symmetry loci with a generalized cluster structure on periodic band matrices. To make such a comparison, we extend here 
the notion of {\sl quasi-homomorphism} of cluster algebras \cite{FraserQH} to the setting of generalized cluster algebras. A reader not interested in this comparison, or in quasi-homomorphisms, could safely skip this section. 

\begin{defn}\label{defn:nnCSseed}
Let $\bbp$ be a fixed abelian group. A {\sl non-normalized CS-seed} is a triple $({\bf x}, B, {\bf p})$, where $({\bf x},B)$ is a FZ- seed with no frozen variables. The data ${\bf p}$ is an $N$-tuple of coefficient strings $(p_{k;s})_{s \in [0,d_k]}$ indexed by $k \in [N]$, with $p_{k;s} \in \bbp$. Each coefficient string encodes an exchange polynomial $Z_k(u,v) = \sum_{s=0}^{d_k}p_{k;s}u^sv^{k-s}$. 

For $k \in [N]$ define monomials $M^{\pm}_k = \prod_{i \in [N]}x_i^{[\pm b_{ik}]+}$. 
Non-normalized seed mutation in direction $k$ replaces the $k$th cluster variable  $x_k \mapsto \frac{1}{x_k}Z(M^+_k,M^-_k)$. It replaces the exchange matrix by $B \mapsto \mu_k(BD)D^{-1}$. It replaces the coefficient strings by the (non-deterministic) rule 
\begin{equation}\label{eq:nnpmutation}
p_{k;s} \mapsto p_{k;d_k-s} \text{ and } \frac{p_{j;s}}{p_{j;0}} \mapsto \frac{p_{j;s}}{p_{j;0}}p_{k;dk}^{[sb_{kj}]_+}p_{k;d_0}^{-[-sb_{kj}]_+} \text{ for $j \neq k$.}
\end{equation}
\end{defn}

For $k \in [N]$ and $s \in [d_k]$ define monomials $\hat{y}_{k;s}:= \frac{p_{k;s}}{p_{k;0}}\left(\frac{M_k^+}{M_k^-}\right)^s$ for $s \in [d_k]$. As a special case, we have $\hat{y}_{k;0} = 1$ for any $k \in [N]$. And when $d_k=1$,k $\hat{y}_{k;1}$ is the usual Fomin-Zelevinsky $\hat{y}_k$ Laurent monomial. 

The following lemma is proved by direct calculation, which we omit.
\begin{lem}
When performing $\mu_k$ to a non-normalized seed, the quantities $\hat{y}_{j;s}$ evolve by the rules 
\begin{align}\label{eq:nnyhatmutation}
\hat{y}_{k;s} &\mapsto \hat{y}_{k;d_k-s}\hat{y}_{k;d_k}^{-1}, \text{ and}\\
\hat{y}_{j;s} &\mapsto \hat{y}_{j;s}\hat{y}_{k;d_k}^{[sb_{kj}]_+}\left(\sum_{t = 0}^{d_k}\hat{y}_{k;t}\right)^{-sb_{kj}} \text{ for $j \neq k$.}
\end{align}
\end{lem}

Definition~\ref{defn:nnCSseed} includes Definition~\ref{defn:CSseed} as a special case. A CS-seed of the form 
$(\tilde{\bf x},\tilde{B},{\bf z})$ can be viewed as a non-normalized seed by restricting $\tilde{\bf x}$ and $\tilde{B}$ to the mutable variables and defining $p_{k;s} := z_{k;s}\prod_{i=N+1}^{N+\ell} \ov{x}_{i}^{b_{ki}}$. The coefficient group $\bbp$ is the abelian group of Laurent monomials in the both the frozen variables $\ov{x}_{N+1},\dots,\ov{x}_{N+\ell}$ and the coefficient string variables $z_{k;s}$. The mutation of $(\tilde{\bf x},\tilde{B},{\bf z})$ is an instance of the mutation rules in Definition~\ref{defn:nnCSseed}.

\begin{defn}\label{defn:seedorbit}
Let $\Sigma = ({\bf x}, B, {\bf p})$ and $\Sigma' = ({\bf x}', B', {\bf p}')$ be non-normalized CS- seeds. We write $\Sigma \sim \Sigma'$ if:  
i) $\frac{x_i}{x'_i} \in \bbp$ for all $i \in [N]$, and ii)  $\hat{y}_{k;s} = \hat{y}_{k;s}'$ for all $k \in [N]$ and $s \in [d_k]$. 
\end{defn}

As part of this definition, we are implicitly requiring that the exchange degrees $(d_k)_{k \in [N]}$ coincide in the two seeds.

\begin{lem}\label{lem:nnseedsandsim}
If non-normalized seeds $\Sigma,\Sigma'$ satisfy $\Sigma \sim \Sigma'$, then $\mu_k(\Sigma) \sim \mu_k(\Sigma')$ for any mutable $k \in [N]$.
\end{lem}

\begin{proof}
The equality $\hat{y}_{k;1} = \hat{y}'_{k;1}$  implies that the $k$th columns of $B$ and $B'$ agree; thus the $B$-matrices themselves agree. 
We check that condition i) is still satisfied after mutating cluster variables in direction $k$. Writing 
$$Z(M^+_k,M^-_k) = \left(M_k^-\right)^{d_k}\sum_s p_{k;s}\left(\frac{M^+_k}{M^-_k}\right)^s = \frac{1}{p_{k;0}}\left(M_k^-\right)^{d_k}\sum_s \hat{y}_{k;s},$$
and noting that the ratio $\frac{p_{k;0}'x_k'}{p_{k;0}x_k}\left(\frac{M_k^-}{(M^-_k)'}\right)^{d_k} \in \bbp$, we see that i) now follows. Condition ii) follows from the formulas \eqref{eq:nnyhatmutation}, which express mutation of $\hat{y}_{k;s}$'s variables purely in terms of these variables, the exchange matrix $B$, and the degree matrix $D$. 
\end{proof}

\begin{defn}
Consider non-normalized seeds $\Sigma$ and $\ov{\Sigma}$, possibly over different coefficient groups $\bbp$ and $\ov{\bbp}$. Let $\mcf_{>0}$ (resp. $\ov{\mcf_{>0}}$) denote the semfield of subtraction-free rational expressions in ${\bf x}$ (resp. $\ov{\bf x}$) with coefficients in $\bbp$ (resp. $\ov{\bbp}$). 

Let $\psi \colon \mcf_{>0} \to \ov{\mcf}_{>0}$ be a homomorphism of semifields satisfying $\psi(\bbp) \subset \ov{\bbp}$ and satisfying 
$\psi(\Sigma) \sim \ov{\Sigma}$. (In particular, it follows that $\psi(\Sigma)$ must be a seed.) Then $\psi$ is a {\sl quasi-homomorphism} of the generalized cluster algebras $\mca(\Sigma)$ and 
$\mca(\ov{\Sigma})$.
\end{defn}

The image of a non-normalized CS-seed pattern under such a semifield map will again be a non-normalized CS-seed pattern. It follows then from Lemma~\ref{lem:nnseedsandsim} that any quasi-homomorphism satisfies $\psi(\Sigma) \sim \ov{\Sigma}$ for {\sl all} seeds. Since every cluster variable $x \in \mca(\Sigma)$ is a subtraction-free expression in the initial cluster variables, we can in particular evaluate $\psi$ on any element of $\mca(\Sigma)$. The evaluation $\psi(x)$ will be a cluster variable of $\mca(\ov{\Sigma})$, perhaps multiplied by an element of $\ov{\bbp}$. Any quasi-homomorphism restricts in this way to a map of cluster algebras $\mca^\circ(\Sigma)$ and 
$\mca^\circ(\ov{\Sigma})$ (localized at their respective frozen variables). In nice cases, it further restricts to a map between the cluster algebras themselves.

\section{Bridge order}\label{secn:leapweak}
Recall the set $\Bound_n(k,\ell)$ of $n$-bounded, $\ell$-periodic, bounded affine permutations $f$ satisfying ${\rm av}(f) = k$. We discuss in this section a partial order on this set, the {\sl bridge order}, a ranked poset intermediate between $\leq_R$ and $\leq_B$. It is spiritually close to 
$(\Bound_n(k,\ell),\leq_R)$, but is better behaved because every maximal element has maximal rank. In the standard case of $\Bound_n(k,n)$, this partial order underlies the BCFW recursion for scattering amplitudes~\cite{AHBC}, and relatedly, of certain plabic graphs known as {\sl bridge graphs} \cite{Karpman,WilliamsBridge}. The material in this section is a convenient tool in inductive proofs, but is not needed to understand most theorem statements in the rest of the paper. 

We begin with a lemma concerning the Bruhat order rather than the bridge order. 
\begin{lem}\label{lem:orderideal}
For any $n \in \bbn$, $\Bound_n(k,\ell)$ is a finite order ideal in $(\tilde{S}^k_\ell,\leq_B)$. 
\end{lem}

\begin{proof}
This follows by appropriately modifying the proof of \cite[Lemma 3.6]{KLS}.
\end{proof}


\medskip

\begin{defn}[Bridge order]\label{defn:leapweak}
Let $f \lessdot ft_{i,j+s\ell}$ be a cover in $(\Bound_n(k,\ell),\leq_B)$ with $i < j+s\ell$. Then it is a {\sl bridge cover} if $f(a) \in \{a,a+n\}$ for each $a \in (i,j+s\ell)$. The {\sl bridge order} $(\Bound_n(k,\ell),\leq_{b})$ is the partial order on $\Bound_n(k,\ell)$ whose cover relations are the bridge covers.\footnote{Typographically, $\leq_B$ is distinct from $\leq_b$ because Bruhat is capitalized, but bridge is not.}\end{defn}

Unlike $\leq_R$ and $\leq_B$, the bridge order $\leq_b$ can only be defined when a value of $n$ has been specified. (That is: $\leq_b$ is not defined by restricting a partial order on $\tilde{S}^k_\ell$.) Like $\leq_R$ and $\leq_B$, the bridge order is graded by the length function (cover relations increase length by~1) and has unique minimal element~${\rm id}_k$.

\begin{rmk}\label{rmk:goodcontrol}
The definition of $\leq_b$ is natural for the following reason. 
Whenever $a \in [n]$ satisfies $f(a) \in \{a,a+n\}$, we have good control on what happens when we remove $a$ from the ground set $[n] \mapsto [n] \setminus \{a\}$. For example, if $f(a) = a$, then $\mcm_f \subset \binom{[n] \setminus a}k$. Geometrically, any $k\times n$ matrix representative for $X \in {\rm GGMS}(\mcm_f)$ will have the zero vector in columns $a,a+\ell,\dots,a+n-\ell$. On the other hand, if $f(a) = a+n$, then $I \in \mcm_f$ implies that $a \in I$. 
\end{rmk}

\begin{example}
The bounded affine permutation $[9,2,11,4] \in \Bound_{12}(4,4)$ is maximal in $\leq_R$ but is not maximal in $\leq_b$. Indeed, we have $[9,2,11,4] \leq_b [11,2,9,4] \leq_b [5,2,15,4]$, with the last of these elements maximal in $\leq_B$ (hence in $\leq_b$). 
\end{example}

\begin{defn}
For $s \leq t$ in bridge order on $\Bound_n(k,\ell)$, denote by 
${\rm Chains}(s,t)$ the set of saturated chains from $s$ to~$t$. We say that ${\bf f,f'} \in {\rm Chains}(s,t)$ differ by a 2-{\sl move} if they differ in a single element (thus, they differ by replacing $z \lessdot x \lessdot v$ by $z \lessdot y \lessdot v$ with $y \neq x$).  Similarly, a 3-{\sl move} replaces a portion
$z \lessdot x \lessdot x' \lessdot v$ by 
$z \lessdot y \lessdot y' \lessdot v$, with $x \neq y$ and $x' \neq y'$. 
\end{defn}

\begin{thm}\label{thm:titsmatsumoto}
Let $n = p\ell$. Any maximum chain chain in $(\Bound_n(k,\ell),\leq_b)$ has length $\frac{k(n-k)-\bb(p-\bb)}{p}$. If $f \in (\Bound_n(k,\ell),\leq_b)$ is a maximal element, and $g \leq f$, then any two chains in ${\rm Chains}(g,f)$ are related by a finite sequence of 2- or 3- moves. 
\end{thm}

We prove Theorem~\ref{thm:titsmatsumoto} in Section~\ref{subsecn:bridgeorder}.

In the case of $\Bound_n(k,n)$, this move-connectedness statement was proved in \cite[Theorem 5.1]{WilliamsBridge}, which describes the possible 2-dimensional faces of {\sl bridge polytopes}, whose edge graphs can be identified with the Hasse diagrams beneath maximal elements in $\Bound_n(k,n)$. The analogous bridge-order ideals for $\Bound_n(k,\ell)$ need not be the edge graph of a polytope. 

\begin{rmk}\label{rmk:lattice}
Let $\widetilde{\Bound}_{p\ell}(k,\ell) = \hat{1} \cup \Bound_{p\ell}(k,\ell)$ be the result of adding a maximal element $\hat{1}$ to $\Bound_{p\ell}(k,\ell)$, so that any pair of elements admit an upper bound. It appears to us that $\widetilde{\Bound}_{p\ell}(k,\ell)$ is a lattice, and that $\Bound_{p\ell}(k,\ell)$ is a meet semi-lattice. Our proof of Theorem~\ref{thm:titsmatsumoto} mimics the standard proof that any two reduced words for an element of a Coxeter group are connected by Coxeter moves, using the lattice property of the weak order. We do not need the lattice property for our intended application in 
Section~\ref{secn:clusters}, and consequently have not sorted out the details required to prove it.

The analogue of Theorem~\ref{thm:titsmatsumoto} fails for $\widetilde{\Bound}_{p\ell}(k,\ell)$. For example, the Hasse diagram of $\widetilde{\Bound}_{2k}(k,2)$ is a $2k+2$-gon (a union of two chains of length $k$ intersecting only in their bottom and top elements). 
\end{rmk}

It will be convenient to have the following explicit description of the maximal elements in $(\Bound_n(k,\ell),\leq_b)$, which are in fact maximal elements in $(\Bound_n(k,\ell),\leq_B)$.

\begin{defn}[Maximal elements]\label{defn:maximalelements}
Given $n = p\ell$, write by long division $k = \aa p+\bb$ with $\bb \in [0,p)$. If $p|k$ and $S \in \binom{[\ell]}\aa$, we define an affine permutation $t_S$ via 
$t_S(i) = i+n$ (resp. $t_S(i) = i$) when $i \mod \ell \in S$ (resp. when $i \mod \ell \notin S$). If $p$ does not divide $k$, $S \in \binom{[\ell]}\aa$ and $s \in [\ell] \setminus S$, then we define an an affine permutation $t_{S,s} \in \Bound_n(k,\ell)$
via $t_{S,s}(i) = i+\bb\ell$ when $i \equiv s \mod \ell$,  
$t_{S,s}(i) = i+n$ when $i\mod \ell \in S$, and $t(i) = i$ when $i\mod \ell \notin S \cup \{s\}$.
\end{defn}

We prove the following in Section~\ref{subsecn:bridgeorder}.

\begin{prop}\label{prop:maximalelts}
Let $n = p\ell$. The maximal elements in $\Bound_{n}(k,\ell)$ are the affine permutations $t_S$ (resp. $t_{S,s}$) in the case that $p | k$ (resp. $p$ does not divide $k$). 
\end{prop}

We end this section by recalling the concept of a {\sl bridge graph}, which we will use in our proof of Theorem~\ref{thm:TPtests}. Let $f_0 \in \Bound_n(k,n)$ be a maximal element and ${\bf f} = f_h \lessdot f_{h-1} \lessdot \cdots \lessdot f_0$ be a saturated chain in the bridge order on $\Bound_n(k,n)$. A construction in \cite{AHBC} associates to ${\bf f}$ a {\sl bridge graph} $G({\bf f})$, a reduced plabic graph whose trip permutation is $f_h$. For a careful description of this recipe see \cite[Section 2.5]{Karpman}. Informally, one starts with a plabic graph consisting entirely of edges connected to boundary vertices (``lollipops'') and then builds the graph by successively adding ``bridges'' which encode the cover relations $f_{i-1} \lessdot f_i$.  We illustrate the construction in an example in Figure~\ref{fig:BridgeGraphs}.

\section{Components of \texorpdfstring{$\Gr(k,n)^{\rho^\ell}$}{the cyclic shift locus}}\label{secn:components}
\subsection{Descriptions of $\ell$-fixed loci}
We give a few different descriptions of $\Gr(k,n)^{\rho^\ell}$, generalizing the description of the 1-shift locus given in Theorem~\ref{thm:Karpfinitelymany}.

\begin{prop}\label{prop:components}
For fixed $k,\ell,n$, let $\ell' = \gcd(\ell,n)$ and let $p = \frac{n}{\ell'}$. Then the cyclic symmetry locus is a disjoint union 
\begin{equation}\label{eq:components}
\Gr(k,n)^{\rho^\ell} = \coprod_{\substack{m_1,\dots,m_p \in [0,\dots,\ell'] \\ \sum m_i = k}} \Gr(m_1,\ell') \times \Gr(m_2,\ell') \times \cdots \times \Gr(m_p,\ell'). 
\end{equation}
\end{prop}

The data $(m_1,\dots,m_p)$ indexing the connected components is a weak composition of $k$ of length $p$ whose parts are bounded by $\ell'$. Equivalently, it is the data of a $k$-subset $1^{m_1}2^{m_2} \cdots p^{m_p}$ of the multiset $\{1^{\ell'},\dots,p^{\ell'}\}$. 

The argument for Proposition~\ref{prop:components} is simple, and has appeared previously.\footnote{One reference is Ben Webster's answer in \\ {\tt https://mathoverflow.net/questions/266274/fixed-points-of-an-involution}} It relies only on an understanding of the eigenspaces of 
$\rho \in \GL_n$. 
Fix an enumeration 
\begin{equation}\lam_j = 
\begin{cases}
\exp(2\pi \sqrt{-1}\frac{i}{n})  &\textnormal{ when $k$ is odd} \\
\exp(2\pi \sqrt{-1}\frac{2i+1}{n})&\textnormal{ when $k$ is even}
\end{cases}
\end{equation}
of the $n$th roots of $(-1)^{k-1}$, for $i = 0,\dots,n-1$.  
Then the eigenvalues of $\rho$ are the $\lam_i$ with corresponding eigenvector $\omega_i := (1,\lam_i,\dots,\lam_i^{n-1})$ 
The sequence of $\ell$th powers $\lam_0^\ell,\dots,\lam_{n-1}^\ell$ is $p$-periodic. 
Thus $\rho^\ell \in \GL_n$ has $\ell'$-dimensional eigenspaces $(E_i)_{i = 1,\dots,p}$ of the form $E_{i} = {\rm span}(\omega_{s} \colon \, s \equiv i \mod \ell)$. 

With these facts in hand, we prove Proposition~\ref{prop:components}. 
\begin{proof}
For a vector space $W$ let $\Gr(k,W)$ denote the Grassmannian of $k$-subspaces of $W$. Let $x \in \Gr(k,n)^{\rho^\ell}$. Then $x \subset \bbc^n$ is a $\rho^\ell$-invariant subspace, hence is spanned by $\rho^\ell$-eigenvectors. Let $m_i := \dim (x \cap E_{i}) \in [0,\ell']$. Clearly $\sum_i m_i = k$. Then 
$(x \cap E_{1},\dots,x\cap E_{p})$
determines a point in $\prod_{i=1}^p \Gr(m_i,E_i)$, and conversely any point $(W_1,\dots,W_p)$ in 
this product of Grassmannians determines a $\rho^\ell$-fixed point $W_1 \oplus \cdots \oplus W_p \in \Gr(k,n)^{\rho^\ell}$. The disjointness of this union is clear. 
\end{proof}

\begin{rmk}
Choosing a $k \times n$ matrix representative for $X \in \Gr(k,n)^{\rho^\ell}$ consisting of eigenvectors, we see that $\rho^\ell$ acts by $\prod_{i=1}^{p}\lam_j^{m_i\ell} \in \bbc$ on the component $\prod_i \Gr(m_i,\ell') \subset \Gr(k,n)^{\rho^\ell}$. 
\end{rmk}

\begin{example}
When $\ell = 1$ and $p=n$, each $m_i \in \{0,1\}$, and the components of \eqref{eq:components} are indexed by $k$-subsets of $[n]$. Since $\Gr(0,1) = \Gr(1,1) = {\rm pt}$, this is Karp's result Theorem~\ref{thm:Karpfinitelymany}. When $\ell=n$ then $p=1$ and  \eqref{eq:components} has a single term $m_1=k$, recovering $\Gr(k,n)^{\rho^n} = \Gr(k,n)$. In the case of $\Gr(2,4)^{\rho^2}$, we get
$$\Gr(2,2) \times \Gr(0,2) \, \cup \, \Gr(0,2) \times \Gr(2,2) \, \cup \, \Gr(1,2) \times \Gr(1,2) = {\rm pt} \cup{\rm pt} \, \cup \,  \bbp^1 \times \bbp^1, $$
maching the calculation we did with our bare hands in Example~\ref{eg:firstegs}. 
\end{example}

We now collect a few other perspectives on $\Gr(k,n)^{\rho^\ell}$ as a geometric space. 

Consider the unitary matrix $\frac{1}{\sqrt{n}}{\rm Vand}(\lam_0,\dots,\lam_{n-1}) \in \GL_n$ whose $j$th row is the eigenvector $\frac{1}{\sqrt{n}}\omega_j$. For odd $k$, multiplication by this matrix is the {\sl discrete Fourier transform} on $\bbc^n$. 
Let $T_p \subset \GL_n$ denote the rank $p$ algebraic torus consisting of $p$-periodic diagonal matrices ${\rm diag}(a_1,\dots,a_n)$ with $a_{i+p} = a_i \in \bbc^*$. 

\begin{prop}\label{prop:dft}
Right multiplication by ${\rm Vand}(z_0,\dots,z_{n-1})$ is an isomorphism of varieties $\Gr(k,n)^{T_p} \to \Gr(k,n)^{\rho^\ell}.$
\end{prop}

Thus, the cyclic symmetry locus differs from a locus of torus fixed points by a unitary transformation of the ambient space, implementing the passage from the standard basis of $\bbc^n$ to the $\rho$-eigenbasis. 
 
\begin{proof}
Abbreviate $V = {\rm Vand}(\lam_0,\dots,\lam_{n-1})$. We have 
$V \rho^\ell = {\rm diag}(\lam_0^\ell,\lam_1^\ell,\dots,\lam_{n-1}^\ell) V \in \GL_n$ with the diagonal matrix on the right hand side in $T_p$. If $x \in \Gr(k,n)^{T_p}$, it follows that 
$$x \, V \rho^\ell = x \, {\rm diag}(z_0^\ell,z_1^\ell,\dots,z_{n-1}^\ell) V= x \, V, $$
so that $x \, V\in \Gr(k,n)^{\rho^\ell}$. 

Conversely, if $x \in \Gr(k,n)^{\rho^\ell}$ then we can choose a basis $w_1,\dots,w_k$ consisting of $\rho^\ell$-eigenvectors. If $w_i$ has $\rho^\ell$-eigenvalue $z_j^\ell$, then $w_i V^{-1}$ is a linear combination of vectors of the form $\{e_{s}\}_{s \equiv  j \mod p}$. Therefore, the line through $w_i V^{-1}$ is $T_p$-stable, hence $x \, V^{-1} = {\rm span}\{w_1V^{-1},\dots,w_kV^{-1}\}$ is $T_p$-stable as claimed. 
\end{proof}
 
\begin{rmk}
In the special case that $\ell=1$ and $p = n$, the torus $T_p \subset \GL_n$ is the standard maximal torus $T$ of diagonal matrices. Combining Proposition~\ref{prop:dft} with Theorem~\ref{thm:Karpfinitelymany} we have that $\Gr(k,n)^T$ consists 
of $\binom n k$ many points. These $T$-fixed points are well known: they are the preimages of the vertices of the moment polytope for $\Gr(k,n)$. For each $I\in \binom {[n]} k$, there is a torus fixed point $x_I \in \Gr(k,n)$ whose unique nonzero Pl\"ucker coordinate is $\Delta_I$. By Proposition~\ref{prop:dft}, the $\rho$-fixed points in $\Gr(k,n)$ take the form  $x_I {\rm Vand}(\lam_0,\dots,\lam_{n-1})$ for $I \in \binom{[n]}k$, which is a way of phrasing 
\cite[Theorem 1.1]{Karp}.
\end{rmk}

\begin{rmk}[Hilbert functions multiply]
The inclusion $\prod_{i=1}^p \Gr(m_i,E_i) \hookrightarrow \Gr(k,n)^{\rho^\ell} \subset \bbp^{\binom n k-1}$ from \eqref{eq:components} is a composition of the following more familiar maps:
$$\prod_{i} \Gr(m_i,E_i) \hookrightarrow \prod_{i} \bbp (\bigwedge^{m_i}E_i) \hookrightarrow \bbp (\bigotimes_i \bigwedge^{m_i}E_i) \subset \bbp (\bigwedge^k (\bigoplus_i E_i)) \cong  \bbp(\bigwedge^k \bbc^n).$$
The first of these maps is the product of Pl\"ucker embeddings, the second is the Segre embedding, the third is inclusion defined by the vanishing of certain Pl\"ucker coordinates, and the last is a linear automorphism of $\bbp^{\binom n k -1}$ implementing the change of basis from the eigenbasis to the standard basis. 

From this description, we deduce the following multiplicativity of Hilbert functions $\dim (\bbc[\prod_i\Gr(m_i,\ell')]_{(d)}) =  \prod_i \dim \bbc[\Gr(m_i,\ell')]_{(d)}$. Here, the first homogeneous coordinate ring refers to the embedding $\prod_{i=1}^p \Gr(m_i,E_i) \subset \bbp^{\binom n k-1}$ from \eqref{eq:components}, and the 
second refers to the Pl\"ucker embedding $\Gr(m_i,E_i) \subset \bbp^{\binom {\ell'} {m_i}-1}$. The dimension of the latter counts the semistandard Young tableaux with $m_i$ rows, with $d$ columns, and with entries in~$[\ell']$.
\end{rmk}

\begin{rmk}[Each component is defined by linear equations]
Let us permute the $\rho^\ell$ eigenvectors $\omega_0,\dots,\omega_{n-1}$ so that the first $\ell$ vectors span the eigenspace $E_1$, the next $\ell$ span $E_2$, and so on. 

Let $X \in \prod_i \Gr(m_i,E_i) \subset \Gr(k,n)^{\rho^\ell}$ with a choice of basis $X = {\rm span}(w_1,\dots,w_k)$. Consider the $k \times n$ matrix whose $M$ whose $j$th row is the coordinate vector of $w_j$ with respect to the permuted eigenbasis. Then the initial $\ell'$ columns of $M$ will have rank $m_1$, the next $\ell'$ columns will have rank $m_2$, and so on. It is equivalent to require that the initial $a\ell'$ columns have rank $m_1 + \cdots m_a$ for $a = 1,\dots,p$, and also to specify that the final $b\ell$ columns have rank $m_{p-b+1}+ \cdots +m_p$ for $b = 1,\dots,p$. Imposing such rank conditions on the initial submatrices amounts to imposing a Schubert condition on $x$ with respect to the flag whose $i$th step is spanned by the first $i$ vectors in the permuted eigenbasis. Imposing rank conditions on the terminal submatrices amounts to imposing a Schubert condition on $X$ with respect to the opposite flag whose $i$th step is spanned by the final $i$ vectors in the permuted eigenbasis.  Thus, the component $\prod_i \Gr(m_i,E_i)$ is the intersection of a Schubert variety and an opposite Schubert variety, i.e. a Richardson variety in the Grassmannian. (Keeping in mind that this Richardson variety is computed with respect to the $\rho^\ell$-eigenbasis, not the standard basis for $\bbc^n$.) In particular, it is cut out from $\Gr(k,n)$ by the vanishing of certain linear equations in Pl\"ucker coordinates. 
\end{rmk}

\subsection{The distinguished component}
We denote by $\Gr(k,n)^{\rho^\ell}_{\geq 0} \subset \Gr(k,n)$ the subset of points which are both TNN and $\ell$-fixed. 

\begin{lemdef}\label{defn:distinguished} Amongst the components of $\Gr(k,n)^{\rho^\ell}$, there is a {\sl distinguished component} $\mcd := \mcd_n(k,\ell) \subset \Gr(k,n)^{\rho^\ell}$, defined by the containment $\Gr(k,n)^{\rho^\ell}_{\geq 0} \subset \mcd$. 

If $k = \aa p + \bb$ with $\bb \in [0,p)$, then $\mcd \cong \prod_{j=1}^{p-\bb}\Gr(\aa,\ell') \times \prod_{j=1}^{\bb}\Gr(\aa+1,\ell')$.

This component is top-dimensional in $\Gr(k,n)^{\rho^\ell}$, and its dimension is given by the formula $\frac{k(n-k)-\bb(p-\bb)}{p}$, i.e. the rank of the poset $\Bound_n(k,\ell')$ (cf.~Proposition~\ref{prop:maximalelts}).  
\end{lemdef}

The placement of parameters in the notation $\mcd_n(k,\ell)$ is intended to be parallel with the notation $\Bound_n(k,\ell)$. 

The isomorphism $\mcd \cong \prod_{j=1}^{p-\bb}\Gr(\aa,\ell') \times \prod_{j=1}^{\bb}\Gr(\aa+1,\ell')$ in the above definition/lemma says that the distinguished component is indexed by a $k$-multiset of $\{1^{\ell'},\dots,p^{\ell'}\}$ that is as ``equi-distributed as possible.'' 

In the special case $\ell=1$, the above definition/lemma  is the statement there is a unique TNN $\rho$--fixed point \cite[Theorem 1.1]{Karp}. (In fact, this point is TP.)

\begin{proof}
By \cite[Theorem 1.1]{Karp}, there is a unique TP $\rho$-fixed point $X_0 \in \Gr(k,n)^\rho_{>0}$. Then $X_0 \in \Gr(k,n)^{\rho^\ell}$, and we can define the component $\mcd$ by requiring that $X_0 \in \mcd$. We prove (independently) below that $\Gr(k,n)^{\rho^\ell}_{\geq 0}$ is homeomorphic to a closed ball, in particular is a connected space. So we have $\Gr(k,n)^{\rho^\ell}_{\geq 0} \subset \mcd$. 

Amongst the roots $\lam_0,\dots,\lam_{n-1}$ of $(-1)^{k-1}$, let 
$\lam_{i_1},\dots,\lam_{i_k}$ be the $k$ roots closest to $1 \in \bbc$ along the unit circle. By \cite[Theorem 1.1]{Karp}, the subspace  
$X_0$ is spanned by the corresponding eigenvectors $\omega_{i_1},\dots,\omega_{i_k}$. It is simple to see that in the enumeration of the $\lam_i$'s given above, the numbers $\{i_1,\dots,i_k\}$ form a cyclic $k$-interval inside $[n]$. Thus, the dimensions $m_i := \dim X_0 \cap E_i$ as in the proof of Proposition~\ref{prop:components} are as equi-distributed as possible. This identifies the numbers $m_1,\dots,m_p$ in the decomposition $\mcd \cong \prod_j \Gr(m_j,\ell')$.

Since $\dim \Gr(m_j,\ell') = m_j(\ell'-m_j)$, we have
$\dim \mcd_0 = \bb(\aa+1)(\ell'-\aa-1)+(p-\bb)\aa(\ell'-\aa) = k(\ell'-a)-\bb(k-\bb+p)$. Multiplying by $p$ and simplifying yields the claimed dimension formula. 

To see that the component is top-dimensional, we seek to minimize $\sum_i m_i^2$ subject to the constraint $\sum m_i = k$ and $m_i \geq 0$. The minimum is attained when the $m_i$'s are as equi-distributed as possible.\footnote{E.g., we claim that the minimum is attained when all $|m_i - m_j| \in \{0,1\}$. If not, one can replace $m_i \mapsto m_i-1$ and $m_j \mapsto m_j+1$ and decrease the value of $\sum m_i^2$. The remaining argument is a calculation.}
\end{proof}

The rest of the paper concerns TNN cells, total positivity tests, and clusters. Thus the distinguished component $\mcd_n(k,\ell)$, and not the entire cyclic symmetry locus $\Gr(k,n)^{\rho^\ell}$, should be considered the ``ambient variety'' for the constructions that follow. 

The following {\sl stability of the distinguished component} is important once we begin thinking about cluster structures. Fix $k$ and $\ell$ and set $n=p\ell$, letting $p$ vary. Then by the above Lemma/Definition, we have 
\begin{equation}\label{eq:stabilityofD}
\mcd \cong (\bbp^{\ell-1})^k \text{ for } p \geq k.
\end{equation}
However, this isomorphism is nontrivial when written in terms of Pl\"ucker coordinates on $\Gr(k,p\ell)$. 

\section{\texorpdfstring{Cell decomposition of $\Gr(k,n)^{\rho^\ell}_{\geq 0}$}{Cyclically symmetric cells}}\label{secn:cells}
We generalize the positroid cell decomposition of $\Gr(k,n)_{\geq 0}$ in the presence of cyclic symmetry, proving analogues of Postnikov's results.  

Denote by $\Gr(\mcm)^{\rho^\ell}_{>0}$ the set of 
$\rho^\ell$-fixed points in a positroid cell $\Gr(\mcm)_{>0}$.

It is clear from the definitions that whenever $\ell|n$, 
we have $ \Bound_n(k,\ell) \subset \Bound_n(k,n)$. Moreover, for a positroid $\mcm_f \subset \binom{[n]}k$, we have 
\begin{equation}\label{eq:periodicity}
f \in \Bound_n(k,\ell)\subset \Bound_n(k,n) \text{ if and only if } \mcm_f \text{ is  $\rho^\ell$-invariant.}
\end{equation}


Our main theorem in this section is the following. 
\begin{thm}\label{thm:cells}
Let $\ell' = \gcd(\ell,n)$. The space $\Gr(k,n)^{\rho^\ell}_{\geq 0}$ is homeomorphic to a closed ball, full-dimensional in the ambient variety $\Gr(k,n)^{\rho^\ell}$. 
It bears a cell decomposition 
\begin{equation}\label{eq:cells}
\Gr(k,n)_{\geq 0}^{\rho^\ell} = \coprod_{f \in \Bound_n(k,\ell') \subset \Bound_n(k,n)} \Gr(\mcm_f)_{>0}^{\rho^\ell}
\end{equation}
whose cell closure order is the {\sl dual of} Bruhat order on $\Bound_n(k,\ell')$ as ranked posets. 
\end{thm}
In particular, this implies that $\overline{\Gr(\mcm_f)_{>0}^{\rho^\ell}} = \Gr(\mcm_f)_{\geq 0}^{\rho^\ell}$. i.e. that taking $\rho^\ell$ fixed points commutes with taking closure.

The inclusion $\subseteq$ asserted in \eqref{eq:cells} follows from  \eqref{eq:periodicity}, and the reverse inclusion $\supseteq$ is trivial. The nontrivial statements in Theorem~\ref{thm:TPtests} are that 1) the TNN locus $\Gr(k,n)_{\geq 0}^{\rho^\ell}$ is a closed ball, 2)
if $\mcm$ is a $\rho^\ell$-invariant positroid, then $\mcm \cap \Gr(k,n)^{\rho^\ell}$ is a cell (in particular, it is nonempty) whose codimension is its $\tilde{S}_{\ell'}$-Coxeter length, and 3) the closure of each cell is a union of cells, and the closure relation is dual to 
$\tilde{S}_{\ell'}$-Bruhat order. We prove 1) using the techniques developed in \cite{GKLI} (with no modifications). We prove 2) by downward induction in the bridge order; the ideas are similar to those in a standard proof of Theorem~\ref{thm:cellsusual}. Our proof of 3) requires some constructions which we think have not appeared previously.

Assertion 2) is subtle: the analogous statement can fail for a realizable matroid that is not a positroid. 

\begin{example}[Symmetrical matroids need not have symmetrical points]
Consider the matroid $\mcm = \{12,23,34,14\}$. This matroid is $\rho^2$-invariant (indeed, it is $\rho$-invariant). It is a realizable matroid over $\bbr$ but is not a positroid. Points in the matroid stratum ${\rm GGMS}_{\bbr}(\mcm)$
have Pl\"ucker coordinates satisfying $0 = \Delta_{12}\Delta_{34}+\Delta_{14}\Delta_{23}$, and no other constraints. There are infinitely many such points, but {\sl none} of these points are $\rho^2$-invariant. Indeed, regardless of whether $\rho^2$ acts by $+ 1$ or $-1$, such points would have Pl\"ucker coordinates satisfying $0 = \Delta_{12}^2+\Delta_{23}^2$, and over $\bbr$ this implies that all Pl\"ucker coordinates are zero. Thus, ${\rm GGMS}_{\bbr}(\mcm) \cap \Gr(2,4)^{\rho^2}$ is empty. 
\end{example}

Theorem~\ref{thm:cells} and Lemma~\ref{lem:orderideal} together imply the following. 
\begin{cor}[Stability of TNN cells]\label{cor:cellsarestable}
Let $k \leq p \leq p'$. Then the cell closure orders on $\Gr(k,p\ell)^{\rho^\ell}_{\geq 0}$ and $\Gr(k,p'\ell)^{\rho^\ell}_{\geq 0}$ coincide. 
\end{cor}

This matches the corresponding stability of the variety $\mcd_n(k,\ell)$~\eqref{eq:stabilityofD}. 


\begin{example}[0-cells]\label{eg:maximalelts}
The maximal elements in $\Bound_n(k,\ell)$, described in Proposition~\ref{prop:maximalelts}, correspond to the 0-cells in $\Gr(k,n)^{\rho^\ell}_{\geq 0}$. When $p|k$ we have 0-cells $\Gr(\mcm_{t_S})^{\rho^\ell}_{>0}$ whose unique nonzero Pl\"ucker coordinate is the $\rho^\ell$-invariant subset $\{i \colon i \mod \ell \in S\}$. 
When $p$ does not divide $k$, we carry out this same construction to get a $(k -\bb) \times n $ matrix representative for the 0-cell 
$\Gr(\mcm_{t_{S}})^{\rho^\ell} \subset \Gr(k-\bb,n)^{\rho^\ell}$. We extend this to a $k \times n$ matrix representative for the cell $\Gr(\mcm_{t_{S,s}})_{>0}^{\rho^\ell} \subset \Gr(k,n)^{\rho^\ell}$ by appending a $\bb \times n$ matrix in the bottom $\bb$ rows. The matrix we append has a matrix representative for the unique point in $\Gr(\bb,p)^\rho_{>0}$ occupying columns 
$s,s+\ell,\dots,n-\ell+s$, and has zero vectors in all its other columns.  
\end{example}

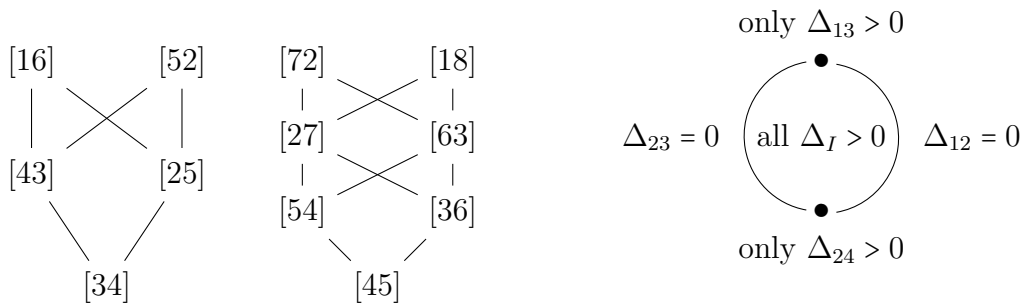
\begin{figure}
\begin{tikzpicture}
\node (v1) at (0.5,-2) {[34]};
\node (v2) at (-0.5,-0.5) {[43]};
\node (v3) at (1.5,-0.5) {[25]};
\node (v4) at (-0.5,1) {[16]};
\node (v5) at (1.5,1) {[52]};
\draw  (v1) edge (v2);
\draw  (v1) edge (v3);
\draw  (v2) edge (v4);
\draw  (v3) edge (v5);
\draw  (v2) edge (v5);
\draw  (v3) edge (v4);

\begin{scope}[xshift = 3.6cm]
\node (vv1) at (0.5,-2) {[45]};
\node (vv2) at (-0.5,-1) {[54]};
\node (vv3) at (1.5,-1) {[36]};
\node (vv4) at (-0.5,0) {[27]};
\node (vv5) at (1.5,0) {[63]};
\node (vv6) at (-0.5,1) {[72]};
\node (vv7) at (1.5,1) {[18]};
\draw  (vv1) edge (vv2);
\draw  (vv1) edge (vv3);
\draw  (vv2) edge (vv4);
\draw  (vv3) edge (vv5);
\draw  (vv2) edge (vv5);
\draw  (vv3) edge (vv4);
\draw  (vv7) edge (vv4);
\draw  (vv7) edge (vv5);
\draw  (vv6) edge (vv5);
\draw  (vv6) edge (vv4);
\end{scope}

\begin{scope}[xshift = 10cm]
\draw (-.2,1) arc (100:260:1);
\draw (.2,-1) arc (-80:80:1);
\node at (0,1) {\large $\bullet$};
\node at (0,-1) {\large $\bullet$};
\node at (0,0) {all $\Delta_I>0$};
\node at (2.0,0) {$\Delta_{12} = 0$};
\node at (-2.0,0) {$\Delta_{23} = 0$};
\node at (0,1.5) {only $\Delta_{13}>0$};
\node at (0,-1.5) {only $\Delta_{24}> 0$};
\end{scope}
\end{tikzpicture}
\caption{The posets $(\Bound_n(2,2),\leq_B)$ and $(\Bound_n(3,2),\leq_B)$. A similar picture (with $k+1$ levels) describes 
$(\Bound_{n}(k,2),\leq_B)$. As $k \to \infty$, the resulting poset is Bruhat order on $\tilde{S}_2$. On the right, we depict the cell decomposition of $\Gr(2,4)_{\geq 0}\overset{\rm homeo}{\cong} \mathbb{D}^2$, whose face poset is opposite to 
$(\Bound_4(2,2),\leq_B)$.}\label{fig:TNNposets}
\end{figure}

\begin{rmk}[TNN cells exhaust affine Bruhat order]
The subset ${\rm id}_k^{-1}\Bound_n(k,\ell) \subset \tilde{S}^0_n$ consists of those $w$ satisfying the condition $w(i) \in [i-k,i+n-k]$ for all $i \in \bbz$. It follows, that ${\rm id}_k^{-1}\Bound_{k\ell}(k,\ell) \subset {\rm id}_{k'}^{-1}\Bound_{k'\ell}(k',\ell)$ whenever $k \leq k'$, and that the ascending union  
$\cup_k {\rm id}_k^{-1}\Bound_{k\ell}(k,\ell)$ equals $\tilde{S}^0_\ell$ as a set. In the limit that $k \to \infty$, one recovers arbitrarily large order ideals in $\tilde{S}^0_\ell$ as (duals of) face posets of $\Gr(k,n)^{\rho^\ell}_{\geq 0}$. One does not see this behavior in the ordinary, i.e. $\Gr(k,n)_{\geq 0}$, setting. We would be interested to see a relation between $\Gr(k,n){\rho^\ell}_{\geq 0}$ and the affine loop group of type $\tilde{A}_{\ell-1}$, but this might be unnatural since our cellular spaces are indexed by {\sl dual} affine Bruhat order. 
\end{rmk}

\begin{example}[$\ell=2$ cell structure]\label{eg:2shiftegs} 
By Theorem~\ref{thm:Karpfinitelymany}, the TNN part of the 1-fixed locus is a point. The next simplest case to study is the TNN part of 2-fixed loci in $\Gr(k,n)$ when $n$ is even. By Grassmann duality we may assume $n \geq 2k$ so that by Corollary~\ref{cor:cellsarestable}, the topology only depends on $k$. 
The Bruhat order on $\Bound_{n}(k,2)$, as well as the cell decomposition of $\Gr(2,4)^{\rho^2}_{\geq 0}$, are depicted in Figure~\ref{fig:TNNposets}. The face poset of $\Gr(k,n)^{\rho^2}_{\geq 0}$ describes a prototypical regular CW structure on a $k$-dimensional closed ball $\mathbb{D}^k$. All $i$-cells are attached via the standard identifications $\partial \mathbb{D}^i \cong \mathbb{S}^{i-1}$, with the latter an $i-1$-sphere. For $i <k$, we attach two cells along their common boundary and the $i$-skeleton is an $i$-sphere. (This is the construction of spheres via iterated suspensions.)
When $i=k$, we attach a single cell to get $\mathbb{D}^k$.

The cells in $\Gr(k,2k)^{\rho^2}_{\geq 0}$ can be given the following uniform description. 
Let $v_1,\dots,v_n$ be the column vectors of a $k \times n$ matrix representative. Then we obtain an $i$-dimensional cell by requiring that 
${\rm rank}(v_j,\dots,v_{i+j-1}) = i-1$ for any odd $j \in [n]$ (with indices treated modulo $n$). The other $i$-cell is its cyclic shift, with the same rank condition imposed for even $j$. \end{example}

\begin{rmk}\label{rmk:countcells}
It would be interesting to find the cardinality (and the length-generating function) of $\Bound_n(k,\ell)$, generalizing this count for $\Bound_n(k,n)$~\cite{WilliamsCount}.  
\end{rmk}

\subsection{Homeomorphisms from bridge order}
We prepare some ingredients in the proof of Theorem~\ref{thm:cells}.

Define ${\rm sgn}(i,j) \in \pm 1$ to be $(-1)^{k-1}$ if $j < i$, and otherwise equal to $1$. This sign is a technicality in what follows, coming from columns wrapping around modulo~$n$ and the extra sign $(-1)^{k-1}$ in the definition of the cyclic shift. 
\begin{defn}
Let $E_{i,j}$ be the $n\times n$ matrix whose only nonzero entry is a $1$ in row $i$ and column $j$. For $a \in \bbc$, define
\begin{equation}
\eps_{ij}(a) = {\rm Id}+a \, {\rm sgn}(i,j)E_{i,j} \in \GL_n.
\end{equation}
As a special case we abbreviate $\eps_i(a) = \eps_{i,i+1}(a)$ for $i \in [n]$ (with indices taken modulo $n$). 
\end{defn}

The matrices $\eps_1(a),\dots,\eps_{n-1}(a)$ are the {\sl Chevalley generators} of the subgroup of upper triangular matrices in $\GL_n$. We will be most interested in these together with $\eps_n$, but because we work with $\leq_b$ rather than $\leq_R$, we sometimes need the $\eps_{ij}$. 

The matrix $\eps_{ij}(a)$ determines an automorphism of $\Gr(k,n)$. Suppose for concretness that $i<j$. When working with $k \times n$ matrix representatives, this automorphism amounts to a column operation. For example, if $v_1,\dots,v_n$ are column vectors describing $X \in \Gr(k,n)$, and $i<j$, then $X \cdot \eps_{ij}(a)$ is represented by the $k \times n$ matrix whose $j$th column is $v_j+av_i$ (the other columns are not affected). 

\begin{lem}\label{lem:homeomorphism}
Let $f \lessdot ft_{ij}$ be a cover in $(\Bound_n(k,n),\leq_b)$. Set $u = \#\{s \in (i,j) \colon \, f(s) = s+n\}$. Then there is a homeomorphism
\begin{align}
\Gr(ft_{ij})_{>0} \times \bbr_{>0} &\to \Gr(f)_{>0} \label{eq:homeomorphism}\\
(X',a)  &\mapsto X'\cdot \eps_{ij}((-1)^{u}a),
 \end{align}
with the subscripts $i,j$ of $\eps_{ij}$ treated modulo $n$. 

The inverse map is as follows. Let $\rmci_f = (\rightI_1,\dots,\rightI_n)$ be the Grassmann necklace for $f$, and suppose that $(X',a)$ map to $X \in \Gr(f)_{>0}$ under \eqref{eq:homeomorphism}. One recovers the parameter~$a$ as the ratio 
\begin{equation}\label{eq:aformula}
a= \frac{\Delta(\rightI_{j})(X)}{\Delta(\rightI_{j} \cup \{i\} \setminus \{j\})(X)} \in \bbr_{>0},
\end{equation}
and then recovers $X' = X \cdot \eps_{ij}((-1)^{u+1}a)$.  
\end{lem}

\begin{proof}
For weak order coverings $t_{i,i+1}$, the homeomorphism $(X',a) \to X$ just described corresponds to {\sl adding a bridge} (white at $i$ and black at $i+1$) in the lingo of \cite[Section 7]{LamCDM}. The inverse map is well-defined by \cite[Proposition 7.10]{LamCDM}, and the well-definedness of the forward map is implicit in \cite[Lemma 7.6 and Theorem 7.12]{LamCDM}. One can deduce the same results for coverings $t_{i,j}$ in the weak order by applying projections 
\cite[Lemmas 7.8 and 7.9]{LamCDM} until the bridge cover is a cover in the weak order (this weak order cover will take place in a smaller Grassmannian). These projections affect the bounded affine permutation and positroid in a very straightforward manner, which allows one to deduce that \eqref{eq:homeomorphism} is a well-defined homeomorphism in the larger Grassmnanian from the corresponding statement in the smaller Grassmannian. As a small wrinkle, let us remind 
that if $s$ satisfies $i < s < j$ and $f(s) = s$, then $v_s$ is a zero column in any matrix representative for $X$. And on the other hand, 
if $f(s) = s+n$, then $v_s$ participates in every nonzero $k \times k$ minor of $X$ . Since the $j$th column of $X'$ is $v_j+(-1)^uav_i$, we have that 
$\Delta(I)(X') = \Delta(I)(X)+a \Delta(I \cup i \setminus j )(X)$
for any $I \in \binom{[n]}k$. The sign $(-1)^u$ is absorbed when we swap the $v_i$ in the $j$th column of $X'$ past those vectors $v_s$ with $f(s) = s+n$.
\end{proof}

\begin{example}\label{eg:leapweakcover}
Consider the affine permutation $f = [7,6,5,10,9,8] \in \Bound_6(4,6)$. We have a covering $f \lessdot ft_{35}=:g$ in the bridge order, with $u=1$.
The Grassmann necklaces are
\begin{align*}
\rmci_{g} &= (1234,1234,1346,1346,1346,1346)   \\ 
\rmci_f &= (1234,1234,1346,1456,1456,1346).      
\end{align*}
Suppose $X \in \Gr(g)_{>0}$ is represented by column vectors 
$(v_1,v_2,v_3,v_4,0,v_6)$ whose only nonzero Pl\"ucker coordinates are $1234$ and $1346$.

For any $a \in \bbr_{>0}$, the matrix $X' = X  \eps_{35}(a)$ has columns $(v_1,v_2,v_3,v_4,(-1)^u av_3,v_6)$. Such a matrix is TNN: for example one has $\Delta(1456)(X')  = a\, \Delta(1346)(X) >0$. (Note that the sign $(-1)^u$ is necessary for this positivity to hold.) A simple check shows such a matrix has Grassmann necklace $\rmci_f$ when $a>0$, so that $X\eps_{35}(a) \in\Gr(f)_{>0}$ as Lemma~\ref{lem:homeomorphism} asserts. To see the inverse map, note that we can recover the parameter $a$ from the matrix 
$X' = (v_1,v_2,v_3,v_4,(-1)^u av_3,v_6)$ by taking the ratio of Pl\"ucker coordinates
$\frac{\Delta(1456)(X')}{\Delta(1346)(X')}$.
\end{example}

\subsection{The TNN part is a ball}
We prove the assertion 1) listed below the statement of Theorem~\ref{thm:cells}. That is, we show that $\Gr(k,n)^{\rho^\ell}_{\geq 0}$ is homeomorphic to a closed ball. The proof is identical to the one $\Gr(k,n)_{\geq 0}$ given in \cite{GKLI}. 
\begin{proof}
Let $f \colon \bbr \times \bbr^{k(n-k)} \to \bbr^{k(n-k)}$ be the contractive flow \cite[Equation 3.9]{GKLI}. Let $\phi \colon {\rm Mat}(k,n-k) \to \Gr(k,n)$ be the smooth embedding defined in \cite[Equation 3.3]{GKLI}; the image is the big Schubert cell taken with respect to ordered basis $\omega_0,\dots,\omega_{n-1}$ of eigenvectors for $\rho \in \GL_n$ from Section~\ref{secn:components}. By \cite[Proposition 3.4]{GKLI}, $\Gr(k,n)_{\geq 0} \subset \phi({\rm Mat}(k,n-k))$, and the resulting map $\phi^{-1} \colon \Gr(k,n)_{>0} \to {\rm Mat}(k,n-k)$ is also smooth. We set $Q:= \phi^{-1}(\Gr(k,n)^{\rho^\ell}_{>0})$, which is therefore a smooth embedded manifold of dimension  $\frac{k(n-k)-r(\bb-r)}{r}$ as proved above. We showed that ${\rm cl}(\Gr(k,n)^{\rho^\ell}_{>0})$ equals $\Gr(k,n)_{\geq 0}^{\rho^\ell}$ in Theorem~\ref{thm:cells}, and it follows that ${\rm cl}(Q) = \phi^{-1}(\Gr(k,n)^{\rho^\ell}_{\geq 0})$ and that ${\rm cl}(Q)$ is compact. By \cite[Corollary 3.8]{GKLI} the contractive flow $f(t,x)$ has the property that if $\phi(x)$ is TNN, then $\phi(f(t,x))$ is TP for $t >0$. On the other hand, since the contractive flow is defined using (the exponential of) $\rho$ and $\rho^{-1}$, it is clear that $\rho^\ell(\phi(f(t,x))) = f(t,\rho^\ell(\phi(x)))$ for any $t \in \bbr, x \in \Gr(k,n)$. It follows that the contractive flow preserves $\phi^{-1}(\Gr(k,n)^{\rho^\ell})$, and in particular maps ${\rm cl}(Q)$ into itself. So all hypotheses of \cite[Lemma 2.3]{GKLI} hold. We have that ${\rm cl}(Q)$, thus $\Gr(k,n)^{\rho^\ell}_{\geq 0}$, is homeomorphic to a closed ball. 
\end{proof}

\subsection{The strata are indeed cells}
Now we prove the assertion (2) outlined below Theorem~\ref{thm:cells}. That is, for $f \in \Bound_n(k,\ell') \subset \Bound_n(k,n)$, we show that $\Gr(\mcm)_{>0}^{\rho^\ell}$ is a cell of specific dimension.

\begin{proof}
Proposition~\ref{prop:maximalelts} describes the maximal elements in $(\Bound_n(k,\ell),\leq_b)$. When $p|k$, the maximal elements $f$ have the property that $\Gr(\mcm_f)_{>0}$ is already a point, and this point is $\rho^\ell$-fixed. When $p$ does not divide $k$, projecting away the 
columns corresponding values of $f$ for which $f(i) \in \{i,i+n\}$, we get a homeomorphism from $\Gr(\mcm_f)^{\rho^\ell}_{>0} \subset \Gr(k,n)$ to the space  $\Gr(\bb,r)^{\rho}_{> 0}$, which is a point by Theorem~\ref{thm:cellsusual}. (We have justified now the description of 0-cells given in Example~\ref{eg:maximalelts}, which was stated without proof.)

By the length formula for maximal elements given in Proposition~\ref{prop:maximalelts}, we see that the claimed dimension formula \eqref{eq:cells} holds for these maximal~$f$. It holds then for arbitrary $f \in \Bound_n(k,\ell)$ by downwards induction in $\leq_b$ using the homeomorphsms \eqref{eq:homeomorphism}.
\end{proof}

\subsection{Cell closure order is Bruhat order}\label{secn:Bruhat}
Now we prove assertion 3) stated below Theorem~\ref{thm:cells}. Standard proofs that the closure partial order on cells is the Bruhat order \cite{LamCDM,Postnikov} emply the boundary measurement map and plabic graphs. We do not currently see how to generalize this argument, so we give an alternative, more direct argument. Thus, we start by giving a direct proof of Theorem~\ref{thm:cellsusual}, which immediately generalizes to the cyclically symmetric setting.

\begin{lem}\label{lem:warmup}
Let $f \lessdot g$ in $(\Bound_n(k,n),\leq_B)$. 
Then $\Gr(g)_{>0} \subset \Gr(f)_{>0}$.
\end{lem}

As a bit of terminology needed in the proof, let $\varphi(a) \in \bbc(a)$ be a rational function in the variable $a$. One can uniquely express express $\varphi(a) = a^m \frac{P(a)}{Q(a)}$ for an integer $m \in \bbz$ and polynomials $P(a),Q(a) \in \bbc[a]$ with nonzero constant term. Then we define the {\sl order} of $\varphi(a)$ to be the integer $m$.


\begin{proof}
Let $X \in \Gr(g)_{>0}$ be given. We construct a family of points $Y(a) \in \Gr{f}_{>0}$ for $a \in \bbr_{>0}$, satisfying $\lim_{a \to 0}Y(a) = X$.

Since $f \lessdot g$, we have $f = gt_{ij}$ for some values $i< j \in \bbz$ satisfying 
\begin{equation}\label{eq:bruhatcover}
i < g(j) < g(i) \leq i+n \text{ and } \{g(a) \colon a \in (i,j)\} \cap [g(j),g(i)] = \emptyset.
\end{equation}
To simplify notation, let us assume that $i=1$. 

Let $S = \{g(t) \colon \, t \in [j], \, g(t) \leq g(j)\} $ and $B = \{g(t) \colon \, t \in [j] \, g(t) \geq g(1)\} $. ($S$ and $B$ are meant to stand for ``small'' and ``big'' respectively.) Using \eqref{eq:bruhatcover} we have $\{g(1),\dots,g(j)\} = S \coprod B$.  

Multiplying $g$ by the reflection $t_{ij}$ has the effect of swapping values 
$g(1),g(j)$. Such a reflection is a composition of simple transpositions. Specifically, we can compute the reflection as a composition of simple transpositions of the following two types: 
\begin{enumerate}
\item Length-decreasing swaps of the form $bs \mapsto sb$ with $s \in S, b \in B$. 
\item Length-increasing swaps of the form $sb \mapsto bs$ with $s \in S, b \in B$, provided the current location of $b$ is strictly right of the starting location of $b$. 
\end{enumerate}
In other words, when we perform swaps of type (2), the element $b$ that is participating has already swapped past (possibly several) elements of $S$. Note that we do not allow swaps involving two elements of $S$ or $B$. 

A greedy argument shows that swapping values $g(1),g(j)$ can be realized as a sequence of swaps of the two above types $(1)$ and $(2)$ only. Let $s_{i_1},\dots,s_{i_\ell}$ be such a list of simple transpositions, so that $f = gt_{ij} = g s_{i_1} \cdots s_{i_\ell}$. Each partial product $gs_{i_1} \cdots s_{i_t}$ determines an element of $\Bound_n(k,n)$, hence a Grassmann necklace 
$\rmci(t) = (\rightI_1(t),\dots,\rightI_n(t))$. 

Starting with $X:= X_{0,a} \in \Gr(g)_{>0}$, we inductively define elements $X_{t+1,a} \in \Gr(g s_{i_1} \cdots s_{i_t})_{>0}$ by the action of Chevalley generators:
\begin{equation}\label{eq:bruhatrecursion}
X_{t+1,a} =
\begin{cases}
X_{t,a} \eps_{i_{t}}(a) & \text{ for swaps of type (1)} \\
X_{t,a} \eps_{i_{t}}\left(\frac{-\Delta(\rightI_{i_t+1}(t))(X_{t,a})}{\Delta(\rightI_{i_t+1}\cup \{i_t\} \setminus \{i_t+1\})(X_{t,a})}\right) & \text{ for swaps of type (2)}.
\end{cases} 
\end{equation}

In other words, we use the forward homeomorphism \eqref{eq:homeomorphism} with the chosen value of $a$ for length-decreasing swaps, and use the inverse map \eqref{eq:aformula} for length-increasing swaps. Thus, $X_{t+1,a} \in \Gr(g s_{i_1}\cdots s_{i_t})_{>0}$ for any $t \in [0,\ell]$, and for any $a>0$.

Since $X_{t,a}$ is obtained from our initial point $X$ by performing column operations, both the numerator and denominator of the ratio of Pl\"ucker coordinates in \eqref{eq:bruhatrecursion} can be expressed in terms of $a$ and the Pl\"ucker coordinates of $X$. Thus, viewing the Pl\"ucker coordinates of $X$ as constants, this ratio is a rational function of $a$. Our key claim is that each time we perform a type (2) swap, this rational function has order $1$ in $a$. Assuming  this key claim, we see that $\lim_{a \to 0} X_{t,a} = X$  for all $t$. Applying this when $t = \ell$ and setting $Y(a) = X_{\ell,a} \in \Gr(f)_{>0}$ we conclude that $\lim_{a \to 0}Y(a) = X$, as desired. 

The key claim follows by a slightly more refined analysis. Let $v_1,\dots,v_n$ be the columns of a $k\times n$ matrix representing $X$. Each time we perform a swap of type (1) or (2), we add a scalar multiple of column $i_t$ to column $i_t+1$. For swaps of type (1), this scalar multiple equals $a$, and for swaps of type (2), we claim inductively that this scalar is a rational function of order one in $a$. 
By induction, we can assume that the $s$th column of $X_t$ is a linear combination of the columns $v_1,\dots,v_s$ of $X$, that the coefficients of this linear combination rational functions in $a$, and that the coefficient of 
$v_{s'}$ in column $s$ has order $s-s'$. For example, if we perform swaps of type (1) in columns 1, then 2, then 3, then the 4th column looks like $v_4+av_{3}+a^2v_{2}+a^3v_1$. 

Suppose we use multilinearity to expand the numerator of \eqref{eq:bruhatrecursion} as $\sum_{I'} c_{I'}\Delta(I')(X)$ and the denominator as $\sum_{I''} c_{I''}\Delta(I'')(X)$ with $c_{I'},c_{I''} \in \bbc(a)$. Since column $i_{t}+1$ expands in terms of columns $1,\dots,i_{t}+1$, the $I''$ which appear in the denominator coincide with those in the numerator, with the exception that those $I'$ for whom $i_t+1 \in I'$ do not appear in the denominator. For the common terms, by the homogeneity statement in the previous paragraph, the order of $a$ dividing $c_{I''}$ is one less than the power dividing $c_{I'}$. So we need to show that the extra $I'$ terms that appear in the numerator and do not appear in the numerator vanish on $X$. 

The argument for this vanishing is as follows. We will prove that each such term is $<_{i_t+1} \rightI_{i_t+1}(t)$, from which the claim follows. 
We make two more observations. First, multiplying by Chevalley generators as in \eqref{eq:bruhatrecursion} does not change the span of the initial columns, so that $\rightI_1(X_{t,a}) = \rightI_1(X)$ for all $t,a$. Second, letting $g_t$ denote the partial product $gs_{i_1} \cdots s_{i_t}$, then we can compute 
$\rightI_{i_t+1}(t)$ as $\rightI_1(X) \setminus \{1,\dots,i_t\} \cup \{g_t(1),\dots,g_t(t)\}$. From the way the swaps (1),(2) are defined, the elements $\{g_t(1),\dots,g_t(i_t)\}$ are lexicographically smaller than the elements $\{g(1),\dots,g(i_t)\}$, so that $\rightI_{i_t+1}(t)<{i_{t+1}} \rightI_{i_t+1}(X)$, hence vanishes on $X$.  
\end{proof}

\begin{proof}[Proof of (2)]
One of the two containments follows softly from the corresponding statement for $\Gr(k,n)_{\geq 0}$: if $X \in \Gr(g)_{>0}^{\rho^\ell}$ is in the closure of $\Gr(f)_{>0}^{\rho^\ell}$, then $X$ is in the closure of $\Gr(f)_{>0}$, and thus $f \leq g$ in $(\Bound_n(k,n),\leq_B)$, hence in $(\Bound_n(k,\ell),\leq_B)$.

For the other containment, note that provided $\ell >1$, the column operations in the proof Lemma~\ref{lem:warmup} can be done $\rho^\ell$-equivariantly. 

\end{proof}

\section{TP tests}\label{secn:TPtests}
In the usual setting, i.e. for $\Gr(k,n)_{>0}$, one naively expects that verifying that a given $X \in \Gr(k,n)$ is TP requires checking $\binom n k $ many inequalities. In fact, the positivity of $k(n-k)+1$ judiciously chosen Pl\"ucker coordinates implies the positivity of {\sl all} Pl\"ucker coordinates. The ``magic number'' of Pl\"ucker coordinates required in such a TP test is $1 + \dim \Gr(k,n)_{>0} =1+ \dim \Gr(k,n)$. (The first is a cell and the second is a variety.)

We prove analogous statements in this section, with  the distinguished component $\mcd_n(k,\ell) \subset \Gr(k,n)^{\rho^\ell}$, rather than $\Gr(k,n)$, playing the role of the ambient variety. 


Restricting attention to the cyclic symmetry locus implies equalities amongst Pl\"ucker coordinates~\eqref{eq:character}. Consequently, one expects that a minimal TP test for $\mcd_n(k,\ell)$ should be even smaller than a minimal TP test for $\Gr(k,n)$. We confirm this expectation in this section and investigate the algebraic relationships between our TP tests (in the language of cluster algebras) in Section~\ref{secn:clusters}.


Recall our notation $\Gr(\mcm_f) \subset \Gr(k,n)$ for the positroid variety to $f \in\Bound_n(k,n)$. When $f \in  \Bound_n(k,\ell) \subset  \Bound_n(k,n)$, we will have a nonempty subset of $\rho^\ell$ fixed points $\Gr(\mcm_f)^{\rho^\ell}$.

\begin{defn}\label{defn:TPtests}
Let $f \in \Bound_n(k,\ell)$. Functions $\{\varphi_1,\dots,\varphi_{t}\} \subset \bbc[\Gr(k,n)]$ are a {\sl TP test} for $\Gr(\mcm_f)^{\rho^\ell}$
if the following holds: for $X \in \mcd_n(k,\ell) \cap \Gr(\mcm_f)$, we have $X \in \Gr(\mcm)_{>0}$ if and only if $\varphi_i(X) \in\bbr_{>0}$ for $i \in [t]$. 
The test is {\sl efficient} if $t = \dim \Gr(\mcm_f)_{>0}^{\rho^\ell}+1$.
\end{defn}

The requirement $\varphi_i(X)>0$ should be interpreted as saying that all of the numbers $\varphi_i(X)$ after simultaneously rescaling by an appropriate complex number. (Or equivalently, that $X$ bears a $k \times n$ matrix representative for whom each of the $\varphi_i$ evaluate positively.) 

This definition is the most salient when $f = {\rm id}_k \in \Bound_n(k,\ell)$. In this case we have $\mcd \cap \Gr(\mcm_f) = \mcd$, and we say that $\varphi_1,\dots,\varphi_t$ is a TP {\sl test for  $\mcd$.} 


\begin{rmk} We require $X \in \mcd_n(k,\ell)$ as part of our definition, rather than requiring merely that $X \in \Gr(k,n)^{\rho^\ell}$. This is a natural imposition: if $X$ is not in the distinguished component, than it is certainly not TP. This is also consistent with cluster algebras philosophy: a cluster structure on a variety provides the variety with a notion of TP part. But any cluster variety is irreducible. 
So the cluster structure, hence the notion of TP part, should be associated to $\mcd_n(k,\ell)$, not to $\Gr(k,n)^{\rho^\ell}.$

As further motivation, if we hold on to the philosophy that the size of an efficient TP test should exceed the dimension of the ambient variety by one,  then an efficient TP test for $\Gr(k,n)^{\rho}$ would consist of a single Pl\"ucker coordinate. But a single Pl\"ucker coordinate is certainly not able to detect that a given point $X \in \Gr(k,n)^{\rho}$ is the unique TP $\rho$-fixed point, and not one of the other $\binom n k -1$ points in $\Gr(k,n)^{\rho}$. 
\end{rmk}

We will demonstrate the existence of efficient TP tests by constructing certain highly symmetrical TP tests for $\Gr(k,n)_{>0}$.

The following terminology is useful both here and in Section~\ref{secn:clusters}.

\begin{defn}\label{defn:optimal}
An (extended) cluster variable $x \in \bbc[\Gr(k,n)]$ is an {\sl $\ell$-cluster variable} if its orbit $\{\rho^{a\ell}(x) \colon a \in [p]\}$ is contained in a cluster for $\Gr(k,n)$. An {\sl $\ell$-optimal cluster} is a $\rho^\ell$-invariant subset of an extended cluster in $\bbc[\Gr(k,n)]$  which is an efficient TP test for $\mcd$. We reserve the terminology $\ell$-{\sl cluster} for those $\ell$-optimal clusters which are moreover extended clusters for $\Gr(k,n)$. An {\sl $\ell$-cluster monomial} is a monomial in the variables of any $\ell$-optimal cluster. 
A collection $\mcc \subset  \binom{[n]}k$is an $\ell$-{\sl optimal collection}  if $\Delta(\mcc)$ is an $\ell$-optimal cluster. 
\end{defn}

Any two functions in a $\rho^\ell$-orbit determine the same element of $\bbc[\mcd_n(k,\ell)]$. Therefore, the size of the TP test from an $\ell$-optimal cluster, i.e., the number $t$ from Definition~\ref{defn:TPtests}, is the number of  $\rho^\ell$-orbits, not the number of variables.

Depending on the parameters $k,n,\ell$, it can happen that 
$\bbc[\Gr(k,n)]$ admits no $\ell$-clusters. 
As a small example, one knows that clusters in $\Gr(2,8)$ correspond to triangulations of an octagon, and it is easy to see that the octagon admits no $\rho^2$-invariant triangulations, hence no $2$-clusters. Necessary and sufficient conditions for the existence of $k$-clusters in $\Gr(k,n)$ were given in \cite{PTZ}. The conditions depend on the value of $k \mod p$. On the other hand, we have the following.

\begin{thm}\label{thm:TPtests}
The Grassmannian $\Gr(k,n)$ admits $\ell$-optimal collections (for any $\ell$). 

Let $\ell' = \gcd(\ell,n)$ and let ${\bf f} = f_h \lessdot f_{h-1} \lessdot \cdots \lessdot f_0$ be a saturated chain ending at a maximal element $f_0 \in (\Bound_n(k,\ell'),\leq_b)$. Then more explicitly, the union of Grassmann necklaces
\begin{equation}\label{eq:chaintocup}
\mcc({\bf f}) = \displaystyle \cup_{i=0}^h \rmci_{f_i} \subset \binom{[n]}k
\end{equation}
is a $\rho^\ell$-invariant weakly separated collection and an efficient TP test for $\Gr(f)^{\rho^\ell}$. 

If ${\bf f,f'} \in {\rm Chains}(f_h,f_0)$ are two such chains, then the weakly separated collections $\mcc({\bf f})$ and $\mcc({\bf f'})$ are related by a finite sequence of $\rho^\ell$-symmetrical square moves. 
\end{thm}

When ${\bf f}$ is a {\sl maximal} chain in $(\Bound_n(k,\ell),\leq_b)$, the construction \eqref{eq:chaintocup} yields an efficient TP test for $\mcd_n(k,\ell)$.
 
\begin{rmk}
Taking $\ell=n$, the weakly separated collections of the form \eqref{eq:chaintocup} are the sets of face labels associated to 
{\sl bridge graphs}. Not every reduced plabic graph admits a bridge decomposition, so we obtain in this way a proper susbet of the set of weakly separated collections. In the same way, when $\ell<n$, the construction \eqref{eq:chaintocup} produces those $\ell$-optimal collections which admit ``$\rho^\ell$-invariant bridge decompositions,'' and in general not every $\ell$-optimal collection has this form.  
\end{rmk} 

\begin{example}\label{eg:integraldomain}
A 2-optimal collection in $\Gr(2,8)$ consists of the of the 8 frozen Pl\"ucker coordinates together with the interior arcs $\{13,35,57,17\}$. We extend this partial triangulation to a full triangulation by adding either $15$ or $37$, but neither of these choices yields a $\rho^2$-invariant triangulation. 

In the quotient algebra $\bbc[\Gr(2,8)]^{\rho^2}$, which is not an integral domain, we have the three-term Pl\"ucker relation 
$$0 = \Delta_{15}\Delta_{37}-\Delta_{13}\Delta_{57}-\Delta_{17}\Delta_{35} =  \Delta_{15}^2 - 2 \Delta_{13}^2 = (\Delta_{15}-\sqrt{2}\Delta_{13})(\Delta_{15}+\sqrt{2}\Delta_{13}).$$
In the further quotient $\bbc[\mcd_8(2,2)]$, 
which {\sl is} an integral domain, one of the two factors on the right must vanish. Total positivity considerations imply that $\Delta_{15}-\sqrt{2}\Delta_{13} = 0\in \bbc[\mcd_8(2,2)]$. So the ``extra'' variable  $\Delta_{15} = \sqrt{2} \Delta_{13}$ is superfluous to any total positivity test. The proof of Theorem~\ref{thm:TPtests} proceeds along similar lines. \end{example}

\begin{example}\label{eg:domainaschain}
To illustrate \eqref{eq:chaintocup}, consider the element $[5,2] \in (\Bound_8(2,2),\leq_b)$. It is contained in a unique maximal chain 
${\bf f}:= [34] \overset{s_0}{\lessdot} [25] \overset{s_1}{\lessdot} [52]\in (\Bound_8(2,2),\leq_b)$ (cf.~Remark~\ref{rmk:lattice} and Figure~\ref{fig:TNNposets}). We have $\rho^2$-equivariant Grassmann necklaces: 
\begin{align*}
\rmci_{[52]} &= (13,35,35,57,57,17,17,13) \hspace{.75cm} \rmci_{[25]} = (13,23,35,45,57,67,17,18) \\
\rmci_{[34]} &= (12,23,34,45,56,67,78,18).
\end{align*}
The set $\Delta(\rmci_{[52]}) = \{\Delta_{13}\}\subset \bbc[\mcd_8(2,2)]$ consists of a single Pl\"ucker coordinate, and is an efficient TP test for the 0-cell $\Gr(\mcm_{[52]})_{>0}^{\rho^2}$. This 0-cell is the North pole 
in Figure~\ref{fig:TNNposets}. 
The union $\Delta(\rmci_{[52]} \cup \rmci_{[25]}) = \{\Delta_{13},\Delta_{23}\} \subset \bbc[\mcd]$ is an efficient TP test for the right 1-cell in Figure~\ref{fig:TNNposets}.   The union $\Delta(\rmci_{[52]} \cup \rmci_{[25]} \cup \rmci_{[34]}) = \{\Delta_{13},\Delta_{23},\Delta_{12}\} \subset \bbc[\mcd]$ is an efficient TP test for the 2-cell $\Gr(\mcm_{[34]})^{\rho^2}_{>0} = \Gr(2,8)^{\rho^2}_{>0}$, as we have argued directly in Example~\ref{eg:integraldomain}.  
\end{example}

\begin{example}\label{eg:chaintobridgegraph}
Continuing Example~\ref{eg:domainaschain}, note that $[5,2] \in {\rm B}_8(2,2)$ is no longer a maximal element once viewed in $({\rm B}_8(2,8),\leq_b)$. We choose arbitrarily the following chain ${\bf f'} \subset ({\rm B}_8(2,8),\leq_b)$ from $[5,2]$ to a maximal element: 
\begin{align}\label{eq:choiceofchain}
{\bf f'}:= [5,2,7,4,9,6,11,8] &\overset{t_{35}}{<} [5,2,9,4,7,6,11,8] \overset{t_{13}}{<} [9,2,5,4,7,6,11,8] \\&\overset{t_{35}}{<} [9,2,7,4,5,6,11,8]  \overset{t_{37}}{<} [9,2,11,4,5,6,7,8].
\end{align}
In Figure~\ref{fig:BridgeGraphs}, we draw the bridge graph encoded by ${\bf f \cup f''}$ where ${\bf f}$ is the chain from Example~\ref{eg:domainaschain}, viewed inside $({\rm B}_8(2,8),\leq_b)$. (Each ${\rm B}_8(2,2)$-cover can be implemented as 4 ${\rm B}_8(2,8)$-covers.) Its cluster is not $\rho^2$-invariant, but it becomes the 2-optimal collection from Example~\ref{eg:integraldomain} once we delete the non-symmetrical variable $\Delta_{15}$. If we were to change the chain ${\bf f'}$, then the five inner faces of the plabic graph in Figure~\ref{fig:BridgeGraphs} might change, but the variables $13$, $35$, $57$, and $17$ will always be present. So the choice of ${\bf f'}$ does  not matter. 
\end{example}

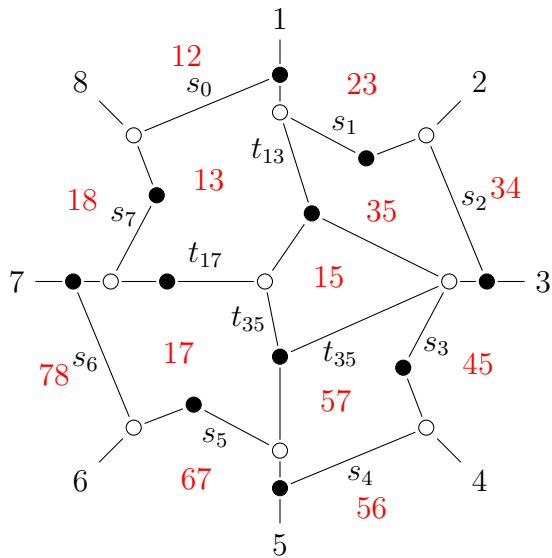
\begin{figure}
\begin{tikzpicture}
\node (v1) at (90:3.5cm) {1};
\node (v8) at (135:3.75cm) {8};
\node (v7) at (180:3.5cm) {7};
\node (v6) at (225:3.75cm) {6};
\node (v5) at (270:3.5cm) {5};
\node (v4) at (315:3.75cm) {4};
\node (v3) at (0:3.5cm) {3};
\node (v2) at (45:3.75cm) {2};
\node (a1) at (90:2.75cm) {};
\node (a8) at (135:2.75cm) {};
\node (a7) at (180:2.75cm) {};
\node (a6) at (225:2.75cm) {};
\node (a5) at (270:2.75cm) {};
\node (a4) at (315:2.75cm) {};
\node (a3) at (0:2.75cm) {};
\node (a2) at (45:2.75cm) {};
\node (b1) at (90:2.25cm) {};
\node (b8) at (145:2.0cm) {};
\node (b7) at (180:2.25cm) {};
\node (b6) at (235:2.0cm) {};
\node (b5) at (270:2.25cm) {};
\node (b4) at (325:2.0cm) {};
\node (b3) at (0:2.25cm) {};
\node (b2) at (55:2.0cm) {};

\node (d1) at (180:.2cm) {};
\node (c7) at (180:1.5cm) {};
\node (c5) at (270:1.0cm) {};
\node (c3) at (65:1.0cm) {};

\node at (112.5:2.8cm) {$s_0$};
\node at (112.5+90:2.8cm) {$s_6$};
\node at (112.5+180:2.8cm) {$s_4$};
\node at (112.5+270:2.8cm) {$s_2$};
\node at (157.5:2.25cm) {$s_7$};
\node at (157.5+90:2.25cm) {$s_5$};
\node at (157.5+180:2.25cm) {$s_3$};
\node at (157.5+270:2.25cm) {$s_1$};
\node at (-50:1.25cm) {$t_{35}$};
\node at (230:0.65cm) {$t_{35}$};
\node at (95:1.75cm) {$t_{13}$};
\node at (160:1.05cm) {$t_{17}$};
\node at (5:.65cm) {\textcolor{red}{15}};
\node at (215:1.65cm) {\textcolor{red}{17}};
\node at (215-90:1.65cm) {\textcolor{red}{13}};
\node at (215-180:1.65cm) {\textcolor{red}{35}};
\node at (215-280:1.75cm) {\textcolor{red}{57}};
\node at (-22.5:2.85cm) {\textcolor{red}{45}};
\node at (-22.5-90.5:2.85cm) {\textcolor{red}{67}};
\node at (-22.5-180:2.85cm) {\textcolor{red}{18}};
\node at (-22.5-270:2.85cm) {\textcolor{red}{23}};
\node at (22.5:3.25cm) {\textcolor{red}{34}};
\node at (22.5-90.5:3.25cm) {\textcolor{red}{56}};
\node at (22.5-180:3.25cm) {\textcolor{red}{78}};
\node at (22.5-270:3.25cm) {\textcolor{red}{12}};

\draw (v1)--(a1)--(a8)--(v8);
\draw (v3)--(a3)--(a2)--(v2);
\draw (v5)--(a5)--(a4)--(v4);
\draw (v7)--(a7)--(a6)--(v6);
\draw (a1)--(b1)--(b2)--(a2);
\draw (a3)--(b3)--(b4)--(a4);
\draw (a5)--(b5)--(b6)--(a6);
\draw (a7)--(b7)--(b8)--(a8);
\draw (b7)--(c7)--(d1)--(c3)--(b1);
\draw (c3)--(b3)--(c5)--(b5);
\draw (c5)--(d1);

\draw [fill= black] (c3) circle [radius = .1];
\draw [fill= black] (c5) circle [radius = .1];
\draw [fill= black] (c7) circle [radius = .1];
\draw [fill= white] (d1) circle [radius = .1];
\draw [fill= white] (b1) circle [radius = .1];
\draw [fill= black] (b8) circle [radius = .1];
\draw [fill= white] (b7) circle [radius = .1];
\draw [fill= black] (b6) circle [radius = .1];
\draw [fill= white] (b5) circle [radius = .1];
\draw [fill= black] (b4) circle [radius = .1];
\draw [fill= white] (b3) circle [radius = .1];
\draw [fill= black] (b2) circle [radius = .1];

\draw [fill= black] (90:2.75cm) circle [radius = .1];
\draw [fill= white] (135:2.75cm) circle [radius = .1];
\draw [fill= black] (180:2.75cm) circle [radius = .1];
\draw [fill= white] (225:2.75cm) circle [radius = .1];
\draw [fill= black] (270:2.75cm) circle [radius = .1];
\draw [fill= white] (315:2.75cm) circle [radius = .1];
\draw [fill= black] (0:2.75cm) circle [radius = .1];
\draw [fill= white] (45:2.75cm) circle [radius = .1];
\end{tikzpicture}
\caption{The $\Gr(2,8)$ bridge graph encoded by the maximal chain in $B_8(2,8)$ described in Examples~\ref{eg:domainaschain} and \ref{eg:chaintobridgegraph}. The bridge edges are labeled $t_{17},t_{35},\dots,$ by the corresponding transposition in $\tilde{S}^0_8$. The innermost $t_{17}$ edge is added first, then the vertical $t_{35}$ edge, then the vertical $t_{13}$ edge, and so on, ending with the $s_0$ edges. The Pl\"ucker coordinates corresponding to faces are in red. Each time a bridge edge is added, we add a Pl\"ucker coordinate to the collection. 
$15$ is deleted. 
} 
\label{fig:BridgeGraphs}
\end{figure}

We break the proof of Theorem~\ref{thm:TPtests} into two steps. The first is the following.

\begin{proof}[Proof that \eqref{eq:chaintocup} is 
efficient, weakly separated, and square-move connected.]
It is clear that \eqref{eq:chaintocup} is $\rho^\ell$-invariant since each of its constituent necklaces is. It will be an $\ell$-optimal collection if we show that it has size $\dim \Pi_{f_h}+1$, is weakly separated, and is a TP test. We show the first two statements now. 

The collection \eqref{eq:chaintocup} has the right size to be an efficient TP test by induction on $h$; the base case is that $\dim \widetilde{\Pi}_{f_0} = 1$ and in this case $\rmci_{f_0}$ consists of a single $\rho^\ell$-orbit. Each time we increment $h$ in \eqref{eq:chaintocup}, we change the necklace in exactly $p$ terms in a $\rho^\ell$-equivariant way, i.e. we add one Pl\"ucker coordinate to the TP test. On the other hand, $\dim \widetilde{\Pi}_{f_h}$ increases by one each time we increment $h$, so the collection \eqref{eq:chaintocup} is efficient by induction. 

We show that the collection \eqref{eq:chaintocup} is weakly separated by realizing it a subset of face labels of a reduced plabic graph (cf. Figure~\ref{fig:BridgeGraphs} for an example). We can view the saturated chain ${\bf f} = f_h,\dots,f_0$ in $(\Bound_n(k,\ell),\leq_b)$ as a chain ${\bf f'}$ of length $ph$ in $(\Bound_n(k,n),\leq_b)$: each cover in the form gives rise to $p$ commuting covers in the latter. Choose any saturated chain ${\bf f''}$ starting at $f_0$ and ending at a maximal element in $\Bound_n(k,n)$. Then ${\bf f' \cup f''}$ corresponds to a bridge graph $G = G({\bf f' \cup f''})$. Each time we take a downward step in the chain ${\bf f' \cup f''}$, we add a bridge to $G$, and  this has the effect of adding a boundary face. Since the Pl\"ucker coordinates in the boundary faces of a plabic graph are the elements of the Grassmann necklace, adding a bridge changes the Grassman necklace in exactly one term. 

When we have added all the bridges corresponding to ${\bf f''}$, the Grassmann necklace is~$\rmci_{f_0}$. Each cover in ${\bf f}$, corresponding to $p$ covers in ${\bf f'}$, changes exactly $p$ terms in the Grassmann necklace, in a $\rho^\ell$-invariant way, as in \eqref{eq:chaintocup}. Thus, the union \eqref{eq:chaintocup} is a subset of the face labels of the bridge graph $G$ as claimed.

Finally, we show the square-move connectedness statement. By Theorem~~\ref{thm:titsmatsumoto} all saturated chains in $\Bound_n(k,\ell)$ are connected by $2$-moves or $3$-moves. These correspond to sequences of (symmetrically performed) $2$-moves or $3$-moves in $\Bound_n(k,n)$. By \cite[Theorem 5.3]{WilliamsBridge}, performing a 2-move does not affect the set of face labels of a bridge graph, and performing a 3-move amounts to a square move on bridge graphs. 
\end{proof}

Before proving that \eqref{eq:chaintocup} is a TP test, we have a definition and a lemma. 
\begin{defn}\label{defn:superfluous}
Suppose that $p  \in [k,\infty)$ and let $\mcc = \mcc({\bf f})$ for a chain ${\bf f}$ as in \eqref{eq:chaintocup}. Extend $\mcc$ to a maximal weakly seprated collection $\mcc^\dagger$ inside the positroid $\mcm_{f_h} \subset \binom{[n]}k$. Then the elements of $\mcc^\dagger \setminus \mcc$ are {\sl superfluous variables}.
\end{defn}

\begin{lem}\label{lem:superfluous}
In the setting of Definition~\ref{defn:superfluous}, let $I \in \rmci_{f_0} \subset \mcc$ and let $I^\dagger$ be a superfluous variable. Then 
$\Delta_I = \aa \Delta_{I^\dagger} \in \bbc[\mcd]$ for a positive real number $\aa$. 
\end{lem}

Thus, the positivity of superfluous variables is implied by the positivity of $I \in \rmci_{f_0}$. 

\begin{proof}
The case $p=k$ is easily handled directly. So we assume $p > k$, so that $\bb = k$ as in Definition~\ref{defn:maximalelements}. Thus $f_0 = t_s$ for some $s \in [\ell]$, and the positroid $\mcm_{f_0}$ is drawn from the ground set $\mathscr{A}_s := \{j \in [n] \colon j \equiv s \mod \ell\}$ (cf. Example~\ref{eg:maximalelts}). We have a rational projection $P  \colon \Gr(k,n) \to \Gr(\bb,r) = \Gr(k,r)$ onto the corresponding columns $\{e_j \colon j  \in \mathscr{A}_s\}$. The domain of definition of $P$ is given by the non-vanishing of any Pl\"ucker coordinate in $\binom{\mathscr{A}_s}k$. In particular, any totally positive point is in the domain of definition. The projection $P$ sends $\rho^\ell$-fixed points to $\rho$-fixed points, and sends totally positive points to totally positive points. 

Since $\mcd = \mcd_n(k,\ell)$ is irreducible, restricting to $\mcd$ gives a rational map $P \colon \mcd \to \Gr(k,p)^{\rho}$ whose domain of definition is a connected topological space. The image of this map is therefore a point. Letting $X_0 \in \Gr(k,n)^{\rho}_{>0}$ and $Y_0 \in \Gr(k,p)^{\rho}_{>0}$ denote the unique totally positive points, we clearly have that $P(X_0) = Y_0$, so that $P$ is a rational map $\mcd \to \{Y_0\}$.

Because $\mcm_{f_0} \subset \binom {\mca_s}k$, we can choose a maximal weakly separated collection $\mcc' \subset \binom{[p]}k$ such that 
$P^*(\mcc') = \mcc_0$. If $J \in \binom {\mca_s}k$ we denote by $J' \in \binom{[r]}k$ the variable for whom $P^*(\Delta_{J'}) = \Delta_J$. Suppose that $X \in \mcd$ and that $\Delta_I(X) \neq 0$. Pl\"ucker coordinates of $X$ are only defined up to scale, but the ratio $\frac{\Delta_{I^\dagger}(X)}{\Delta_I(X)}$ is a well-defined number. 

We have
\begin{align}\label{eq:whatisalpha}
\frac{\Delta_{I^\dagger}(X)}{\Delta_{I}(X)} &= P^*\left(   \frac{\Delta_{(I^\dagger)'}}{\Delta_{I'}}\right)(X) \\
&= \frac{\Delta_{(I^\dagger)'}(Y_0)}{\Delta_{I'}(Y_0)} \in \bbr_{>0}
\end{align}
since $Y_0$ is TP. 
\end{proof}

\begin{proof}[The collection is a TP test]
In the previous step of the proof, we showed that $\mcc = \mcc({\bf f})$ is efficient provided we check it is a TP test. We also showed that $\mcc$ is a weakly separated collection in $\mcm_{f_h}$, so we can choose an extension $\mcc \subset \mcc^\dagger$ of this collection to a maximal weakly separated collection in $\mcm_{f_h}$. Then $\mcc^\dagger$ is a TP test for $\Gr(M_{f_h})$ \cite[Corollary 4.4]{GalashinLam}. If $p > k $ then Lemma~\ref{lem:superfluous} says that the positivity of $\mcc$ implies the positivity of $\mcc^\dagger$, so $\mcc$ is an efficient TP test. When $p=k$ it is not hard to see that $\mcc$ is already maximal. 

Now we address the cases $p< k$, so that $f_0 = t_{S,s}$ as in Proposition~\ref{prop:maximalelts}. We have a rational projection $\Gr(k,n) \to \Gr(k,k+p-\bb) \cong \Gr(p-\bb,(\aa+1) p)$ by projecting on to the $\rho^\ell$-orbit of columns $\{e_j \colon j \mod \ell \in S \cup s\}$, and then applying Grassmann duality. 
It sends $\rho^\ell$-fixed points to $\rho^{\aa+1}$-fixed points. From the proof of Lemma~\ref{lem:superfluous} we get a map $\Gr(p-\bb,(\aa+1) p)^{\rho^{\aa+1}} \to \Gr(p-\bb,p)^{\rho}$, thus altogether a map $\Gr(k,n) \to \Gr(p-\bb,p)^{\rho}$. The argument now concludes as in Lemma~\ref{lem:superfluous}.
\end{proof}

\section{Generalized clusters and cyclic symmetry loci}\label{secn:clusters}
Let $\mcd_n(k,\ell) \subset \Gr(k,n)^{\rho^\ell}$ be the distinguished component. We define certain seeds in $\bbc(\mcd_n(k,\ell)$ when the order of the orbifold point is at least $k$. Recall \eqref{eq:stabilityofD} that in these cases we have $\mcd_n(k,\ell) \cong (\bbp^{\ell-1})^k$. 

We conjecture that our seeds determine an upper generalized cluster algebra structure on $\bbc[\mcd_n(k,\ell)]$. Our main result in this direction is that each one-step mutation out of the initial seed yields an element of the coordinate ring. We discuss approaches to verifying the conjecture in general in subsequent sections.

\subsection{Initial seed}
\begin{defn}[Initial cluster]
\label{defn:initPluckers}
Given $k,\ell,n \geq 2$, we set $\ell' = \gcd(\ell,n)$ and set $p = \frac{n}{\ell'}$. We assume that $\ell' \geq 2$, and that $p \in [k,\infty)$. Set $N := (k-1)(\ell'-1)$. We define a sequence $I_1,\dots,I_N \in \binom{[n]}k$ as follows. For $j \in [k-1]$ and $i \in [\ell'-1]$, set 
\begin{equation}\label{eq:Ilist}
I_{(j-1)(\ell'-1)+i} = [i,i+j-1] \cup \{1+j\ell',1+(j+1)\ell',\dots,1+(k-1)\ell'\}
\end{equation}
In other words, we have that $I_i = \{i,\ell'+1,2\ell'+1,\dots,(k-1)\ell'+1\}$ for $i \in [\ell-1]$, that $I_i = \{i,i+1,2\ell'+1,\dots,(k-1)\ell'+1\}$ for $i \in [\ell',2(\ell'-1)]$, and so on. 
Define also $I_{N+i} = [-k+1,\dots,i] \in \binom{[n]}k$ for $i \in [\ell']$. 

We write $\mcc_n(k,\ell) = \{I_i \colon i \in [N+\ell']\}$ and write $\ov{\mcc_n(k,\ell)} = \{\rho^{a\ell}(I_i) \colon i \in [N+\ell'], a \in [p]\}$ for the $\rho^\ell$-invariant collection containing $\mcc_n(k,\ell)$.  
\end{defn}

An example of the sequence $I_1,\dots,I_{N+\ell'}$ is in Figure~\ref{fig:onestep}. Note that the variables $I_{N+1},\dots,I_{N+\ell'}$ are $\rho^\ell$-orbit representatives for the frozen variables in $\bbc[\Gr(k,n)]$. Note also that the number of elements $N+\ell' = k(\ell'-1)+1$ is the dimension of homogeneous coordinate ring $\bbc[\mcd] \cong \bbc[(\bbp^{\ell'-1})^{k}]$. 

Let $q$ be an indeterminate. For $k,n \in \bbn$ set 
\begin{equation}
[k]_q = \frac{q^{\frac{k}{2}}-q^{-\frac{k}{2}}}{q^{\frac{1}{2}}-q^{-\frac{1}{2}}} \hspace{.5cm}  [n]!_q = \prod_{k=1}^n[k]_q \hspace{.5cm} {n \brack k}_q = \frac{[n]!_q}{[k]!_q[n-k]!_q}, 
\end{equation}
the $q$-analogs of $k \in \bbn$, $n!$, and $\binom nk$, respectively.

\begin{defn}\label{defn:abstractCSseed}
For arbitrary $k,\ell \geq 2$ and $N = (k-1)(\ell-1)$, we have a CS- seed $\Sigma_{\rm cyc}(k,\ell) = (\ov{\tilde{x}},\tilde{B},{\bf z})$ are defined as follows. The initial variables $(\ov{x}_i)_{i \in [N+\ell]}$ are indeterminates with the last $\ell$ of these frozen. The exchange matrix is skew-symmetric, so that $\tilde{B}$ can be defined by an extended quiver with arrows 
$x_{i+1} \to x_i \to x_{i+\ell} \to x_{i+1}$ for $i=1,\dots,N$. 
The exchange degrees  are $d_1 = k$ and $d_i = 1$ for $i \in [2,N]$. We
abbreviate the coefficient string variables $z_s := z_{1;s}$.  We denote the resulting CS- cluster algebra by $\mca_{\rm cyc}(k,\ell)$, upper CS- cluster alebra by $\mca^{\rm up}_{\rm cyc}(k,\ell)$, etc. 

Given $n$, set $p := \frac{n}{\gcd(\ell,n)}$ as usual and assume that $p \geq k$. Define 
\begin{equation}\label{eq:trigcoeffs}
\eta_{s} := {k \brack s}_q \text{ evaluated at } q = e^{\frac{2\pi i}{p}}.
\end{equation}
We define $\mca_n(k,\ell) \subset \bbc(\mcd_n(k,\ell))$ the result of identifying $\ov{x}_i \in \mca_{\rm cyc}(k,\ell')$ with $\Delta_{I_i} \in \bbc[\mcd_n(k,\ell)]$, and performing the coefficient specialization $z_s \mapsto \eta_s$. We denote by $\Sigma_n(k,\ell)$ its initial seed. 
\end{defn}

\begin{figure}
\label{fig:onestep}
\begin{center}
\begin{tikzcd}
I_1 \arrow[rd,]           &                       \\
I_2 \arrow[u,] \arrow[rd] & I_5 \arrow[l]         \\
I_3 \arrow[u] \arrow[rd]            & I_6 \arrow[l]         \\
I_4 \arrow[u] \arrow[rd]            &\boxed{I_7} \arrow[l] \\
I_5 \arrow[u] \arrow[rd]            & \boxed{I_8} \arrow[l] \\
I_6 \arrow[u] \arrow[rd]            & \boxed{I_9} \arrow[l] \\
\boxed{I_7} \arrow[u]               & \boxed{I_{10}}       
\end{tikzcd}
\hspace{1.5cm}
\begin{tikzcd}
 159 \arrow[dddd, "{}"', bend right=60]           &                                   \\
259 \arrow[u ] \arrow[dddd, bend right=60] &                                   \\
 359 \arrow[u] \arrow[dddd, bend right=60]             &                                   \\
459 \arrow[u] \arrow[rd]                              &                                   \\
569 \arrow[u] \arrow[uuu, bend left=49] \arrow[rd]    &  \boxed{8,9,10} \arrow[l]   \\
679 \arrow[u] \arrow[uuu, bend left=49] \arrow[rd]    & \boxed{9,10,11} \arrow[l]  \\
\boxed{789} \arrow[uuu, bend left=49] \arrow[u]       & \boxed{10,11, 12} \arrow[l]
\end{tikzcd}
\end{center}
\caption{Two drawings of the initial seed $\Sigma_{12}(3,4)$. The first drawing has repeated variables but is easier to look at. The second removes the repetitions and replaces $I_i$ with the subset $I_i \in \binom{[12]}3$ as in Definition~\ref{defn:initPluckers}.}
\end{figure}
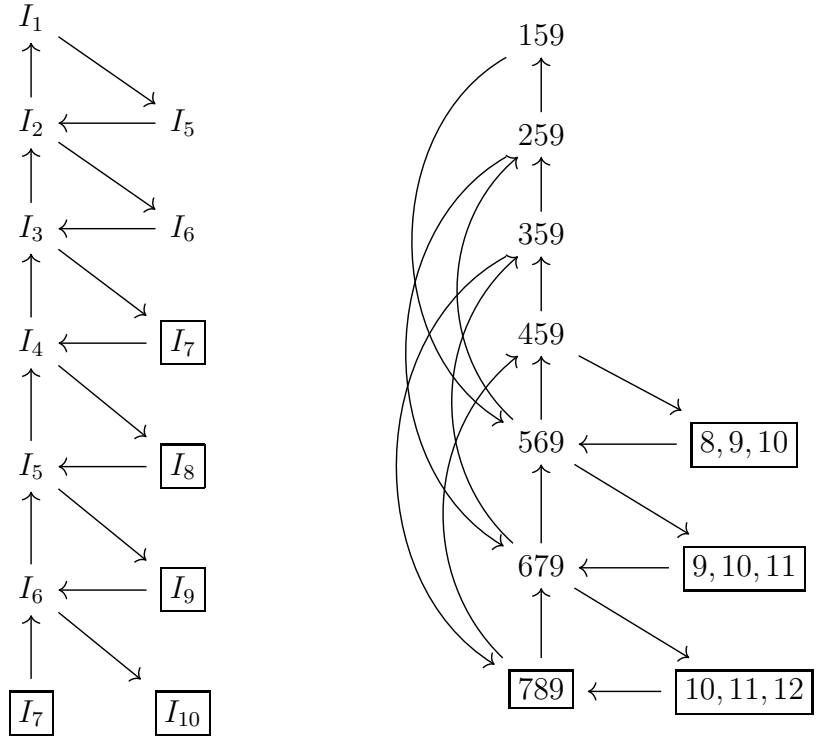

\begin{rmk}
The specialization \eqref{eq:trigcoeffs} is palindromic and is compatible with the convention $z_{0} = z_{k} = 1$. Expressing $e^{i\theta} = \cos \theta + i \sin \theta$, we can rewrite
\begin{equation}\label{eq:triggercoeffs}
\eta_s = \prod_{j=1}^{s} \frac{\sin \frac{(k+1-j) \pi}{p}}{\sin \frac{j\pi}{p}} \in \bbr_{\geq 0}.
\end{equation}
We have the following special cases of the numbers $(\eta_s)_{s \neq 0,k}$, depending on the value of~$p$:
\begin{itemize}
\item When $p=k$, then each $\eta_s = 0$. This is the specialization corresponding to the right companion cluster algebra. So $\mca_{k\ell}(k,\ell)$ {\sl is} the right companion cluster algebra, and in particular is an FZ-cluster algebra.
\item As $p \to \infty$, $e^{\frac{2\pi i}{p}}$ limits to~$1$, so that each $\eta_s$ limits to $\binom k s$. This is the specialization corresponding to the left companion cluster algebra.
\item When $p=k+1$, each $\eta_s=1$. 
\end{itemize}    
\end{rmk}
 
\begin{rmk}\label{rmk:qbinomialthm}
By the $q$-binomial theorem, the exchange polynomial $\sum_{s=0}^k \eta_su^sv^{k-s}$ factors as $\prod_{j=}^{k-1} (u-q^{j-\frac{k-1}{2}}v)$ where $q = e^{\frac{2 \pi i}{p}}$.\end{rmk}

Let $\mcl(\Sigma_n(k,\ell))$ denote the lower bound algebra associated to the seed $\Sigma_n(k,\ell) \subset \bbc(\mcd_n(k,\ell)$ and let $\bbc[\Delta(\mcc_n(k,\ell))^{\pm 1}]$ denote the algebra of Laurent polynomials in this seed (thought of as functions on $\mcd$).

Our main result is: 
\begin{thm}\label{thm:onestep} We have the containment of algebras 
\begin{equation}\label{eq:easycontain}
\mcl_{n}(k,\ell) \subseteq \bbc[\mcd] \subseteq \bbc[\Delta(\mcc_n(k,\ell))^{\pm 1}]
\end{equation}
for any $p \in [k,\infty)$. 
\end{thm}

\begin{conj}\label{conj:containments}
The strengthened inclusions 
\begin{equation}\label{eq:hardcontain}
\mca_n(k,\ell) \subset \bbc[\mcd_n(k,\ell)] \subset \mca^{\rm up}_n(k,\ell)
\end{equation}
hold for any $\ell$ and any $p \geq k$. 
\end{conj}

We expect in fact that $\bbc[\mcd] = \mca^{\rm up}_n(k,\ell)$ always and that $\mca_n(k,\ell) \subsetneq \bbc[\mcd]$ typically. Example~\ref{eg:AnotU} establishes the proper containment 
$$\mca_8(4,2) \subsetneq \bbc[\mcd_8(4,2)] \subseteq \mca^{\rm up}_8(4,2).$$ 

In the subsequent sections, we discuss two different approaches to proving Conjecture~\ref{conj:containments}.  Our first approach is to compare with a generalized cluster algebra whose upper generalized cluster algebra is known. Specifically, we use a generalized cluster structure on a space of infnite, periodic, band matrices due to Gekhtman, Shapiro, and Vainshtein. We carry out this comparison in Section~\ref{secn:quantum}. A second approach is to compare the generalized cluster algebra $\mca_n(k,\ell)$ with its ``unfolding'' $\bbc[\Gr(k,n)]$. We carry out this approach in Section~\ref{secn:folding}. Using these methods, we can prove:
\begin{itemize}
\item In the finite cluster type cases, we have $\mca_n(k,\ell) =\bbc[\mcd]$. In the finite mutation type cases $\mcd_9(3,3)$, $\mcd_{12}(3,3)$, $\mcd_8(4,2)$, and $\mcd_{8}(4,3)$, we have \eqref{eq:hardcontain}. Both statements are proved by unfolding. 
\item When $k=3$, the containments \eqref{eq:hardcontain} hold 
for the versions of these algebras in which we localize at frozen variables. This is proved by comparison with band matrices. By filling in a gap (an {\sl isospectrality conjecture} which we have not yet proved), this strategy should work for arbitrary $k$ and arbitrary $\ell>k$.
\end{itemize}

\begin{rmk}
A third approach to proving the containments \eqref{eq:hardcontain}
is via the ``usual'' commutative algebra methods. One typically proves an inclusion of the (upper) cluster algebra in $\bbc[\mcd]$ via the Starfish Lemma \cite[Proposition 3.6]{tensors}.  
To apply this lemma, one must check that i) $\bbc[\mcd]$ is a normal domain,  ii) the inclusion $\mcl \subset \bbc[\mcd]$ holds, iii) all initial cluster variables $\Delta_{I_i}$ and $\Delta_{I_j}$ are coprime in $\bbc[\mcd]$, and iv) the variables $\Delta_{I_i}$ and $\mu_i(\Delta_{I_i})$ are coprime for each initial mutable variable. Assertion i) holds, and ii) holds by \eqref{eq:easycontain}. We are unable to check the coprimeness conditions iii) and iv). The main difficulty is that $\bbc[\mcd]$ is not a unique factorization domain. 

It is not true that every Pl\"ucker coordinate is a cluster variable. (This would be the easiset way of proving the reverse inclusion $\bbc[\mcd] \subset \mca^{\rm up}$). One typically uses an argument with upper bound algebras to prove the reverse inclusion \cite[Theorem 3.11]{GSVDouble}). One would need to check that these arguments go through after specializing the coefficient string variables $z_s \mapsto \eta_s$. 
\end{rmk}

\begin{lem}\label{lem:seedandtest}
The collection $\ov{\mcc}_n(k,\ell)$ is an efficient TP test for $\mcd$. 
\end{lem}

Thus, the notion of positivity coming from the cluster structure $\mca_n(k,\ell)$ (positivity of cluster variables) coincides with the notion of total positivity (positivity of Pl\"ucker coordinates).

\begin{rmk}\label{rmk:doitforpositroidvarieties}
Our construction of seeds makes sense for the $\rho^\ell$-fixed points in any positroid variety (i.e., not only for the top-dimensional positroid variety $\Gr(k,n)$). It seems natural to expect analogous containments of algebras  in this setting, but we have not checked this is in any examples. 
\end{rmk}

The most nontrivial assertion in Theorem~\ref{thm:onestep} is the {\sl regularity } of neighboring variables, i.e. the inclusion $\mcl_n(k,\ell) \subset \bbc[\mcd]$. We prove this at the end of this section. For the special variable $I_1$, it requires proving that the following polynomial in Pl\"ucker coordinates
\begin{equation}\label{eq:CSreln}
\displaystyle \sum_{s=0}^k \,  {k \brack s}_{q = \exp(\frac{2\pi i}{p})}\, \Delta_{I_2}^s\, \Delta_{I_{\ell+1}}^{k-s}  \in \bbc[\mcd_{p\ell}(k,\ell)]_{(k)}
\end{equation}
is divisible by $\Delta_{I_1}$.
We refer to the resulting algebraic identity as the {\sl $\SL_k$-higher 
orbifold Ptolemy relation}, in the spirit of \cite{ChekShap,FGMod}.

\begin{example}\label{eg:orbifoldexamples}
The right hand sides of the $\SL_2$, $\SL_3$, and $\SL_4$ orbifold-Ptolemy relations describing mutation out of the initial seed are 
\begin{align*}
&\Delta_{I_2}^2+\frac{\sin \frac{2 \pi}{p}}{\frac{\sin \pi}{p}}\Delta_{I_2}\Delta_{I_{\ell+1}}+ \Delta_{I_{\ell+1}}^2, \\
&\Delta_{I_2}^3+\, 
\frac{\sin \frac{3\pi}p}{\sin \frac \pi p}\Delta_{I_2}^2\Delta_{I_{\ell+1}}
+ \, \frac{\sin \frac{3\pi}p}{\sin \frac \pi p}\Delta_{I_2}\Delta_{I_{\ell+1}}^2
+\, \Delta_{I_{\ell+1}}^3, \\
&\Delta_{I_2}^4+\, 
\frac{\sin \frac{4\pi}p}{\sin \frac \pi p}\Delta_{I_2}^3\Delta_{I_{\ell+1}}
+\, \frac{\sin \frac{4\pi}p}{\sin \frac \pi p}\frac{\sin \frac{3\pi}p}{\sin \frac {2\pi} p}\Delta_{I_2}^2\Delta_{I_{\ell+1}}^2
+\, \frac{\sin \frac{4\pi}p}{\sin \frac \pi p}\Delta_{I_2}\Delta_{I_{\ell+1}}^3
+\, \Delta_{I_{\ell+1}}^4.
\end{align*}
In the $\SL_2$ case, the variables on the left hand side are $I_1 = \Delta_{1,\ell+1}$ and $\Delta_{2,\ell+2}$. In the $\SL_3$ case, $I_1 = \Delta_{1,\ell+1,2\ell+1}$ and the neighboring variable is a quadratic expresion in Pl\"ucker coordinates $\Delta_{2,\ell+2,2\ell+1}\Delta_{\ell+1,2\ell+2,3\ell+1}-\Delta_{\ell+1,\ell+2,2\ell+1}\Delta_{2,2\ell+2,3\ell+1}$, the simplest non-Pl\"ucker cluster variable. In the $\SL_4$ case, $I_1 = \Delta_{1,\ell+1,2\ell+1,3\ell+1}$, and the neighboring variable is a cubic expression in Pl\"ucker coordinates. For example, when $\ell=2$, this expression is $$\Delta_{2457}(\Delta_{3679}\Delta_{589,11}-\Delta_{5679}\Delta_{389,11})-\Delta_{3457}(\Delta_{2679}\Delta_{589,11}-\Delta_{5679}\Delta_{289,11}).$$
\end{example}

Now we prove Lemma~\ref{lem:seedandtest} and the right inclusion from \eqref{eq:easycontain}. 

\begin{proof}[Proof of Lemma~\ref{lem:seedandtest}]
First we argue that $\ov{\mcc_n(k,\ell)}$ is of the form $\mcc({\bf f})$ as in \eqref{eq:chaintocup}, for the maximal chain
\begin{equation}\label{eq:whichchain}
{\bf f} = f_0 \gtrdot f_0s_1 \gtrdot f_0s_1s_2 \cdots \gtrdot f_0s_1 \cdots s_{N+\ell-1} = {\rm id}_k \in \Bound_n(k,\ell),
\end{equation}
where $f_0 \in \Bound_n(k,\ell)$ is a certain maximal element, and $s_i \in \tilde{S}^0_\ell$ denote simple transpositions. More specifically, $f_0$ has window notation $[1+k\ell,2,3,\dots,\ell]$, which is $t_1$ in the notation of Definition~\ref{defn:maximalelements}. The corresponding Grassmann necklace $\rmci_{f_0}$ consists of $I_1$ from Definition~\ref{defn:initPluckers} and its $\rho^\ell$ shifts. 

Performing the sequence of covers above has the effect of moving the value $1+k\ell$, which is initially in the first position of the window notation, rightwards until it is in position $1+n$. By $\ell$-periodicity, when we move $k\ell+1$ rightwards, we simultaneously move $(k-1)\ell+1,(k-2)\ell+1$, etc., rightwards. Altogether, we move the value $k\ell+1$ a total of $k$ windows rightwards. And in each window, this requires $\ell-1$ swaps, for a total of $k(\ell-1) = N+\ell-1$ swaps. Straightforward bookkeeping shows that the $i$th element in the Grassmann necklace $I_i \in \rmci_{f_i} \setminus \rmci_{f_{i-1}}$ where $I_i$ is as in Definition~\ref{defn:initPluckers}. So $\ov{\mcc_n(k,\ell)} = \mcc({\bf f})$.

Let $\mcc^\dagger$ be an extension of $\ov{\mcc_{n}(k,\ell)}$ to a maximal weakly separated collection. So $\Delta(\mcc^\dagger)$ is a cluster in $\bbc[\Gr(k,n)]$. By Lemma~\ref{lem:superfluous}, each element of $\mcc^\dagger$ is linearly related to an element of $\mcc_{n}(k,\ell)$ inside $\bbc[\mcd_n(k,\ell)]$. By the Laurent phenomenon for $\bbc[\Gr(k,n)]$ we can express $\Delta_I$ as a Laurent polynomial in $\Delta(\mcc')$, hence as a Laurent polynomial in the elements of $\Delta(\mcc_n(k,\ell))$ once we restrict functions to $\mcd$. 
\end{proof}

Next, we establish the left inclusion $\mcl_n(k,\ell) \subset \bbc[\mcd_n(k,\ell)]$ of \eqref{eq:easycontain}. That is, we show that each mutation out of the initial cluster yields a regular function on $\mcd_n(k,\ell)$. The exchange relations for $\Delta_{I_i}$ when $i \geq 2$ unfold to exchange relations in $\bbc[\Gr(k,n)]$. Each of these relations is a three-term Pl\"ucker relation. Thus, 
$\mu_i(\Delta_{I_i})$ will be a Pl\"ucker coordinate, and therefore $\mu_i(\Delta_{I_i}) \in \bbc[\mcd_n(k,\ell)]$. So we immediately reduce to checking that the CS-exchange relation \eqref{eq:CSreln} is divisible by $\Delta_{I_1}$ inside $\bbc[\mcd]$. The next several results work up to a proof of this. We conclude this section with several examples illustrating the various steps.

\begin{defn}\label{defn:Weylsidentity}
Let $n = p\ell$ with $p \geq k$. Fix the standard Hermitian inner product $( ,)$ on $\bbc^k$ in which the standard basis vectors $e_1,\dots,e_k$ are orthonormal. The {\sl generalized cross product} of vectors $v_1,\dots,v_{k-1} \in \bbc^k$ is the unique vector $v_1 \times v_2 \times \cdots \times v_{k-1} \in \bbc^k$ satisfying 
$(v,v_1 \times \cdots \times v_{k-1}) = \det(v,v_1,\dots,v_{k-1})$ for all $v \in \bbc^k$. Suppose that $X \in \tGr(k,n)$ is represented by a $k \times n$ matrix with column vectors $v_1,\dots,v_n$. For $i \equiv 1 \mod \ell$, set $v_i':= v_{i+1} \times v_{i+\ell} \times v_{i+2\ell} \times \cdots \times v_{i+(k-2)\ell}$. We obtain a function 
\begin{equation}\label{eq:L}
X = (v_1,\dots,v_n) \overset{L}{\mapsto} \det(v_1',v'_{1+\ell},\cdots,v'_{1+(k-1) \ell}) \in \bbc[\Gr(k,n)].\end{equation}
\end{defn}

It is possible to give an explicit expression for the function \eqref{eq:L}
as a homogeneous polynomial of degree $k-1$ in the Pl\"ucker coordinates of $X$. So the function \eqref{eq:L} indeed lies in $\bbc[\Gr(k,n)]$ as claimed. For $a \in \bbn$ let $\bbc[\Gr(k,n)]_{(a)}$ denote the subspace spanned by degree $a$ monomials in Pl\"ucker coordinates. Recall $I_1 \in \binom{[n]}k$ as in Definition~\ref{defn:initPluckers}.

\begin{lem}\label{lem:Weylsidentity}
There exists an element $L  \in \bbc[\Gr(k,n)]_{(k-1)}$ satisfying 
\begin{equation}\label{eq:Weylsidentity}
\Delta_{I_1} L = \det(\Delta_{i,j+1,j+\ell,j+2\ell,\dots,j+(k-2)\ell})_{i,j \in I_1} \in \bbc[\Gr(k,n)].
\end{equation}
Specifically, $L$ is the function \eqref{eq:L}.
\end{lem}

\begin{proof} In light of Definition~\ref{defn:Weylsidentity}, this is the assertion that the product of two determinants is the determinant of the product. 
\end{proof}

\begin{rmk}\label{rmk:twistmap}
Recall from the proof of Lemma~\ref{lem:seedandtest} that 
$\ov{\mcc_n(k,\ell)}$ is of the form $\mcc({\bf f})$ as in \eqref{eq:chaintocup}. Consider the necklace $\rmci_{f_{\ell+1}}$ which that the property that $I_2,I_{\ell+1} \in \rmci_{f_{\ell+1}}$ where 
$I_i$ are as defined in Definition~\ref{defn:initPluckers}.

Muller and Speyer associated to {\sl any} $f \in \Bound_n(k,n)$ a {\sl twist automorphism} of the positroid variety labeled by $f$. Although Muller and Speyer studied this map as an automorphism of the positroid variety $\Gr(\mcm_f)$, the definition makes sense as a rational endomorphism of $\Gr(k,n)$. One checks that the regular function $L$ from \eqref{eq:L} is the pullback of $\Delta_{I_1}$ along the right twist map associated to $f_{i+1}$.  The identity \eqref{eq:Weylsidentity} is an instance of \cite[Lemma 6.5]{MSTwist}.

We make the guess that there is a similar story for each of the clusters $\mcc({\bf f})$: mutation at $I_1$ should correspond to pulling back $\Delta_{I_1}$ along an appropriate Grassmann necklace $\rmci_{f_i}$ with $f_i \in {\bf f}$.  
\end{rmk}

The remainder of this section is devoted to establishing that the determinantal identity \eqref{eq:Weylsidentity} becomes the CS-relation \eqref{eq:CSreln} once we restrict functions to the distinguished component.

The first observation is immediate. 
\begin{lem}\label{lem:Toep}
When viewed in $\bbc[\Gr(k,n)]^{\rho^\ell}$, the matrix on the right hand side of \eqref{eq:Weylsidentity} is a Toeplitz matrix. The first column of this Toeplitz matrix is $(\Delta_{I_2},0,\dots,\Delta_{I_{\ell+1}})$. The $s$th entry in the first row is the Pl\"ucker coordinate $\Delta_{1,2+(s-1)\ell,1+(s+1)\ell,\dots,1+(s+k-3)\ell}$. 
\end{lem}

Our next short term goal is to simplify the first row of this Toeplitz matrix. This will take some preparations. 

For a set $S$ of positive integers, let $T_S:= \{1,1+\ell,\dots,1+(k+|S|-1)\ell\} \setminus \{1+s\ell \colon s \in S\}$. Thus, $T_S$ consists of the first several numbers equivalent to 1 modulo $\ell$, omitting certain ``gaps'' which are encoded by~$S$. As a special case, note that $T_\emptyset = I_1$.

\begin{defn}
We denote by $\eta_S := \frac{\Delta_{T_S}}{\Delta_{T_\emptyset}} \in \bbc[\mcd]$. By the argument in Lemma~\ref{lem:superfluous},  $\eta_S$ is a positive real number. \end{defn}

\begin{lem}When $S = \{s\}$ is a singleton, then the number $\eta_s$ is given in 
\eqref{eq:trigcoeffs}. \end{lem}

That is, we have not overloaded our notation. 

\begin{proof}
Karp \cite[Theorem 1]{Karp} gives a formula for the Pl\"ucker coordinates of the unique point $X_0 \in \Gr(k,r)^{\rho}_{>0}$.
$$\Delta_{i_1,\dots,i_k}(X_0) = \prod_{j < s}\sin \frac{(i_s-i_j)\pi}{n}.$$ The right hand side of this formula  only depends on the multiset of ``gaps'' between numbers in~$I$. Using this formula, we can check that the ratio $\frac{\Delta_{T_s}}{\Delta_{T_0}}$ is given by the formula \eqref{eq:trigcoeffs}. After projecting onto columns $1,\ell+1,\dots$, the numerator of this ratio depends on the multiset of gaps in $[0,k] \setminus s$ while the denominator corresponds to the multiset of gaps in $[0,k-1]$. We may ignore the contributions of pairs taken from $[0,k-1] \setminus s$, since both numerator and denominator has these. The gaps for the numerator, then, have size $k,k-1,\dots,k-s+1,k-s-1,\dots,1$, whereas the gaps for the denominator have size $s,s-1,\dots,1,1,\dots,k-s-1$, establishing the formula. \end{proof}

Our next lemma expresses the first row of the Toeplitz matrix in Lemma~\ref{lem:Toep} in terms of $\Delta_{I_2}$ and $\Delta_{I_{\ell+1}}$.

\begin{lem}\label{lem:firstlinear}
In  $\bbc[\Gr(k,r \ell)]^{\rho^\ell}$, we have the following linear relation: 
\begin{equation}\label{eq:firstlinear}
\Delta_{1,2+(s-1)\ell,1+s\ell,1+(s+1)\ell,\dots,(s+k-3)\ell} = \eta_{[s-2]}\Delta_{I_{2}}+\eta_{[s-1]}\Delta_{I_{\ell+1}}.
\end{equation}
We also have the following equality of positive real number: 
\begin{equation}\label{eq:straightenrs}
\eta_j\eta_{[s]} = \eta_{[s-1] \cup \{j+s\}} + \eta_{[s] \cup \{j+s\}}.
\end{equation}
\end{lem}

\begin{proof}
The relation \eqref{eq:firstlinear} is empty when $s=1$. For $s>1$, we multiply the left hand side of \eqref{eq:firstlinear} by $\rho^{(s-1)\ell}(\Delta_{I_1})$ and apply a three-term Pl\"ucker relation to get two terms on the right hand side. Dividing both sides by $\Delta_{I_1}$, and working in $\bbc[\Gr(k,n)]^{\rho^\ell}$, the relation \eqref{eq:firstlinear} results. The second relation \eqref{eq:straightenrs} is proved similarly.
\end{proof}

Combining Lemma~\ref{lem:Toep} with~\eqref{eq:firstlinear}, we have 
\begin{align}
\Delta_{I_1} L &= \det(\Delta_{i,j+1,j+\ell,j+2\ell,\dots,j+(k-2)\ell})_{i,j \in I}  \\
&= \det((\eta_{[j-i]}\Delta_{I_{2}}+\eta_{[j-i+1]}\Delta_{I_{\ell+1}})_{i,j \in [k]}) \in \bbc[\mcd_0]. \label{eq:WeylsinXo}
\end{align}
This latter determinant \eqref{eq:WeylsinXo} is manifestly a homogeneous polynomial of degree $k$ in the Pl\"ucker coordinates $\Delta_{I_2}$ and $\Delta_{I_{\ell+1}}$. So \eqref{eq:WeylsinXo} is already a CS- exchange relation.

As a final step, we directly evaluate this determinant to obtain the Higher orbifold Ptolemy formula \eqref{eq:CSreln}. The proof is based an explicit sequence of row reductions, combined with judicious use of \eqref{eq:straightenrs}. The details are in Section~\ref{secn:proofToep}. 
\begin{lem}\label{lem:principalminors}
Consider the infinite Toeplitz matrix $M = (\eta_{[j-i]}\Delta_{I_{2}}+\eta_{[j-i+1]}\Delta_{I_{\ell+1}})_{i,j \in \bbn}$, with rows $R_1,R_2,\dots$. Then the $t \times t$ principal minor is given by 
\begin{equation}\label{eq:principal}
\det(M_{i,j \in [t]}) = \sum_{s = 0}^t \eta_s\Delta_{I_{\ell+1}}^s\Delta_{I_{2}}^{t-s}.
\end{equation}
Moreover, let $M'$ be the matrix obtained from $M$ by the row operation $R_1 \mapsto \eta_1R_2-R_1$. Then, 
\begin{equation}\label{eq:offprincipal}
\det(M')_{i \in 1 \cup [3,t], j\in [2,t]} = \eta_t\Delta_{I_{\ell+1}}^{t-1}. 
\end{equation}
\end{lem}

We only need \eqref{eq:principal} in the sequel, but we will conclude it from \eqref{eq:offprincipal}, which we prove directly. 


The harder inclusion $\mcl_n(k,\ell) \subset \bbc[\mcd]$ asserted in Theorem~\ref{thm:onestep} now follows. We summarize the steps below:
\begin{proof}[Proof of the left inclusion \eqref{eq:easycontain}]
By the discussion we started with, we merely need to show that $\mu_1(\Delta_{I_1}) \in \bbc[\mcd]$. Let  $L \in \bbc[\Gr(k,n)]$ be the regular function from Definition~\ref{defn:Weylsidentity}. We claim that $L = \mu_1(\Delta_{I_1})$ completing the proof since  $L\in \bbc[\mcd]$. Indeed, in $\bbc[\Gr(k,n)]$ the product $\Delta(I_1)L$ is equal to the right hand side of \eqref{eq:Weylsidentity}, which equals the determinant \eqref{eq:WeylsinXo} in $\bbc[\mcd]$, which equals \eqref{eq:CSreln} by Lemma~\ref{lem:principalminors}. 
\end{proof} 

\begin{rmk} Each step in this proof relies only on Pl\"ucker relations, so does not have any dependence on the order $p$ of the orbifold point. The dependence on $p$ in \eqref{eq:CSreln} arises when we replace the ratio of Pl\"ucker coordinates $\frac{\Delta_{T_s}}{\Delta_{T_\emptyset}}$ by the positive real number \eqref{eq:triggercoeffs}, and ultimately depends on the formulas for the Pl\"ucker coordinates of the TP $\rho$-fixed point.
\end{rmk}

We close this section by illustrating these calculations when $k = 2,3,4$ and $\ell=2$. We use the notation $J = \Delta_{I_{\ell+1}}$ and $K = \Delta_{I_2}$ as in the proof of Lemma~\ref{lem:Toep}. 
\begin{example}[$\SL_2$ calculation]
Consider $\mcd_{2p}(2,2)$. We have $L = \Delta_{24}$ and 
\begin{align*}
\Delta_{13}L &= \det \begin{pmatrix}\Delta_{12} & \Delta_{14} \\ -\Delta_{23} & \Delta_{34}\end{pmatrix}  \text{ (by \eqref{eq:Weylsidentity})}\\
&= \det \begin{pmatrix}J & \Delta_{14} \\ -K & J\end{pmatrix} \text{ (by Lemma~\ref{lem:Toep})}\\ 
&= \det \begin{pmatrix}J & \eta_\emptyset K+ \eta_1 J\\ -K & J\end{pmatrix} \text{ (by Lemma~\ref{lem:firstlinear})}\\
&=J^2+\eta_1JK+K^2 \\
&=J^2+\frac{\sin \frac{2\pi}{p}}{\sin \frac{\pi}{p}}JK+K^2. 
\end{align*}
\end{example}

\begin{example}[$\SL_3$ calculation]
For $\mcd_{2p}(3,2)$, the regular function $L$ is the quadratic expression in Pl\"ucker coordinates from Example~\ref{eg:orbifoldexamples}. We have
\begin{align*}
\Delta_{135} L&= \det \begin{pmatrix}\Delta_{123} & \Delta_{145} & \Delta_{167} \\ 0 & \Delta_{345} & \Delta_{367} \\ \Delta_{235} & 0 & \Delta_{567} \end{pmatrix} \text{ (by \eqref{eq:Weylsidentity} ) }\\
& = \det \begin{pmatrix}\Delta_{123} & \Delta_{145} & \Delta_{167} \\ 0 & \Delta_{123} & \Delta_{145} \\
\Delta_{235} & 0 & \Delta_{123} \end{pmatrix} \text{ (by Lemma~\ref{lem:Toep})}\\
&= \det \begin{pmatrix}J & K+\eta_1J & \eta_1K+\eta_{[2]}J\\ 0 & J& K+\eta_1J \\
K & 0 & J \end{pmatrix} \text{ (by Lemma~\ref{lem:firstlinear})}\\
&= J^3+K(K^2+(\eta_1^2-\eta_{[2]})J^2+\eta_1JK) \\
&= K^3+\eta_1K^2J+\eta_2KJ^2+J^3  \text{ (by \eqref{eq:straightenrs})}\\
&= K^3+\frac{\sin \frac{3 \pi}{p}}{\sin \frac{\pi}{p}}K^2J+\frac{\sin \frac{3 \pi}{p}}{\sin \frac{\pi}{p}}KJ^2+J^3. 
\end{align*}
\end{example}

\section{Connections with band matrices and \texorpdfstring{$U_q(\hat{\mathfrak{sl}}_k)$}{quantum affine algebras}}
\label{secn:quantum}
\subsection{CS-cluster structure on periodic band matrices}
Let ${\rm Band}(k,\ell)$ denote the space of matrices $M = (M_{ij})_{i,j \in \bbz}$ satisfying 
$$M_{i,j} = 0 \text{ unless } j-i \in \{0,\dots,k\}, \text{ and } M_{i+\ell,j+\ell} = M_{i,j} \text{ for all } i,j.$$
Thus, $M$ is supported on $k+1$ diagonals, each of which is an $\ell$-periodic sequence. Let us assume moreover, as in \cite{GSVStaircase}, that the parameters satisfy $k<\ell$. Then we can repackage $M$ as the following pair of $\ell \times \ell$ matrices 
\begin{equation}\label{eq:characteristicpolyXY}
A = 
\begin{pmatrix}
 0 & \cdots & 0 & M_{11} & M_{12} & \cdots & M_{1k} \\
 0 & \cdots  & \cdots  & 0 & M_{32} & \cdots & M_{2k} \\
 0& \cdots& \cdots  & \cdots & \ddots & \ddots & \vdots \\
0& \cdots& \cdots  & \cdots & \cdots & 0& M_{k,k} \\ 
0& \cdots& \cdots  & \cdots & \cdots & \cdots & 0 \\  
0& \cdots& \cdots& 0& 0 & \cdots & 0
\end{pmatrix} \text{ and }
B = 
\begin{pmatrix}
M_{1,k+1} & 0 & \cdots & \cdots & 0  \\
M_{2,k+1} & M_{2,k+2} & 0  & \cdots  &\vdots \\
\vdots & \vdots & \ddots& \ddots  & \vdots \\
\vdots & \vdots & \ddots& \ddots  & \vdots  \\ 
\vdots & \vdots & \ddots& \ddots  & 0   \\  
M_{\ell,k+1}& \cdots & \cdots& \cdots& M_{\ell,k+\ell+1} 
\end{pmatrix}.
\end{equation}
Define 
\begin{equation}\label{eq:characteristicpoly}
c(A,B) = \det(tB+A) = \prod_{i=1}^{\ell}M_{i,i+k}\det(tI+B^{-1}A), \text{ and } c_s = [t^{\ell-s}]c(A,B) \in \bbc[{\rm Band}(k,\ell)]. \end{equation}
The notation $[t^{\ell-s}]$ stands for ``coefficient of $t^{\ell-s}$.''
The second equality above only makes sense when $B$ is invertible, in which case the factor  $\det(tI+B^{-1}A)$ is a characteristic polynomial. 

Let $M_I^J \in \bbc[{\rm Band}(k,\ell)]$ denote the matrix minor with row set $I$ and column set $J$, thought of as an element of the coordinate ring  $\bbc[{\rm Band}(k,\ell)]$. The latter is a polynomial ring in the $\ell(k+1)$ matrix entries. Gekhtman, Shapiro, and Vainshtein defined a CS-seed (in the sense of Definition~\ref{defn:nnCSseed}) 
in the field of functions $\bbc({\rm Band}(k,\ell))$ and proved that the upper cluster algebra associated to this initial seed coincides with $\bbc[{\rm Band}(k,\ell)]$. We denote the resulting upper CS-cluster structure by $\mca_{\rm GSV} = \mca_{\rm GSV}(k,\ell)$.

We recall the initial seed for $\mca_{\rm GSV}$ given in {\sl loc. cit.}, Let $N:= (k-1)(\ell-1)$. For $i \in [N]$, we have an initial mutable variable $x_i := M_{R_i}^{C_i}$, a minor of size $N-i+1$. The column  set $C_i$ is the interval $[k+i,(k-1)\ell+1]$. The row set $R_i$ consists of the $N-i+1$ largest elements of $\{j \in [(k-1)\ell] \colon j \neq 1 \mod \ell\}$. The unique special variable is the largest of these minors, namely $x_1$.

The $2\ell+k-1$ frozen variables for $\mca_{\rm GSV}$ are the $2\ell$ entries $M_{i,i}$ and $M_{i,i+k}$ on outer diagonals, and also the regular functions $(-1)^{i(\ell-i)}c_s(A,B)$. The latter variables are coefficient string variables. We denote the tropical semifield in these variables by $\bbp_{\rm GSV}$. The initial exchange relations are encoded by a certain quiver which we do not spell out here, cf.~\cite[Figure 2]{GSVStaircase}.  


\begin{thm}[{\cite[Theorem 5.1]{GSVComplete}}]\label{thm:GSV}
The upper generalized  cluster algebra $\mca^{\rm up}_{\rm GSV}(k,\ell)$ coincides with the coordinate ring $\bbc[{\rm Band}(k,\ell)]$. 
\end{thm}

We emphasize that above theorem only holds when $k < \ell$: no candidate seed is given in the $k \geq \ell$ cases. 

We now describe a quasi-isomorphism of $\mca_{\rm GSV}(k,\ell)$ with $\mca_{\rm cyc}(k,\ell)$. Let $\mcf_{>0}$ denote the algebra of subtraction-free rational functions in the matrix minors $M_{R_I}^{C_I}$ defined above, with coefficients in $\bbp_{\rm GSV}$. We let $\ov{x}_i \in \mca_{\rm cyc}$ denote the initial extended cluster variable and $\ov{Z}_i \in \mca_{\rm cyc}$ the stable variables in the exchange relation for $\ov{x}_1$.

Consider the ambient semifields $\mcf_{>0, {\rm GSV}}$ and 
$\mcf_{>0, {\rm cyc}}$ corresponding to $\mca_{\rm GSV}$ and $\mca_{\rm cyc}$. 

We declare a semifield map $\psi^* \colon \mcf_{>0,  \text{ GSV}} \to \mcf_{>0, \textnormal{ cyc}}$ 
by the following formulas:
\begin{align}\label{eq:abstractGSVmap}
\psi^*(x_i) &= \ov{x}_i\prod_{b=i}^{N-1} \ov{x}_{N+(b+k \text{ mod } \ell)}, \, 
\psi^*(M_{i,i}) = \psi^*(M_{i-k-1,i-1}) =  \ov{x}_{N+(i-1 \text{ mod } \ell)}, \text{ and }\\
\psi^*(\tilde{c}_s) &= Z_s\prod_{i=1}^\ell \ov{x}_{N+i}.\label{eq:abstractGSVmapii}
\end{align}


\begin{lem}
The semifield map $\psi^*$ is a quasi-homomorphism from $\mca_{\rm GSV}$ to $\mca_{\rm cyc}$. 
\end{lem}

\begin{proof}
We need to show that the image of the initial seed for $\mca_{\rm GSV}$ under $\psi^*$ is $\sim$ to $\Sigma_{\rm cyc}(k,\ell)$. From the formulas, $\psi^*(\bbp_{\rm GSV}) \subset \bbp_{\rm cyc}$, so these two seeds are defined over the same coefficient group. We also that i) from Definition~\ref{defn:seedorbit} is satisfied. It remains to check~ii).  

By direct inspection, the initial exchange matrices $B_{\rm cyc}$ and $B_{\rm GSV}$ coincide (cf.~\cite[Figure 2]{GSVStaircase}). This says that ii) is satisfies if we set all frozen variables equal to 1. 


We first treat the nonspecial variables, i.e. the cases $i \geq 2$. In this cases we can only have one nontrivial Laurent monomial $\hat{y}_{1}^{(1)} = \hat{y}_{i}$. In the ``typical'' calculation, we have that 
$\hat{y}_i = \frac{x_{i-\ell+1}x_{i-1}x_{\ell+i}}{x_{i-\ell}x_{i+1}x_{\ell+i-1}} \in \mca_{\rm GSV}$. Applying $\psi^*$, all of the contributions of the $\ov{x}_{N+(b+k)}$'s cancel. And this matches $\hat{y}_i \in \mca_{\rm cyc}$, since $I_i$ is not adjacent to any frozens unless $i \in [N-\ell+1,N]$. The degenerate cases $i \in [\ell-1]$ and $i \in [N-\ell+1,N]$ are easily treated. 


Now we treat the case $i=1$. Rather than working with the exchange monomials, we directly compute the image of the exchange polynomial for $x_1$ under $\psi^*$. The exchange polynomial for the special variable looks different depending on whether $k=2$ or $k>2$; we will focus on the $k>2$ case for simplicity. In this case, the exchange polynomial for $\varphi_1 \in \mca_{\rm GSV}$ looks like 
$$Z_{1, \text{GSV}} := \prod_{i=1}^{\ell}M_{ii}\prod_{i=2}^{\ell}M_{i,i+k}^{k-1}\varphi_{\ell+1}^k+\sum_{s=1}^{k-1}\tilde{c}_s \prod_{i=2}^{k-1}M_{i,i+k}^{k-1-s}\varphi_{\ell+1}^{k-s}\varphi_2^s+M_{1,k+1}\Delta_2^k.$$
For brevity, set $\bb := \ov{x}_{N+(k+1 \mod \ell)}\prod_{b=2}^{N-1}\ov{x}_{N+(b+k \mod \ell)}^k \in \bbp_{\rm cyc}$. When we apply $\psi^*$ to the preceding exchange polynomial, the expression simplifies considerably:
$$ \psi^*(Z_{1, \text{GSV}}) = \bb\sum z_s \ov{x}_2^s\ov{x}_{\ell}^{k-s} = \bb Z_{1, {\rm cyc}}.$$
Condition ii) when $i=1$ immediately follows. 
\end{proof}

By the lemma, we have a composite map $\mca^\circ_{\rm GSV} \to \mca^\circ_{\rm cyc} \to \mca_n(k,\ell)$, with the second of these maps the specialization map \eqref{eq:trigcoeffs}. 

We would now like to claim that the above composite map is a ``familiar'' map. Most saliently, we would like to claim that $\psi^*$ is the pullback of a map of varieties $\mcd^\circ_n(k,\ell) \to {\rm Band}^\circ(k,\ell)$, because this would allow us to prove \eqref{eq:hardcontain} in these cases. The constructions below work both for $\mcd$ and for the Zariski-open subset $\mcd^\circ$, but quasi-isomorphisms most naturally work with localized coefficients. 

\begin{defn}
Denote by $\phi$ the regular map
\begin{equation}\label{eg:GrassToBand}
\phi \colon \Gr(k,n) \to {\rm Band}(k,n) \text{ sending } X \mapsto (\Delta_{[i,i+k]\setminus j}(X)),
\end{equation}
where we treat indices of Pl\"ucker coordinates modulo $n$ and interpret a Pl\"ucker coordinate as $0 \in \bbc$ when its  index set is not a $k$-subset. When $\ell|n$, $\phi$ restricts to a map (also denote) $\phi \, \colon \, \mcd_n(k,\ell) \subset \Gr(k,n)^{\rho^\ell} \to {\rm Band}(k,\ell)$ (for any $n$). Let $\tau \in {\rm Aut(Band)}(k,\ell)$ be the {\sl transpose} automorphism on band matrices, defined by $\tau(M)_{ij} = M_{j,i+k}$. Finally, let $\phi' := \tau \circ \phi \circ \rho^{-k}$, which we can view as a regular map ${\rm Band}^\circ(k,\ell) \to \mcd_n(k,\ell)$. 
\end{defn}

The map $\phi$ was studied previously in \cite[Appendix A]{FraserQH} (cf.~also \cite{MSTwist}). It has the following interesting property. Suppose that $M_{I}^J$ is a minor whose row set is an interval $I \subset [n-k]$. Then the minor $M_I^J(\phi(X))$ factors as a product of Pl\"ucker coordinates of $X$. A specific formula is given in \cite[Lemma 10.2]{FraserQH}. This algebraic relationship between row-solid minors of band matrices and Pl\"ucker coordinates was exploited to construct a correspondence between (ordinary, FZ-) cluster structures on band matrices and Grassmannians. We are proposing here that a similar correspondence holds between $\mca_{\rm GSV}$ and $\mca_n(k,\ell)$.

\begin{conj}\label{conj:GSV}
Suppose that $k<\ell$ (so that Theorem~\ref{thm:GSV} applies). Then the pullback $(\phi')^*$ is the composite map 
$\mca^\circ_{\rm GSV} \overset{\psi^*}{\to} \mca^\circ_{\rm cyc} \to \mca^\circ_n(k,\ell)$.
\end{conj}

In particular, it would follow that the composite map is (the pullback of) a regular map $\mcd^\circ_n(k,\ell) \to {\rm Band}^\circ_n(k,\ell)$. Since the composite map sends (quasi)-cluster variables to (quasi)-cluster variables, it would follow that each cluster variable in $\mca_n(k,\ell)$ is a regular function on $\mcd^\circ_n(k,\ell)$, proving the inclusion $\mca^\circ_n(k,\ell) \subset \bbc[\mcd_n^\circ(k,\ell)]$. And since the intersection of Laurent rings maps in to the intersection of Laurent rings under a coefficient specialization, we would deduce the inclusion $\bbc[\mcd_n(k,\ell)] \subset \mca^{\circ \text{ up}}_n(k,\ell)$.

To check that Conjecture~\ref{conj:GSV} holds, it suffices to check that $(\phi')^*$ satisfies the defining equations \eqref{eq:abstractGSVmap} and \eqref{eq:abstractGSVmapii}. 

\begin{prop}\label{prop:GSV}
The map $(\phi')^*$ satisfies \eqref{eq:abstractGSVmap} for any value of $k$. It satisfies \eqref{eq:abstractGSVmapii} when $k=2$ or $3$ (and conjecturally, for all $k$). 
\end{prop}

After localizing at frozen variables, the inclusions \eqref{eq:hardcontain} now follow when $k=3$ and $\ell>3$. (They also hold when $k=2$, but we prove the stronger statement without localizations via folding below.) They also hold when $k=3$ and $\ell=2,3$ by folding arguments (see below). This establishes the localized version of \eqref{eq:hardcontain} in all $k=3$ examples.

\begin{rmk}\label{rmk:GSV}
To check \eqref{eq:abstractGSVmapii}, we see two possible approaches: the direct approach is to check that the coefficients of the characteristic polynomial $\det(tI+B^{-1}A)$ \eqref{eq:characteristicpoly} are {\sl independent} of $X \in \mcd_n(k,\ell)$, with $A,B$  associated to $M = \phi'(X)$ as in \eqref{eq:characteristicpolyXY}. Thus Conjecture~\ref{conj:GSV} is an isospectrality conjecture as alluded to above. The corresponding eigenvalues are the roots of the $\SL_k$ orbifold polynomial, which are described in Remark~\ref{rmk:qbinomialthm}. We are currently able to prove this isospectrality when $k=2$. 

A second possible approach is to check that applying $(\phi')^*$ commutes with mutation at the special variables. That is, we are claim that when $k<\ell$, then the $\SL_k$-higher orbifold Ptolemy relation is a specialization of the ``master identity'' \cite[Equation (3.7)]{GSVStaircase}. By algebraic independence of the initial cluster variables, this would imply that the coefficients of $Z_{1,\text{GSV}}$ must be sent to those for $Z_{1,\text{\rm cyc}}$. Our proof of Proposition~\ref{prop:GSV} employs this approach. The missing detail when $k>3$ is a sufficiently explicit definition of the result at mutating at the special variable in $\mca_{\rm GSV}$. 
\end{rmk}

\begin{proof}
We have $\phi^* = \psi^*$ when evaluated on frozen matrix entries directly from the definitions. Recall that $\ov{x}_i = \Delta_I$ is a Pl\"ucker coordinate defined in Definition~\ref{defn:initPluckers}. For mutable variables $x_i = M_{R_i}^{C_i} \in \mca_{\rm GSV}$, we have
\begin{align*}
\phi^*(x_i) &= (\rho^{-k})^*\phi^*\tau^*(M_{R_{N-i+1}}^{C_{N-i+1}}) \\
&= (\rho^{-k})^*\phi^*(M_{C_{N-i+1}}^{\rho^k(R_{N-i+1})}) \\
&= \Delta_{I_i}\prod_{b=i}^{N-1} \Delta_{[N+b+1,N+b+k]}.
\end{align*}
In the last line we have used \cite[Lemma 10.2]{FraserQH} and done some bookkeeping. Note that this calculation matches \eqref{eq:abstractGSVmap} (and indeed, was how we guessed the formulas \eqref{eq:abstractGSVmap}).  

In the special cases $k \in \{2,3\}$, we check the second assertion \eqref{eq:abstractGSVmapii} by the second method in Remark~\ref{rmk:GSV}. We focus on the $k=3$ case, which is more complicated. We gave an explicit description of the mutation at the special variable $I_1 \in \mca_n(k,\ell)$ in Definition~\ref{defn:Weylsidentity}. We need a correspondingly explicit description of the mutation at the special variable $x_1\in \mca_{\rm GSV}$. A complicated description, valid in a broader setting, is given in \cite[Equation (3.11)]{GSVStaircase}. For our case of interest (i.e., for band matrices when $k=3$), we have the simpler expression as a $2 \times 2$ determinant in matrix minors:\footnote{We thank Michael Gekhtman for explaining this to us.} 

\begin{equation}\label{eq:MishaPC}
z:= M_{[2,2\ell+1] \setminus \{2,\ell+1,2\ell+1\}}^{[5,2\ell+1]} M_{[2,2\ell+1] \setminus \{2,\ell+2,2\ell+1\}}^{[5,2\ell+1]}
-
M_{[2,2\ell+1] \setminus \{2,\ell+1,\ell+2\}}^{[5,2\ell+1]} M_{[2,2\ell+1] \setminus \{\ell+1,\ell+2,2\ell+1\}}^{[5,2\ell+1]}.
\end{equation}
By a similar calculation as above, we have 
\begin{align*}
\zeta^*(z) &\propto \Delta_{2,\ell+1,2\ell+1}\Delta_{2,\ell+2,2\ell+1}- 
\Delta_{2,\ell+1,\ell+2}\Delta_{\ell+1,\ell+2,2\ell+1}\\ &= 
\Delta_{2,\ell+1,2\ell+1}\Delta_{\ell+2,2\ell+2,3\ell+1}- 
\Delta_{2,\ell+1,\ell+2}\Delta_{2\ell+1,2\ell+2,3\ell+1}
\end{align*}
in agreement with our explicit description of mutation at $I_1$ given above. 
\end{proof}

\begin{rmk}
While our results require that $p \geq k$, the requirement $k < \ell$ is not natural from the perspective of cyclic symmetry loci. Thus, one might be able to ``reverse engineer'' a cluster structure on $\bbc[{\rm Band(k,\ell)}]$ in the $k \geq \ell$ cases from the cluster structure on~$\bbc[\mcd_n(k,\ell)]$.
\end{rmk}

\subsection{Quantum affine algebras at roots of unity}
Consider $\eps \in \bbc^*$ such that $\eps^2$ is an $\ell$th root of unity. Let $\mcc_{\epsilon^\bbz}$ be the monoidal category associated to $\mathfrak{g} = \mathfrak{sl}_k$ as in \cite[Section 3.3]{Gleitz}. It is a monoidal subcategory of the category of representations of the quantum affine algebra when the deformation parameter $q$ has been specialized to $\eps$. In the $k=2$ cases, and in the special case $k=3$ and $\ell=2$, Gletiz showed that the (complexified) Grothendieck ring $K_0(\mcc_{\epsilon^\bbz})$ is a generalized cluster algebra structure of finite Dynkin types $C_{\ell-1}$ and $G_2$, respectively \cite[Theorem 4.1 and 5.4]{Gleitz}. Moreover, a conjectural 
initial seed (consisting of classes of simple modules) for all $k=3$ cases is given in \cite[Conjecture 5.5]{Gleitz}.

We make the following conjecture, which agrees with Gleitz's results and conjectures and extends them to arbitary $k$ and $\ell$.
\begin{conj}\label{conj:quantumaffine}
The Grothendieck group $K_0(\mcc_{\epsilon^\bbz})$ admits a generalized upper cluster algebra structure in which each cluster monomial is the class of a simple module. An initial seed for this cluster algebra can be obtained by setting all frozen variables $\ov{x}_{N+i} = 1$ in Definition~\ref{defn:abstractCSseed} and identifying the mutable variables $\ov{x}_i$ and coefficient string variables $z_s$ with the classes of appropriate simple modules. 
\end{conj}

We can make this more explicit: via the correspondence between simple modules and Pl\"ucker coordinates given in~\cite[Section 3]{CDFL}, one has a conjectural set of simple modules serving as the initial cluster. 

We are currently investigating this conjecture with Michael Gekhtman and Kurt Trampel, who independently made a related conjecture. Since the cluster algebras $\mca_{\rm cyc}(k,\ell)$ and $\mca_{\rm GSV}(k,\ell)$ are quasi-isomorphic, the above conjecture is insensitive to whether we work with $\mca_{\rm cyc}$ versus $\mca_{\rm GSV}$.

This conjecture is compatible with the $k=2$ and $k=3, \ell=2$ cases covered in \cite{Gleitz}. (See discussion of finite type cyclic symmetry loci below.) It is also compatible with the conjectural description of $k=3$ cases in {\sl loc. cit.} Specifically, consider the 
the $\rho^\ell$-invariant collection $\mcc'$ containing the Pl\"ucker coordinates 
$$\{1,\ell+1,2\ell+1\} \text{ and } \{\ell,\ell+1,\ell+2+x \colon \, x \in [\ell-1]\} \text{ and } \{\ell+1,2\ell-y,2\ell+1-y \colon \, y \in [0,\ell-3]\}.$$
One may check that this $\ell$-optimal collection arises from the construction \eqref{eq:chaintocup}, and can be obtained from our initial seed $\Sigma_n(k,\ell)$ by a sequence of mutations. Thus, we get a seed $\Sigma(\mcc') \subset \mca_n(k,\ell)$. Under a slight change of conventions, the correspondence \cite{CDFL} sends the Pl\"ucker coordinates in $\mcc'$ to the initial simple modules in {\sl loc. cit.}, and the exchange matrices coincide.

We clarify that the value of $p$, (or equivalently, of $n$), plays no role in Conjecture~\ref{conj:quantumaffine}.

\section{Folding along cyclic symmetry}\label{secn:folding}
We state in this section certain conjectures relating the CS- cluster algebra $\mca_n(k,\ell)$ to its ``unfolding'', i.e. the FZ-cluster algebra $\bbc[\Gr(k,n)]$. This was our starting perspective, which led us to the coefficient strings \eqref{eq:trigcoeffs}. We explain how these conjectures would imply~\eqref{eq:hardcontain}.



\begin{conj}\label{conj:uniformity}
One can label each vertex $v$ in the exchange graph of $\mca_{\rm cyc}(k,\ell)$ by an $\ell$-optimal cluster ${\bf x}(v) \subset \bbc[\Gr(k,n)]$, such that once we restrict functions to $\mcd_{p\ell}(k,\ell)$, 
${\bf x}(v)$ becomes an extended cluster in $\mca_{p\ell}(k,\ell)$, for any $p \geq k$. 

Any such partial cluster ${\bf x}(v)$ is contained an extended cluster   ${\bf x}^\dagger(v)\subset \bbc[\Gr(k,n)]$ with the following property . Once we restrict functions to $\mcd$, any $x \in {\bf x}^\dagger(v) \setminus {\bf x}(v)$ is linearly related to a monomial in in the elements of ${\bf x}(v)$.
\end{conj}

When we refer to this Conjecture~\ref{conj:uniformity} for a specific value of $p$, we are merely assuming that each cluster lifts to an $\ell$-cluster that satisfies the ${\bf x}^\dagger(v)$ condition. But Conjecture~\ref{conj:uniformity} asserts moreover that the lifting can be done uniformly in $p$.


We have not thoroughly tested either part of this conjecture. We state it because it is compatible with the evidence given in Lemma~\ref{lem:superfluous}, and is sufficient to prove the equality \eqref{eq:hardcontain}:

\begin{prop}
If Conjecture~\ref{conj:uniformity} holds for $k,\ell,p$, then the strengthened inclusions \eqref{eq:hardcontain} hold. 
\end{prop}

\begin{proof} The first part of Conjecture~\ref{conj:uniformity} implies that any cluster variable $x \in \mca_n(k,\ell)$ lifts to an element of $\bbc[\Gr(k,n)]$, hence is regular on $\mcd_n(k,\ell)$. Thus 
$\mca_n(k,\ell) \subset \bbc[\mcd_n(k,\ell)]$. For the opposite containment, by the Laurent phenomenon for $\bbc[\Gr(k,n)]$, any 
Pl\"ucker coordinate $\Delta_I$ is a Laurent polynomial in the elements of ${\bf x}^\dagger(v)$. Restricting this Laurent polynomial expression to $\mcd$ and using the assumption on ${\bf x}^\dagger(v)$, it follows that $\Delta_I$ is a Laurent polynomial in the elements of 
${\bf x}(v)$. \end{proof}

\begin{example}
Conjecture~\ref{conj:uniformity} holds when $k=2$. Every cluster lifts to a $\rho^\ell$-symmetrical partial triangulation of the $n$-gon, and any extension ${\bf x}^\dagger$ of this partial triangulation to a triangulation satisfies the second part of Conjecture~\ref{conj:uniformity} by Lemma~\ref{lem:superfluous}. A similar argument holds more generally when ${\bf x}(v)$ is of the form \eqref{eq:chaintocup}, but not every cluster has this form once $k>2$. 
\end{example}

Our next two sections briefly sketch how one {\sl might} prove Conjecture~\ref{conj:uniformity} in the special cases that $p \in \{k,k+1\}$. For these values of parameters, $\bbc[\Gr(k,n)]$ admits $\ell$-clusters (not merely $\ell$-optimal clusters), so the ${\bf x}^\dagger(v)$ condition is vacuous. The argument we sketch here rigorously works in the finite mutation type cases, which allows us to conclude \eqref{eq:hardcontain} as stated in Section~\ref{secn:clusters}.

\subsection{Unfolding when $p=k$}
In this case, $\mca_{p\ell}(k,\ell)$ is an FZ- cluster algebra (it is the right companion cluster algebra). Since the 
the ${\bf x}^\dagger(v)$ condition is vacuous in this case, what we need to prove is that every cluster in $\mca_{p\ell}(k,\ell)$ is the image of an $\ell$-cluster once we restrict functions to $\mcd$. 

To this end, we conjecture the following: every mutation in $\mca_{p\ell}(k,\ell)$ ``unfolds'' to a $\rho^\ell$-symmetrical sequence of mutations in $\bbc[|Gr(k,n)]$. It is easy to see that such a statement implies Conjecture~\ref{conj:uniformity} in the $p=k$ case. The difficulty in carrying out this argument is verifying that an arbitrary sequence of $\rho^\ell$-symmetrical mutations is well-defined. That is, one must verify that after performing a sequence of mutations at $\rho^\ell$-orbits, one cannot arrive at a a quiver that admits an arrow between two vertices in a $\rho^\ell$-orbit. This statement is purely about quiver mutation (not about algebra). We have checked it in the finite type and finite mutation type examples.

\subsection{Unfolding when $p=k+1$}
Now we explain a similar (conjectural) mechanism that would explain Conjecture~\ref{conj:uniformity} when $p=k+1$. It relies on the following lemma, which we think is of independent interest.

Let ${\bf q}_{p}$ denote the oriented $p$-cycle on vertices $u_1,\dots,u_p$. Let $\tau_p$ be the following permutation-mutation sequence (cf.~\cite[Section 7]{GoncharovShenDT}): $\tau_p = \mu_1,\circ \cdots \circ\mu_{p-1} \circ (u_{p-1},u_p) \circ \mu_{p-1} \circ \cdots \circ \mu_1$. It differs from the {\sl Donaldson-Thomas transformation} (also known as {\sl green-to-red sequence}) by a cyclic permutation of variables.  Let $\widetilde{{\bf q}_{p}}$ be the {\sl framed- and co-framed version} of this quiver: one adds to ${\bf q}_{p}$ the 
auxilliary vertices $u_i',u_i''$ and arrows $u'_i \to u_i$ and $u_i \to u''_i$ for all $i \in [p]$. For a vertex $v \in \widetilde{{\bf q}_p}$, let $x(v)$ be an indeterminate assigned to $v$, treated as an initial cluster variable. Let $x'(v) = \tau_p(x(v))$ be the image of this initial cluster variable under the permutation-mutation sequence $\tau_p$. 
\begin{lem}\label{lem:vanillaDT}
Let $a,b,c$ be indeterminates. Then under the specialization of initial variables $x(u_i) = a$, $x(u'_i) =b$, and $x(u''_i) = c$, the variable $x'(u_i)$ specializes to 
\begin{equation}\label{eq:CSfolding}
\frac{\sum_{i=1}^{p-1} b^ic^{p-1-i}}{a}.
\end{equation}
Let ${\bf q'}_p = (u'_i,u'_{i-1},\dots,u'_{i-p+1}) \circ (u''_i,u''_{i+1},\dots,u'_{i+p-1}) \circ \tau_r(\widetilde{{\bf q}_r})$ be the quiver obtained by performing $\tau_r$ followed by two cyclic permutations of variables. 
Then ${\bf q'}_p$ contains 
${\bf q}_p$ as an induced subquiver, an in addition has arrows $u''_i \to u_i$ and $u_i \to u'_i$  for all $i$, as well as arrows $u'_i \to u''_j$ for all $j \neq i+1$. 
\end{lem}

This lemma says that a CS-exchange relation of degree $p-1$ and coefficient string $(c_0,\dots,c_{p-1}) = (1,1,\dots,1)$ can be simulated by performing a sequence of FZ- mutations in the cyclically symmetrical quiver $\widetilde{{\bf q}_p}$. The description of ${\bf q'}_p$ says that the $B$-matrices change as expected: one reverses the 2-path $u'_i \to u_i \to u''_i$ and adds a total of $p-1$ arrows directed from variables $u'_i$ to variables $u''_j$.

\begin{proof} By Section~\cite[Section 7]{GoncharovShenDT}, $\tau_p$ is an involution, and the composition of $\tau_p$ with a cyclic permutation of variables has the effect of reversing all arrows $u'_i \to u_i$ while preserving the subquiver ${\bf q}_r$. The inverse transformation therefore would correspond to performing the inverse cyclic permutation followed by the same permutation-mutation sequence. Performing $\tau_p$ followed by the cyclic permutations in the statement of the lemma consequently preserves ${\bf q}_p$ while reversing the directions of the arrows $u'_i \to u_i$ and $u_i \to u''_i$.  What needs to be checked is the statement about arrows between the vertices of the form $\{u'_i\} \cup \{u''_j\}$. From the rotational symmetry, it suffices to study arrows involving $u'_1$ and/or $u''_1$. One can see this by doing examples in the quiver mutation applet: mutation $\mu_{p-1} \circ \cdots \circ \mu_1$ creates arrows $u'_1 \to u''_j$ for $j \in [p-1]$ and does not create any arrows $u'_1$ to $u'_i$ or $u''_1$ to $u''_j$. The remainder of the cluster transformation $\tau_p$ does not create any arrows involving $u'_1$. 

The formula for the variables $x'(u_i)$ is argued in a similar fashion. By rotational symmetry, it suffices to check the formula for $x'(u_p)$. One has $x'(u_p) = \mu_{p-1} \circ \cdots \circ \mu_1(x(u_{p-1}))$. By induction on $p$ we have that $\mu_{p-2} \circ \cdots \circ \mu_1(x(u_{p-2})$ specializes to $\sum_{j=0}^{p-2}b^jc^{p-2-j}$. By the way the quiver mutates, one sees that 
$$\mu_{p-1} \circ \cdots \circ \mu_1(x(u_{p-1})) = 
\frac{1}{x_{p-1}}(\prod_{j \leq p-1}x(u'_j)+x(u''_{p-1})\mu_{p-2} \circ \cdots \circ \mu_1(x(u_{p-2})),$$ which specializes to 
$\frac{1}{a}(b^{p-1}+c\sum_{j=0}^{p-2}b^jc^{p-2-j})$ establishing \eqref{eq:CSfolding}.
\end{proof}


In the situation $p=k+1$, the $\rho^\ell$-orbit of the special variable forms an oriented $k+1$-cycle in the extended quiver $\tQ(\ov{\mcc_n(k,\ell)})$. (This follows from the weak separation combinatorics: 
the $\rho^\ell$-orbit $\{\rho^{a\ell}(I_1) \colon a \in [k+1]\}$ of the special variable forms a black clique in the plabic tiling for $\mcc_{n}(k,\ell)$, because these subsets exhaust all $k$-element subsets drawn from the $k+1$-element subset $\{1,1+\ell,\dots,n-\ell+1\}$.)

With Lemma~\ref{lem:vanillaDT} in hand, we conjecture that any sequence of mutations in $\mca_{(k+1)\ell}(k,\ell)$ unfolds to a sequence of mutations in $\bbc[\Gr(k,n)]$. A mutation at a non-special variable should unfold to a mutation at the corresponding $\rho^\ell$-orbit in $\bbc[\Gr(k,n)]$. And mutation at the special variable should unfold to the mutation sequence $\tau_p$. The difficulty, as in the $p=k$ case, is arguing that these unfoldings remain well-defined.

\begin{rmk}[Unfolding when $p \geq k+2$]\label{rmk:extendmethod}
The above mechanism could be extended to the $p \geq k+2$ cases as follows. First, unfold the initial cluster $\mcc \subset \mca_n(k,\ell)$ to an $\ell$-cluster $\ov{\mcc}$ and extend it to a Grassmannian cluster $\mcc^\dagger$. By Lemma~\ref{lem:superfluous}, the superfluous variables (i.e. the elements of $\mcc^\dagger \setminus \ov{\mcc}$) are linearly related to elements of $\mcc$ when thought of as functions on $\mcd$. When one wishes to mutate at the special variable in the cluster algebra $\mca_n(k,\ell)$, one should restrict the cluster $\mcc^\dagger$ to the subcluster $\mcc^{\dagger}_{\rm sub} \subset \mcc^{\dagger}$ consisting of the superfluous variables together with the $\rho^\ell$-orbit of the special variable. Conjecturally, this subcluster would admit a cluster DT transformation, and one can simulate mutation at the special variable by  performing this mutation sequence. This has the effect of reversing the arrows of $\mcc^\dagger$ which point into and out of $\mcc^{\dagger}_{\rm sub}$.
The interpretation of the orbifold Ptolemy relation as a twist map, and the relations between twist maps and DT transformations, 
are further supporting evidence that this approach might work (cf.~Remark~\ref{rmk:twistmap}). 
\end{rmk}

\section{Gallery of examples}\label{secn:examples}
\subsection{Finite type examples}
We list the finite type Grassmannians $\Gr(k,n)$ admitting nontrivial factorizations $n = p\ell$ in the following table. 
 
\begin{center}
\begin{tabular}{|ccccc|}\hline
Locus &Type($\Gr(k,n)$)  &Type($\mca_n(k,\ell)$)  & \# $\ell$-collections &  \# $\ell$-cluster variables\\
\hline 
$\mcd_n(2,\ell) $&  $A_{n-3}$  & $C_{\ell-1}$ &$\binom{2\ell-2}{\ell-1}$ & $\ell(\ell-1)$ \\ \hline
$\mcd_6(3,2)$ & $D_4$& $G_2$ & 8 & 8\\ \hline
$\mcd_8(3,2)$ & $E_8$& $G_2$ & 8 & 8\\ \hline
$\mcd_6(3,3)$ & $D_4$ & N/A & 6 & 6\\ \hline
$\mcd_8(3,4)$ & $E_8$ & N/A & 88 & 28 \\ \hline
\end{tabular}
\end{center}

On a case by case basis, one sees that the position of $\mca_{p\ell}(k,\ell)$ in the trichotomy of finite cluster vs. finite mutation but infinite cluster type vs. infinite type cluster algebras matches that of the Grassmannian cluster algebra $\bbc[\Gr(k,k\ell)]$.

\begin{thm}\label{thm:finiteclusters}
When $\mca_{p\ell}(k,\ell)$ has finite type (and $p \geq k$), we have equality of algebras $\mca_{p\ell}(k,\ell) = \bbc[\mcd_{p\ell}(k,\ell)]$.  
\end{thm}

The proof of this theorem is given in Examples~\ref{eg:kis2} and \ref{eg:kis3}. 

The cases of $\mcd_6(3,3)$ and $\mcd_8(3,4)$ are $p<k$ cases, to whom we have associated no cluster algebra. In Examples~\ref{eg:firstprestable} and \ref{eg:secondprestable}, we discuss the structure of the set of $\ell$-cluster monomials in these two cases. Ignoring a technicality (we do not prove the linear independence of cluster monomials in the case of 
$\mcd_8(3,4)$), we show that the $\ell$-cluster monomials are a basis for $\bbc[\mcd_n(k,\ell)]$, and discuss the resulting ``cluster complexes.''  

\begin{rmk}
The fourth column of the above table can often be computed using the cyclic sieving result \cite{EuFu} for the action of the Coxeter transformation on cluster complexes of finite type. This result allows one to count the 4- and 2-clusters in $\Gr(3,8)$, as well as the 3-clusters in $\Gr(3,6)$. One {\sl cannot} use this result to count the 2-clusters in $\Gr(3,6)$ because $\rho^2$ is not a power of the Coxeter transformation.  
\end{rmk}

\begin{rmk}
The folding of Dynkin diagrams $A_{2\ell-3} \to C_{\ell-1}$ and $D_4 \to G_2$ are classical, and the quotient maps $\bbc[\Gr(2,2\ell)] \twoheadrightarrow \bbc[\mcd_{2\ell}(2,\ell)]$ and $\bbc[\Gr(3,6)] \mapsto \bbc[\mcd_6(3,3)]$ are cluster-algebraic incarnations of these foldings. The first of these appears already in \cite[Section 12.3]{CAII}, via the model of type $C$ cluster algebras in terms of centrally symmetric triangulations.

The cluster algebra $\mca_{9}(3,3)$ (and more generally $\mca_{3p}(3,3)$ with $p \geq 3$) is of finite mutation type $G_2^{(1,1)}$ (cf.~\cite[Figure 1.1]{FSTskew}). The quotient map $\bbc[\Gr(3,9)] \to \bbc[\mcd_9(3,3)]$ reflects a known folding of extended affine root systems $E_8^{(1,1)} \to G_2^{(1,1)}$ \cite[Table 6.3]{FSTskew}.  The cluster algebras $\mca_{2p}(4,2)$ ($p \geq 4$) also have finite mutation type, and their cluster type is an {\sl exceptional $s$-block} in the terminology \cite[Figure 1.1]{FSTorb}. It is not hard to check that $\mca_{n}(k,\ell)$ has infinite mutation type outside of these listed examples. 
\end{rmk}

\begin{example}[$k=2$ examples]\label{eg:kis2}
Let $n = p\ell$ with $p \geq 2$. Then $\mca_{n}(2,\ell)$ has finite Dynkin type $C_{\ell-1}$. Cluster algebras with this Dynkin type have $\ell(\ell-1)$ many mutable variables and $\binom {2(\ell-1)}{\ell-1}$ many clusters. In our case, we have mutable variables $(\Delta_{ij})_{i \in [\ell], \, j-i \in [2,\ell+1]}$ and frozen variables $(\Delta_{i,i+1})_{i \in [\ell]}$. These (and their $\rho^\ell$-shifts) are {\sl all} of the $\ell$-cluster variables in $\bbc[\Gr(2,n)]$. Each cluster lifts to an $\ell$-collection in this case (cf.~Conjecture~\ref{conj:uniformity}), so that $\mca_{n}(2,\ell) \subset \bbc[\mcd_{n}(2,\ell)]$. The cluster algebra is graded with all extended cluster variables in degree one. 

We establish the inclusion $\bbc[\mcd_n(2,\ell)] \subset \mca_n(2,\ell)$. by showing that the cluster monomials in $\mca_{n}(2,\ell)$ span. Since cluster monomials are linearly independent, this amounts to verifying that the number of cluster monomials of degree $d$ coincides with $\dim \bbc[\mcd_n(2,\ell)]_{(d)}$. By the multiplicativity of the Hilbert function, the latter number is given by $(\dim \bbc[\bbp^{\ell-1} ]_{(d)})^2 = \binom{d+\ell-1}{\ell-1}^2$. 

Now we count the cluster monomials of degree~$d$. By \cite[Proposition 1]{Simion}, the number of $a$-subsets of compatible mutable variables in a cluster algebra of this Dynkin type is given by $\binom{\ell-1}a\binom{\ell-1+a}{\ell-1}$. For a fixed such $a$-subset, there are $\binom{d+\ell-1}{d-a}$ many cluster monomials using exactly these mutable variables. 
Using the enumerations of cluster variables and clusters given above, the number of cluster monomials of degree $d$ is given by 
\begin{align*}
\sum_{a\geq 0}\binom{\ell-1}a \binom{\ell-1+a}{\ell-1} \binom{d+\ell-1}{d-a}&= \sum_{a\geq 0}\binom{\ell-1}{a}\binom{d-\ell+1}{\ell-1} \binom{d}{d-a} \\
&= \binom{d+\ell-1}{\ell-1}\sum_{a\geq 0}\binom{\ell-1}{a}\binom{d}{d-a}  \\
&= \binom{d+\ell-1}{\ell-1}^2 \\
&= \dim \bbc[\mcd]_{(d)}
\end{align*}
The first equality relies on the binomial coefficient identity $\binom{n-h}{k}\binom{n}{h} = \binom{n}k\binom{n-k}{h}$, and the third equality relies on the Vandermonde convolution identity. 
\end{example}



\begin{rmk}An alternative approach to proving the equality of Hilbert functions is to establish directly that $\bbc[\mcd] \subset \mca$ by expressing each Pl\"ucker coordinate as a polynomial in extended cluster variables, and appealing to general results that cluster monomials are a basis for any cluster algebra of finite type. We are not aware of a reference for this fact for CS-cluster algebras, but it is well-known for FZ-algebras. 
\end{rmk}

\begin{example}[$k=3,\ell=2$]\label{eg:kis3}
Next we consider the cases $\mca_{2p}(3,2)$ for $p \geq 3$, which is a CS-cluster algebra of Dynkin type $G_2$. Let $X = \Delta_{124}\Delta_{356}-\Delta_{123}\Delta_{456}$ and $Y = \Delta_{235}\Delta_{146}-\Delta_{234}\Delta_{156}$ be the two non-Pl\"ucker cluster variables in $\bbc[\Gr(3,6)]$. Then $\mca_{2p}(3,2) \subset \bbc[\mcd]$ has 8 mutable variables (6 in degree one and 2 in degree two), 8 clusters, and 2 frozen variables. Specifically, the cluster variables are 
$246$, $124$, $X$, $356$, $135$, $235$, $Y$, and $236$, enumerated so that adjacent cluster variables form a generalized cluster. By direct enumeration, these (and their $\rho^2$-shifts) are a complete list of $2$-cluster variables in $\bbc[\Gr(3,6)]$. In particular $\mca_{n}(3,2) \subset \bbc[\mcd_{n}(3,2)]$. 

We argue the opposite containment as in the previous example. We have $\dim\bbc[\mcd_n(3,2)]_{(d)} =(\dim \bbc[\bbp^1]_{(d)})^3 =(d+1)^3$. So the desired equality amounts to the equality of generating functions
$$\sum_{d \geq 0}(d+1)^3 x^d = \frac{1}{(1-x)^2}+ \frac{6x}{(1-x)^3}+\frac{2x^2}{(1-x)^2(1-x^2)}+\frac{4x^2}{(1-x)^4}+\frac{4x^3}{(1-x)^3(1-x^2)},$$
which can be directly verified.
\end{example}

\begin{example}\label{eg:firstprestable}
Consider the cyclic symmetry locus $\Gr(3,6)^{\rho^3}$, a finite type Grassmannian corresponding to a $p<k$ case. The Hilbert function is $\dim \bbc[\mcd_6(3,3)]_{(d)} = \dim \bbc[\Gr(1,3)]_{(d)} \cdot \dim \bbc[\Gr(2,3)]_{(d)}  = \binom{d+2}2^2$. There are exactly twelve 3-cluster variables, grouped into 6 orbits with orbit representatives
$124$, $125$, $235$, $236$, $346$, $134$. Consecutive terms in this list form 3-clusters (upon adding their $\rho^3$-shifts).
Moreover, these are {\sl all} of the 3-clusters in $\bbc[\Gr(3,6)]$. Folding the quiver naively, the mutable part is an oriented 2-cycle. If one were to cancel this oriented cycle, we would get an $A_1 \coprod A_1$ quiver, which has 4 clusters, not six, and does not seem to correctly describe the algebra $\bbc[\mcd_6(3,3)]$. 

Our {\sl proposal} is that one should treat these six 3-clusters as ``clusters'' for $\mcd_{6}(3,3)$, whose exchange graph is a hexagon, 
even though it is not yet apparent what cluster algebra formalism describes the mutations between these clusters. Taking this idea seriously, we can again compute the number of ``cluster monomials'' of degree $d$, obtaining the number
\begin{equation}\label{eq:binoms}
6\binom{d+2}4+6\binom{d+2}3+\binom{d+2}2.
\end{equation}
One can check that \eqref{eq:binoms} simplifies to $\binom {d+2}2^2$, matching the Hilbert series, and justifying our proposal. 
\end{example}

\begin{example}\label{eg:secondprestable}
We study the finite cluster type Grassmannian $\Gr(3,8)$ and its locus $\mcd_8(3,4)$, which is another $p<k$ case. Using cylic sieving, there are 88 4-clusters in $\bbc[\Gr(3,8)]$ connected to each other by $\rho^4$ symmetrical mutations in $\bbc[\Gr(3,8)]$. To further justify the proposal in Example~\ref{eg:firstprestable}, we advocate that one should treat these as ``clusters'' in $\bbc[\mcd_8(3,4)]$, together with the corresponding exchange graph and cluster complex, even though it is not clear what cluster algebraic formalism describes the mutations connecting them. 

The Hilbert function is $\dim \bbc[\mcd_8(3,4)]_{(d)} = \dim \bbc[\Gr(2,4)]_{(d)} \dim \bbc[\Gr(1,4)]_{(d)} = \frac{(d+1)(d+3)(d+2)^2}{12} \binom{d+3}3$. To count the cluster monomials of degree $d$, we first count the subsets of pairwise compatible cluster variables on a computer, keeping track of their degree. There are $24$ mutable variables (20 in degree one and 4 in degree two). 
There are $112$ pairs of compatible mutable variables (86 pairs in degree $1^2$, 24 in $1^12^1$, and 2 in degree $2^2$). 
There are $156$ triples of compatible mutable variables (124 triplets in degree $1^3$, 44 in $1^22^1$, and 8 in degree $1^12^2$). Finally, there are $88$ clusters (56 of these in degree $1^4$, $24$ in degree $1^32^1$, and $8$ in degree $1^22^2$).

Thus, the generating function for cluster monomials of degree $d$ is given by 
\begin{align*}
&\frac{1}{(1-x)^4}+ 
\frac{20x}{(1-x)^5}+\frac{4x^2}{(1-x)^4(1-x^2)}
+\frac{86x^2}{(1-x)^6}+\frac{24x^3}{(1-x)^5(1-x^2)}+\frac{2x^4}{(1-x)^4(1-x^2)^2} \\
+&\frac{124x^3}{(1-x)^7}+\frac{44x^4}{(1-x)^6(1-x^2)}+\frac{8x^5}{(1-x)^5(1-x^2)^2}
+\frac{56x^4}{(1-x)^8}+\frac{24x^5}{(1-x)^7(1-x^2)}+\frac{8x^6}{(1-x)^6(1-x^2)^2},
\end{align*}
and a calculation establishes that this is the rational function $\sum_d \dim \bbc[\mcd_8(3,4)]_{(d)}x^d$. 
\end{example}

\subsection{Failure of $\rho$-equivariance}
The cyclic shift map on $\Gr(k,n)$ preserves the $\ell$-fixed locus and the TNN locus $\Gr(k,n)_{\geq 0}$, hence determines an automorphism of $\mcd_n(k,\ell)$.

We now show by example that $\rho$ need not be a cluster automorphism of $\mca_n(k,\ell)$. That is, the seeds $\Sigma_n(k,\ell)$ and 
$\rho^*(\Sigma_n(k,\ell))$ are (sometimes) mutation-inequivalent. 

\begin{example}\label{eg:rhodoesnotact} We show that the seeds $\Sigma_8(4,2)$ and $\rho^*(\Sigma_8(4,2))$ are mutation-inequivalent. The mutable variables are $I_1,I_2,I_3$ given by $\Delta_{1357},\Delta_{2357},\Delta+{3457}$  and the frozen variables $I_4,I_5 = \Delta_{1234},\Delta_{2345}$. 
Since $p=k=4$, $\Sigma_8(4,2)$ is a FZ-cluster algebra whose quiver has valued quiver arrows $I_1 \overset{(1,4)}{\to} I_3 \rightrightarrows I_2 \overset{(4,1)}{\to} I_1$, $I_4 \rightrightarrows I_3$ and $I_2 \to I_4$.

Let $a,b$ be indeterminates, and specialize the variables $I_1,\dots,I_5$ as follows: $I_1 \mapsto a^4$, $I_2,I_3 \mapsto a^3b$, and 
$I_4,I_5 \mapsto a^2b^2$. Thus, the exponent of $a$ (resp. $b$) records the grading of $I_i$ in the odd (resp. even) columns. This is a $\bbz^2$-grading on the algebra $\bbc[\mcd_8(4,2)]$, and moreover on the cluster algebra $\mca_8(4,2)$.

We claim by induction that cluster variables at vertex $1$ have degree of the form $a^{s+4}b^s$ for some $s \in \bbz$ while those at vertices $2$ and $3$ have degree $a^{s+2}b^s$ for some $s \in \bbz$. Since the frozen variables have degree $a^2b^2$, we can ignore multiplication by frozen variables in the remainder of the argument. 

The inductive claim holds for the initial seed. Suppose the claim holds at some intermediate seed of the form $\{a^{s+4}b^s, a^{t+2}b^t,a^{u+2}b^u,a^2b^2,a^2b^2\}$, and suppose we wish to perform a mutation at vertex $1$. The right hand side of the exchange relation is a sum of two monomials $M^+$ and $M^-$ and we can compute the degree of the neighboring cluster variable by subtracting the degree of the current cluster variable from the degree of $M^+$. The degree of $M^+$ is the form $m \cdot (a^{t+2}a^t)^4$ where $m$ is a monomial in the frozen variables, thus the degree of $M^+$ is 
of the form $a^{4t'+8}b^{4t'}$ for some $t' \in \bbz$. The degree of the neighboring cluster variable is therefore $a^{4t'+8-s-4}b^{4t'-s} = a^{(4t'-s)+4}b^{(4t'-s)}$, verifying the claim for mutation at vertex 1. The proof for vertices 2 and 3 is analogous. 

Finally, if $\rho^* \in \Aut(\bbc[\mcd])$ could be realized as a sequence of mutations, it would be possible to mutate from a seed containing $a^4$ to one containing $b^4$ in the specialization. And we have proved something even stronger: since we cannot arrive at a variable of the form $a^{4+s}b^s$, we see that $\rho^* \in \Aut(\bbc[\mcd])$ is not even a quasi-cluster transformation (i.e., a mutation sequence ``up to frozen variables''.) 
\end{example}

Since applying $\rho$ to any $\ell$-optimal collection clearly yields an $\ell$-optimal collection, we conclude that {\sl the set of $\ell$-optimal collections need not be connected by $\ell$-symmetrical square moves.} This is in contrast with the $\ell = n$ situation. In the latter case, an $\ell$-optimal collection is a maximal weakly separated collection, and these are known to tbe square-move connected. 

On the other hand, we can use Theorem~\ref{thm:titsmatsumoto} to give a ``large class'' of $\ell$-optimal collections which can be obtained from $\ov{\mcc}_n(k,\ell)$ which by $\rho^\ell$-symmetrical sequences of square moves (and thus, by mutations in $\mca_n(k,\ell)$).

We showed in the proof of Lemma~\ref{lem:seedandtest} that the initial $\ell$-optimal collection $\ov{\mcc}_n(k,\ell)$ is of the form $\mcc({\bf  f})$ for a maximal chain ${\bf f} = f_N \lessdot \cdots \lessdot f_0 \in \Bound_n(k,\ell)$. 

\begin{lem}\label{lem:seedandtestbis} Suppose that ${\bf f' } = f'_N \lessdot \cdots \lessdot f'_0 = f_0 \in \Bound_n(k,\ell)$ is another maximal chain ending at the same maximal element $f_0$. Then $\Delta(\mcc({\bf f'}))$ is a cluster in $\mca_{n}(k,\ell)$.
\end{lem}

That is, we can obtain $\Delta(\mcc({\bf f'}))$ from $\Delta(\mcc({\bf f}))$ by a sequence of mutations. 

\begin{proof}
By Theorem~\ref{thm:titsmatsumoto}, we can obtain ${\bf f'}$ from ${\bf f}$ by a sequence of 2- and 3- moves. By \cite[Theorem 5.3]{WilliamsBridge}, performing a 2-move in  $\Bound_n(k,n)$ does not affect the set of face labels of a bridge graph, and performing a 3-move amounts to a square move on bridge graphs. Any 3-move in  $\Bound_n(k,\ell)$ corresponds to $p$ commuting, $\rho^\ell$-equivariant, 3-moves in $\Bound_n(k,n)$, thus to $p$ commuting, $\rho^\ell$-equivariant sequences of square moves on the $\ell$-optimal cluster. It is easy to see that these $p$ square moves fold to a mutation (at a non-special variable) in $\mca_n(k,\ell)$. 
\end{proof}

It is natural to ask what happens when we remove the hypothesis $f_0 \in {\bf f'}$ from Lemma~\ref{lem:seedandtestbis}. 
One can verify that in this case, $\Delta(\mcc({\bf f'}))$ will be related to one of the clusters addressed in Lemma~\ref{lem:seedandtest} by an appropriate cyclic shift.

\begin{rmk}
Let $\mcd^\circ \subset \mcd$ be the localization at the frozen Pl\"ucker coordinates and define similarly $\Gr^\circ(k,n)$. Let $\mcb$ denote the extended affine braid group on $d$ strands, where $d = \gcd(k,n)$. It is the semidirect product of the affine braid group with an infinite cylic group generated by an element $\rho_\mcb$. There is a homomorphism $\Psi \colon \mcb \to {\rm Aut}(\Gr^\circ(k,n))$ \cite{FraserBraid} satisfying $\Psi{\rho_\mcb} = \rho$. (We denote $\Psi$ using subscripts.) This map restricts to a homomorphism 
$Z_\mcb(\rho_\mcb^\ell) \to {\rm Aut}(\mcd^\circ)$ whose domain is the centralizer of $\rho_\mcb^\ell$. Indeed, if $x \in \Gr^\circ(k,n)^{\rho^\ell}$ and $\sigma \in Z_\mcb(\rho^\ell) \subset \mcb$, then 
$$\rho^\ell(\Psi_\sigma(x)) = \rho^\ell(\Psi_\sigma(\rho^{-\ell}(x))) = \Psi_{\rho_\mcb^\ell \sigma\rho_\mcb^{-\ell}}(x) = \Psi_\sigma(x),$$
so that $\Psi_\sigma \in \Aut(\Gr(k,n)^{\rho^\ell}.$ Since $\Psi_\sigma$ preserves the TP part $\Gr(k,n)_{>0}$, we have $\Psi(\sigma) \in \Aut(\mcd)$. It would be interesting to investigate the conditions on $k,n,\ell$ which guarantee that $\Psi(Z_\mcb(\rho^\ell))$ consists of quasi-cluster transformations. We have checked that this is true for $\mcd_9(3,3)$. It {\sl fails} for $\mcd_8(4,2)$ by Example~\ref{eg:rhodoesnotact}. 
\end{rmk}

\subsection{Proper containment $\mca \subset \mca^{\rm up}$}
\begin{example}[$\mca \subsetneq \mcu$]\label{eg:AnotU} We establish in this example that we can have the containments $\mca_8(4,2) \subsetneq \bbc[\mcd_8(4,2)] \subseteq \mca^{\rm up}_8(4,2)$. We have $\dim \bbc[\mcd_8(4,2)]_{(1)} = \prod_{i=1}^k \dim \bbc[\bbp^{\ell-1}]_{(1)} = \ell^k = 16$. We have proved the containments \eqref{eq:hardcontain} and Conjecture~\ref{conj:uniformity} because this is a $p=k$ case in which the unfolding argument works. In particular, every cluster variable in $\mca_8(4,2)$ is the image of an $\ell$-cluster variable in $\bbc[|Gr(4,8)]$. By direct enumeration, 
exactly 42 elements of $\binom {[8]}4$ are $2$-cluster variables. They come grouped into 12 $\rho^2$-orbits, with orbit representatives listed here:
\begin{equation}\label{eq:48orbits}
1357,2468,1234,2345,1235,2346,1345,2456,1347,2458,1246,2357.
\end{equation}
If $x \in \bbc[\mcd_8(4,2)]_{(1)}$ is a cluster variable, then it must be amongst those listed in \eqref{eq:48orbits}. So $\dim \mca_{8}(4,2)_{(1)} \leq 12 < 16 = \dim \bbc[\mcd_8(4,2)]_{(1)}$, establishing the strict containment $\mca_8(4,2) \subsetneq \bbc[\mcd_8(4,2)]$. 

We can make a more specific statement.  By the $\bbz^2$-grading argument in Example~\ref{eg:rhodoesnotact}, none of the Pl\"ucker coordinates $2468$, $2346$, $2456$, $2458$, or $1246$ is a cluster variable, and by performing some mutations one sees that the remaining 7 Pl\"ucker coordinates {\sl are} extended cluster variables. So in fact $\dim \mca_8(4,2)_{(1)} = 7$. Note also that each of the above non-cluster variables would be in the cyclically shifted cluster structure $\rho^*(\mca_8(4,2))$. So in this example, the algebra generated by all extended cluster variables and their cyclic shifts coincides with $\bbc[\mcd_8(4,2)]$. 
\end{example}

\section{Miscellaneous proofs}\label{secn:proofs}
\subsection{Maximal elements and move-connectedness for $\leq_b$}\label{subsecn:bridgeorder}
\begin{proof}[Proof of Proposition~\ref{prop:maximalelts}]
First we explain that the elements $t_S,t_{S,s}$ are maximal. It follows quickly from $n$-boundedness that if $i \in {\rm Dec}(f)$ and $f \leq g$, then $f(i) = g(i)$ (this is true even in Bruhat order). We have ${\rm Dec}(t_S) = \bbz$ so that $t_S$ is clearly maximal. Maximality of $t_{S,s}$ is similar: ${\rm Dec}(t_{S,s}) = \bbz  \setminus (s+\ell\bbz)$, and an element of $\tilde{S}^k_\ell$ is determined by its values on such a set.   

Next let $f \in \Bound_n(k,\ell)$ be a maximal element in bridge order, we will argue that $f $ is either an $t_S$ or an $t_{S,s}$ accordingly. We can list the elements of $[\ell] \setminus {\rm Dec}(f)$ as $i_1 < i_2 < \cdots <  i_h$. First we claim that $h \in \{0,1\}$. Otherwise, by maximality of $f$, one must have that $f(i_1) > f(i_2) > \cdots > f(i_h)$, because otherwise we could perform a transposition which stays $n$-bounded and raises us in bridge order. But continuing in this way, using $\ell$-periodicity we see similarly that $f(i_h)>f(i_1)+\ell$ (otherwise we could perform a transposition of these two values). But this is clearly not possible. So indeed $h \in \{0,1\}$. If $h = 1$, then using $n$-boundedness and the assumption that $i_1 \notin {\rm Dec}(f)$, $f(i_1)$ must equal $i_1+s\ell$ for some $s \in [1,p-1]$. Let $a$ denote the number of $i \in [\ell]$ such that $f(i) = i+n$, so that  
$\sum_{i}f(i)-i = an+hs\ell = (ap+hs)\ell$. But since $f \in \tilde{S}^k_\ell$, we also have $\sum_{i=1}^\ell f(i)-i = k\ell$. This implies the relationship $k  = ap+sh$. If $p|k$ we conclude that $h=0$ and $a= \aa$. If not, we conclude that $h=1$, $s = \bb$, and $a = \aa$. The explicit description of the possible $f$ satisfying this numerology yields Proposition~\ref{prop:maximalelts}. 

Now we show that each maximal element has the same rank, given by the claimed formula for the height of the poset. 
The length of an element of $\tilde{S}^k_\ell$ can be computed as an inversion number.
In the case of $t_S$, each $i \in S$ participates in exactly $p(\ell-\aa)$ inversions, for a total of $\aa p (\ell-\aa) = \frac{k(n-k)}{p}$ many inversions, in agreement with the claimed formula for the height of the poset. In the case of $t_{S,s}$, each $i \in S$ participates in 
$(\ell-\aa-1)p$ inversions with elements of ${\rm Dec}(t_{S,s})$ and in $p-\bb$ inversions with elements of the form $s+a\ell$ for $a \in \bbz$. The element $s \in S$ participates in $(\ell-\aa-1)\bb$ inversions, for a total of 
$\aa p(\ell-\aa-1)+\aa(p-\bb)+\bb(\ell-\aa-1) = k(\ell-\aa)-\aa\bb-\bb$ inversions. This equals $\frac{1}{p}(k(n-k)+\bb(\bb-p))$ in agreement with the claimed formula. 
\end{proof}


The following technical lemma is used in the proof of Theorem~\ref{thm:titsmatsumoto}.

\begin{lem}\label{lem:comparedirectly} Let $f,t \in \Bound_n(k,\ell)$ with $t$ maximal. Then $f \leq t$ if and only if 
\begin{enumerate}
\item if $t(i) = i$ then whenever 
$j \in (f^{-1}(i),i) \setminus {\rm Dec}(f)$, we have either $f(j)>i$ or $t(f(j)) = f(j)$. 
\item if $t(i) = i+n$ then whenever $j \in (i,f^{-1}(i+n)) \setminus {\rm Dec}(f)$, we have either $f(j)<i+n$ or $t(j) = j+n$. 
\end{enumerate}
\end{lem}

\begin{proof}
By the explicit description of the maximal elements $t \in \Bound_n(k,\ell)$, saying that $f \leq t$ is the same as saying that there exists $f' \in \Bound_n(k,\ell)$ such that $f \leq f'$ such that $f'(i) = t(i)$ for all 
$i \in {\rm Dec}(t)$. Indeed, from the explicit description, $f'$ therefore equals $t$. 

In other words, to say that $f \leq t$ is to say that it is possible to ``sort'' the values $t(i)$ for $i \in {\rm Dec}(t)$ into position $i$ (where each sort-move is an a cover relation in bridge order). 

Now we argue the necessity of condition (1). If $t(i) = i$, then $n$-boundedness of $f$ implies that $f^{-1}(i) \leq i$, so we need to move the value $i$ right into position $i$. If $j \in (f^{-1}(i),i) \setminus {\rm Dec}(f)$ and $f(j) < i$, then sorting $i$ right past $f(j)$ is a downward move in bridge order. We cannot perform such a move unless the value $f(j)$ has already been sorted into a decorated position. In particular then, we must have $t^{-1}(f(j)) \in (f^{-1}(i),i) \cap {\rm Dec}(t)$. We cannot have $t^{-1}(f(j)) = f(j)-n$, because
$i \leq f^{-1}(i)+n < t^{-1}(f(j))+n = f(j)$. Otherwise $t(f(j)) = f(j)$. So condition (1) holds. 

The argument for necessity of condition (2) is dual. 

And the stated conditions are sufficient. By the argument just given, we can greedily sort each $t(i)$ for $i \in {\rm Dec}(t)$ into place assuming the conditions (1) and (2) hold. For example if $t(i) = i$ and $j \in (f^{-1}(i),i)$ then either $f(j)>i$ (so we can freely sort $i$ past $f(j)$) or $t(f(j)) = f(j)$. In the latter case, we would have $f^{-1} < j \leq f(j) < i$, so that we can sort the value $f(j)$ right into position $f(j)$ before we sort $i$ right into position $i$. 
\end{proof}


Note that if $t_{i,j+s\ell}$ is a cover in bridge order, then $|j+s\ell-i| < \ell$. So we can always index such covers by $t_{a,b}$ with $a < b < a+\ell$ and $a \in [\ell]$, which is the notation we prefer below. 

\begin{lem}\label{lem:bothcover}
Suppose we have elements $z \lessdot x,y \leq t \in \Bound_{pl}(k,\ell)$. Then there exists an element $x \vee y \in \Bound_{p\ell}(k,\ell)$, satisfying $x \vee y \leq t$, and chains 
$z \lessdot x \lessdot \cdots \lessdot x \vee y$ and 
$z \lessdot y \lessdot \cdots \lessdot x \vee y$ related by either a 2- or 3-move.  
\end{lem}

We believe that $x\vee y$ is in fact a least upper bound of $x,y$ in bridge order, but we do not carefully prove this.

\begin{proof} Suppose $x,y$ cover $z$ and moreover with $x = zt_{ab}$, $y = zt_{cd}$. There is some flexibility in this notation, since e.g. $t_{ab} = t_{a+\ell,b+\ell}$. 

We claim that up to this flexibility, the covering either takes the form 
\begin{itemize}
\item $x = zt_{ab}$, $y = zt_{cd}$ where $a < b < c < d < a+\ell$,
\item $x = zt_{ab}$,$y = zt_{bd}$ where $b \in \{z(a),z(d)-n\}$ and $d < a+\ell$,
\item $x = zt_{ab}$, $y = zt_{bd}$ where $d < a+\ell$. 
\end{itemize}

First we argue that these are indeed the only cases. We may assume $a \leq c$. We can immediately rule out that $c < b$ since the covering $z < zt_{ab}$ would imply that $c \in {\rm Dec}(z)$, hence that $y \notin \Bound_n(k,\ell)$. Similarly we can rule out $d>a+\ell$. We have cases $b<c$ or $b=c$. 

If $b < c$ and $d = a+\ell$ then using the flexibility above we can rename 
$t_{ab},t_{c,a+\ell}$ as $t_{c,a+\ell},t_{a+\ell,b+\ell}$, 
which is an instance of the ``$b=c$'' case. 
Otherwise, we have $a < b < c < d$ as in the first case. 

If $b=c$ and $d = a+\ell$, then every element of $(a,b)\cup (b,a+\ell) \in {\rm Dec}(z)$. A simple calculation shows that the order filter above $x$ is a chain, the order filter above $y$ is a chain, and these two chains only meet at the maximal element  $\hat{1} \in \widetilde{\Bound}_{n}(k,\ell)$. So $x,y$ are not covered by a common element $t$ as in the statement of the lemma. 

So we are left assuming $a < b = c < d <a+\ell$, and we have broken this up into the further two cases above because they will correspond to different coverings. 

Now we construct the element $x\vee y$ in each of the three cases.  In the first case, we have $x\vee y := zt_{ab}t_{cd} = zt_{cd}t_{ab}$. In the second case, if $b = z(a)$, we have $x \vee y := t_{ab}t_{ad} = zt_{bd}t_{ab}$, and if 
$b = z(d)-n$, we have $x \vee y := zt_{bd}t_{ad} = zt_{ab}t_{bd}$. These covers involve leaping over decorated positions. Finally, in the third case we have the usual braid relation $x \vee y := zt_{ab}t_{bd}t_{ab} = zt_{bd}t_{ab}t_{bd}$. In the first two cases the chains are related by a 2-move, and in the third case they are related by a 3-move.
\end{proof}

\begin{proof}[Proof of Theorem~\ref{thm:titsmatsumoto}]
We already argued that all maximal chains have the claimed length while arguing Proposition~\ref{prop:maximalelts}. It remains to prove the move-connectedness statement. We do this by induction on $\ell(t)-\ell(s)$. Let $f_h \lessdot f_{h-1} \lessdot \cdots \lessdot f_0$ be the chain ${\bf f}$ and use primes 
$f'_h \lessdot \cdots \lessdot f'_0$ to denote elements of the chain ${\bf f'}$. Let $f_i \neq f'_i$ be the earliest step in which these chains disagree. Then $i>h$ by hypothesis. 

We have $f_{i-1} \leq f_i,f'_i \leq f_0$. Thus by Lemma~\ref{lem:bothcover} there exists $f_i \vee f_i'$ with $f_i,f'_i \leq f_i \vee f'_i \leq f_0$. And moreover, we have chains $f_{i-1} \lessdot f_i \lessdot \cdots \lessdot f_i \vee f'_i$ and $f_{i-1} \lessdot f'_i \lessdot \cdots \lessdot f_i \vee f'_i$ related by a 2- or 3-move. Composing these with an arbitrary chain in ${\rm Chains}(f_{i} \vee f'_i,t)$ we get chains ${\bf g} = f_{h} \lessdot \cdots \lessdot f_i \lessdot \cdots \lessdot f_i \vee f'_i \cdots \lessdot f_0$ and ${\bf g'} = f'_{h} \lessdot \cdots \lessdot f'_i \lessdot \cdots \lessdot f_i \vee f'_i \cdots \lessdot f_0$ differing by a 2- or 3-move. By induction, the chains  ${\bf g} $ and ${\bf f}$ are related by a finite sequence of 2- or 3-moves, and likewise for 
${\bf g'}$ and ${\bf f'}$. Thus, all of the chains ${\bf f}$, ${\bf f'}$, ${\bf g'}$, ${\bf g}$ are connected by sequences of 2 or 3-moves, as required.  
\end{proof}

\subsection{{Proof of Lemma~\ref{lem:Toep}}}\label{secn:proofToep}
We explicitly compute the determinant of the Toeplitz matrix from Lemma~\ref{lem:Toep} as needed to prove Theorem~\ref{thm:onestep}. 

\begin{proof} Abbreviate $J:= \Delta_{I_{\ell+1}}$ and $K:= \Delta_{I_{2}}$. The first column of $M$ looks like $K+\eta_{[1]}J,J,0,\dots$. Expanding along this column, 
\begin{align*}
\det(M_{i,j \in [t]})&= (K+\eta_1J)\det(M_{i,j \in [1,t-1]})-J\det(M_{i \in 1 \cup [3,t], j \in [2,t]}) \\
&= \sum_{s = 0}^{t-1} K^sJ^{t-s}\eta_s+J(\eta_1\det(M)_{i,j \in [2,t]}-\det(M)_{i \in 1 \cup [3,t], j \in [2,t]})\\
&= \sum_{s = 0}^{t-1} \eta_sK^sJ^{t-s}+J\det(M')_{i \in 1 \cup [3,t], j \in [2,t]},
\end{align*}
so \eqref{eq:principal} follows by induction on $t$ if we establish \eqref{eq:offprincipal}.

To do this, we perform row operations to $M'$ (which do not change the value of the determinant) and apply \eqref{eq:straightenrs} repeatedly. To establish \eqref{eq:offprincipal} for a given $t$, we replace the first row $\eta_1R_2-\eta_1$ of $M'$ by the linear combination 
$\eta_1R_2-R_1-(\eta_2R_3+\eta_3R_4+\cdots+\eta_{t-1}R_{t})$. We claim that after this row operation is performed, the first $(t-1)$ entries in the first row equal zero, and the $t$th entry equals $(-1)^{t-1}\eta_{t}J$. Once we prove this, \eqref{eq:offprincipal} follows by expanding along the first column (picking up a sign of $(-1)^{t-1}$), noting that the matrix $M'_{i \in [3,t], \, j \in [2,t-1]} = M_{i \in [3,t], j \in [2,t-1]}$ is upper triangular with $J$ on each diagonal. 

The case $t=2$ corresponds the $1\times 1$ matrix with entry $\eta_1(K+\eta_1J)-(\eta_1K+\eta_{[2]}J)$, which equals $\eta_2J$ using \eqref{eq:straightenrs}. More generally, the entry in the first row and in column $i+2$ equals
$$\eta_1(\eta_{[i]}K+\eta_{[i+1]J})-(\eta_{[i+1]}K+\eta_{[i+2]J}) = \eta_{[i-1] \cup i+1}K+\eta_{[i] \cup i+2}J,$$
again using \eqref{eq:straightenrs}.

Subtracting off $\eta_2R_3$, we zero out the first nonzero entry, and the entries to the right take the form
\begin{align*}
(\eta_{[i-1] \cup i+1}K+\eta_{[i] \cup i+2}J)-\eta_2(\eta_{[i-1]}K+\eta_{[i]}J) &= (\eta_{[i-1] \cup i+1}-\eta_2\eta_{[i-1]})K+(\eta_{[i] \cup i+2}-\eta_2\eta_{[i]})J \\
&= -(\eta_{[i-2] \cup i+1})K-(\eta_{[i-1] \cup i+2})J.
\end{align*}
The first nonzero entry is in column 3 with value $i=1$, and is given by $-(\eta_{[i-2] \cup i+1})K-(\eta_{[i-1] \cup i+2})J = -(\eta_{3})J$ as claimed. Continuing in this fashion, each time we subtract off $\eta_{t-1}R_t$, we zero out one more entry in the first column, and we rewrite each of the nonzero entries using \eqref{eq:straightenrs}, with the first nonzero entry given by $(-1)^{t-1}\eta_tJ$.  
\end{proof}

\end{document}